\documentclass[]{article}
\usepackage[english]{babel}
\usepackage[normalem]{ulem}
\usepackage[all]{xy}
\usepackage{amsmath,amsthm,booktabs,color,dsfont,epsfig,graphicx,mathrsfs,multirow,scalefnt}
\usepackage{amsfonts, amssymb}
\usepackage{float}

\usepackage{mathtools}

\newtheorem{theo}{Theorem}[section]
\newtheorem{cor}[theo]{Corollary}
\newtheorem{lem}[theo]{Lemma}
\newtheorem{claim}[theo]{Claim}
\newtheorem{prop}[theo]{Proposition}
\newtheorem{defn}[theo]{Definition}
\newtheorem{rmk}[theo]{Remark}
\newtheorem{ex}[theo]{Example}
\newtheorem{conj}{Conjecture}
\newtheorem{prob}{Problem}

\newcommand{\Z}{\mathbb{Z}}
\newcommand{\N}{\mathbb{N}}
\newcommand{\NZ}{\mathbb{M}}

\newcommand{\s}{\sigma}

\newcommand{\A}{\mathbf{A}}

\newcommand{\G}{\mathcal{G}}

\newcommand{\NN}{\mathcal{N}^f}

\newcommand{\x}{\mathbf{x}}
\newcommand{\y}{\mathbf{y}}
\newcommand{\z}{\mathbf{z}}

\newcommand{\bfN}{\mathbf{N}}
\newcommand{\bfM}{\mathbf{M}}

\title{Some notes on the classification of shift spaces: Shifts of Finite Type; Sofic Shifts; and Finitely Defined Shifts.}

\author{
\small{Marcelo Sobottka}\\
\footnotesize{UFSC -- Department of Mathematics}\\
\footnotesize{88040-900 Florian\'{o}polis - SC, Brazil}\\
\footnotesize{\texttt{marcelo.sobottka@ufsc.br}}

}

\date{}

\begin{document}

\maketitle

\begin{abstract} The aim of this article is to find appropriate definitions for shifts of finite type and sofic shifts in a general context of symbolic dynamics. We start showing that the classical definitions of shifts of finite type and sofic shifts, as they are given in the context of  finite-alphabet shift spaces on the one-dimensional monoid $\N$ or $\Z$ with the usual sum, do not fit for shift spaces over infinite alphabet or on other monoids. Therefore, by examining the core features in the classical definitions of shifts of finite type and sofic shifts, we propose general definitions that can be used in any context. The alternative definition given for shifts of finite type inspires the definition of a new class of shift spaces which intersects with the class of sofic shifts and includes shifts of finite type. This new class is named finitely defined shifts, and the non-finite-type shifts in it are named shifts of variable length. For the specific case of infinite-alphabet shifts on the lattice $\N$ or $\Z$ with the usual sum, shifts of variable length can be interpreted as the topological version of variable length Markov chains.
\end{abstract}



\section{Introduction}\label{sec:Introduction}

For shift spaces over finite alphabets and on one-dimensional lattices ($\N$ or $\Z$) we have a plenty of useful classifications based on the properties of the shift's language (see \cite{BealBlockeletDima2013,BealBlockeletDima2014,HamachiKrieger2020,Kamabe2008,Krieger2017,
LindMarcus,Williams1973}). Also finite-alphabet shifts on the multi-dimensional lattice $\Z^d$ can be easily classified using basically the same criteria used in one-dimensional lattices. However, when considering shift spaces over infinite alphabets or on general monoids, the problem of classify them (in a way consistent with that used for finite-alphabet shifts on $\N$ or $\Z$) becomes more complex. In fact, as we will see here, even when considering the usual lattices $\N$ and $\Z$, while for finite-alphabet shift spaces we have topology, metric and language being strongly linked, in the context of infinite alphabets the relationship between them becomes more intricate. As example of this, we will see that topological conjugacies are not sufficient to keep features of the metric or of the language,  but it is necessary to use uniform topological conjugacies (with finite-to-one local rules, in some cases) to assure that features of the metric or of the language are preserved.

For instance, in the classical framework, shifts of finite type are defined as the class of shift spaces where just a finite number of patterns are not allowed to appear in the language. Although such a class of shift spaces provides very rich examples of dynamical behaviors whenever the alphabet is finite, for infinite-alphabet shifts to have only a finite number of patterns forbidden in the language is not an interesting restriction, since we have infinitely many symbols and most of the shift will behave as a full shift over an infinite alphabet. More problematic is the fact that, according to this definition, a classical shift of finite type over a finite alphabet is not a shift of finite type if it is thought to be contained in a full shift over infinite symbols. 

Following, if the lattice is $\N$ or $\Z$, we can define the class of edge shift spaces as the class of shift spaces that can be obtained from the (bi)infinite walks on some directed graph whose set of edges is the alphabet. Such a class can be defined as for finite-alphabet shifts as for infinite-alphabet ones  (in the later case, by considering graphs where the set of edges - and possibly also the set of vertices - is infinite). However, while in the finite-alphabet case, edges shifts conform a sub-class of shifts of finite type, in the infinite-alphabet case it is not in general a shift of finite type (if considering the classical definition of shifts of finite type). 

The differences between finite-alphabet shifts and infinite-alphabet shifts on the lattice $\N$ or $\Z$ grow worse when one tries to extend for infinite-alphabet shifts the definition of sofic shifts as it is given for finite-alphabet shifts. When the alphabet is finite one can define sofic shifts as the class of shift spaces that are obtained from any of the following two equivalent ways: From  the (bi)infinite walks on a finite labeled directed graph (Definition 3.1.3 in \cite{LindMarcus}); or as the image of a shift of finite type through a {\em sliding block code} (Theorem 3.2.1 in \cite{LindMarcus}). However,  as we shall prove in Theorem \ref{theo:motivation}, any  infinite-alphabet shift space on the lattice $\N$ can be generated from the infinite walks on some infinite labeled directed graph.

The problem of classifying shift spaces according to the features of their languages and/or topological behavior becomes still harder when one considers shift spaces on monoids other than $\N$ or $\Z$. For example, in such general context, even if the alphabet is finite, one cannot use graphs to define shift spaces with features analogous to those of sofic shift spaces on $\N$ or $\Z$.

A special attention must be paid to continuous shift-commuting maps between shift spaces. Such maps are the fundamental ones when studying dynamical systems on shift spaces or coding theory, and as we will see in the following sections, all the classifications used for finite-alphabet shift spaces on the lattices $\N$ or $\Z$ can be translated as the existence of specific types of sliding block codes.
We remark, that while the Curtis-Hedlund-Lyndon theorem ensures that continuous shift-commuting map between shift spaces over a finite alphabet are {\em sliding block codes} (see \cite{Helund1969}), when the alphabet is infinite they correspond to a larger class of maps called {\em generalized sliding block codes} (see \cite{SobottkaGoncalves2017}).

In this work, we shall propose alternative definitions for shifts of finite type and sofic shifts, that can be used in any general context, and that coincide with the classical ones in the context of finite-alphabet shifts on the lattices $\N$ and $\Z$. Furthermore, we shall present the class of {\em finitely defined shifts}, which is a natural extension of the class of shifts of finite type, and which contains a novel class of shift spaces that can only occur in the context of infinite alphabets. The shift spaces in this novel class are called {\em shifts of variable length} and can be interpreted as the topological versions of the variable length Markov chains \cite{BuhlmannWyner1999}.

The paper is organized as follows: In Section \ref{sec:general_shift_spaces} we present the basic definitions and results for shift spaces over any alphabet and on any monoid, while in Section \ref{sec:SBCs_and_GSBCs} we recall the definitions and basic properties of sliding block codes and generalized sliding block codes; In Section \ref{sec:HBCs_and_HBSs} we generalize the definition of higher block codes, which form a special class of generalized sliding block codes, 
and present some conditions under which shift spaces can be recoded as higher block shifts; In sections \ref{sec:SFTs} and \ref{sec:sofic} we give our general definitions for shifts of finite type and sofic shifts, and prove some basic results for them; In Section \ref{sec:graphs} we present some relationships between sofic shifts on the classical lattices $\N$ or $\Z$ and directed labeled graphs;  In Section \ref{sec:FDSs_and_SVLs} we introduce the new class of finitely defined shifts, which is composed by shifts of finite type and the so named shifts of variable length; In Section \ref{sec:relationships} we present the relationship between the classes that were defined in the previous sections; And in Section \ref{sec:final_discussion} we present some open problems and conjectures. At the end of the paper we present two appendixes where we propose alternative definitions for the concepts of memory for local rules of sliding block codes and for shift spaces, that are based on the metric of the monoid.

\section{Shift spaces over general alphabets and on general monoids}\label{sec:general_shift_spaces}

Let $\NZ$ be a {\bf monoid}, that is, a non-empty set with an  associative binary operation and an identity element. Let $1$ denote its identity, and for $g,h\in\NZ$ let $gh$ denotes the operation of $g$ with $h$. We will say that $\NZ$ is a {\bf metric monoid} if there exists a metric $d$ on $\NZ$ which is invariant under any left product, that is, for all $f,g,h\in\NZ$ we have $d(fg,fh)=d(g,h)$. A metric monoid $(\NZ,d)$ will be said to be {\bf conservative} if the cardinality of its closed balls is only function of their radii.

Given a nonempty set $A$ (an {\bf alphabet}) we will consider the {\bf full shift} $A^\NZ$ as the set of all configurations over $A$ indexed by $\NZ$. Given $N\subset\NZ$ and $\x=(x_i)_{i\in \NZ}\in A^\NZ$ we will denote the restriction of $\x$ to the indexes $N$ as $\x_N:=(x_i)_{i\in N}\in A^N$.

On $A$ we consider the discrete topology and on $A^\NZ$ we consider the respective prodiscrete topology. Given a finite set of indexes $D\subset \NZ$ and $(a_i)_{i\in D}\in A^{D}$ we define the cylinder
\begin{equation}\big[(a_i)_{i\in D}\big]_{A^\NZ}:=\{\x\in A^\NZ:\ x_i=a_i,\ \forall i\in D\}.\end{equation}
If $D=\{g\}$ we may simply denote $[b_g]_{A^\NZ}$ to the cylinder which fixes the symbol $b$ in the entry $g$. We remark that the collection of all cylinders forms a basis for the topology in $A^\NZ$. In the following proposition we recall some basic properties of full shift spaces which will be implicitly used along this article. 

\begin{prop}\label{prop:basic_properties} Let $A$ be an alphabet such that $\sharp A\geq 2$ and $A^\NZ$ be the respective full shift. It follows that:
\begin{enumerate}

\item\label{prop:basic_properties_Hausdorff} $A^\NZ$ is Hausdorff and totally disconnected;

\item\label{prop:basic_properties_compact} $A^\NZ$ is compact if and only if $A$ is finite;

\item\label{prop:basic_properties_firstcountable} $A^\NZ$ is first countable (and metrizable) if and only if $\NZ$ is countable;

\item\label{prop:basic_properties_secondcountable} $A^\NZ$ is second countable if and only if $\NZ$ and $A$ are countable.

\end{enumerate}
\end{prop}

\begin{proof} \phantom\\

\begin{enumerate}

\item Since cylinders are the basic open sets of $A^\NZ$, to check that $A^\NZ$ is totally disconnected we only need to prove that cylinders are closed sets. In fact, given $Z:=[(a_i)_{i\in D}]_{A^\NZ}$ and $\x\in A^\NZ\setminus Z$ there exists $k\in D$ such that $x_k\neq  a_k$, and thus $U:=[x_k]_{A^\NZ}\subset A^\NZ\setminus Z$ is an open neighborhood of $\x$.

 To check that $A^\NZ$ is Hausdorff, note that given $\x,\y\in A^\NZ$, $\x\neq\y$, there exists $g\in\NZ$ such that $x_g\neq y_g$. Hence $\x$ and $\y$ are separated by the cylinders $[x_g]_{A^\NZ}$ and  $[y_g]_{A^\NZ}$. 

\item If $A$ is finite, then it is compact for the discrete topology, and, from the Tychonoff Theorem, $A^\NZ$ will also be compact.
Conversely, if $A^\NZ$ is compact, then the open cover $\{[a_1]_{A^\NZ}:\ a\in A$\} admits a finite subcover, which implies that there are only finitely many symbols in $A$. 

\item If $\NZ$ is countable, then we can take an enumeration of it, say $\NZ=\{g_{\ell}\}_{\ell\geq 1}$ and for all $\x=(x_i)_{i\in\NZ}$ we have that the family of cylinders $\{[(x_{g_\ell})_{1\leq \ell\leq k}]_{A^\NZ}:\ k\geq 1\}$ is countable basis for the neighborhood of $\x$, that is, $A^\NZ$ is first countable (and the it is a sequential space). To check the converse, observe that if $\NZ$ is not countable, and since cylinders are the basic open sets of the topology and each cylinder is defined using only a finite number of coordinates, then a countable family of cylinders will cover all the coordinates of a sequence. 

We notice that  $\A^\NZ$ be first countable means that we can take any enumeration $\{g_\ell\}_{\ell\in\N}$ of $\NZ$ and to define the metric on $A^\NZ$ given for all $\x,\y\in A^\NZ$ by $d(\x,\y):=0$ if $\x=\y$, and $d(\x,\y):=2^{\min\{\ell:\ x_{g_\ell}\neq y_{g_\ell} \}}$ if $\x\neq\y$. It can be easily checked that such metrics generate the topology of cylinders. 

\item If $A$ and $\NZ$ are countable, then  the family of all cylinders of $A^\NZ$ will be countable. Conversely, if $A$ or $\NZ$ is not countable, then, since the topology contains an uncountable number of disjoint cylinders, it does not exist a countable basis for the topology. 
 
\end{enumerate}

\end{proof}

For each $g\in\NZ$ we define the $g$-shift map $\s^g:A^\NZ\to A^\NZ$ by $\s^g\big((x_i)_{i\in\NZ}\big)=(x_{gi})_{i\in\NZ}$ for all $(x_i)_{i\in\NZ}\in A^\NZ$. We remark that $\s^g$ is uniformly continuous\footnote{Even when $A^\NZ$ is not metrizable, it is uniformizable, that is, there exists a prodiscrete uniform
structure on $A^\NZ$ - see Section 1.9 and Appendix B in  \cite{Ceccherini-Silberstein--Coornaert}.} for all $g\in\NZ$  \cite[Proposition 1.2.2]{Ceccherini-Silberstein--Coornaert}.

Given a set $S\subset A^\NZ$ and $g\in \NZ$ we shall denote the translation of $S$ to $g$ as $\s^{g^{-1}}(S):=(\s^g)^{-1}(S)$, that is, the set such that $\s^g(\s^{g^{-1}}(S))=S$. We remark that here the notation $\s^{g^{-1}}(S)$ is used only for the sake of notation of the inverse image of $S$ by $\s^g$, and it is possible that $\s^{g^{-1}}(S)$ does not exist (that is, $\s^{g^{-1}}(S)=\emptyset$). However, if $\NZ$ is a group, then $g^{-1}$ will denote the inverse of $g$, and for any nonempty set $S$ we have $\s^{g^{-1}}(S)\neq\emptyset$ corresponding to the $g^{-1}$-shift of $S$. In particular, if $S$ is a cylinder (say $S=\big[(a_i)_{i\in D}\big]_{A^\NZ}$) and $\s^{g^{-1}}(S)$ exists, then 
\begin{equation}\label{eq:g-translation}\s^{g^{-1}}(S)=\big[(b_j)_{j\in gD}\big]_{A^\NZ},\quad\text{ where } b_j=a_i \text{ if } j=gi.\end{equation}

For each $k\in\N^*$, define $\mathcal{N}_A^k:=\displaystyle \bigcup_{\tiny\begin{array}{c}N\subset \NZ\\ \#N=k\end{array}}A^N$. Let $\NN_{A^\NZ}:=\displaystyle \bigcup_{\tiny k\in\N^*}\mathcal{N}_A^k= \bigcup_{\tiny\begin{array}{c}N\subset \NZ\\ \#N<\infty\end{array}}A^N$, and given a set of {\bf forbidden patterns} $ F\subset\NN_{A^\NZ}$, define the set

\begin{equation}\label{eq:X_F} X_F:=\{\x\in A^\NZ:\ \forall N\subset\NZ, \forall g\in\NZ,\ \left(\s^g(\x)\right)_{N}=(x_{gi})_{i\in N}\notin F\}.\end{equation}

Note that, since $F$ is a set of finite patterns, then in the definition of $X_F$ above it is sufficient to consider only $N\subset\NZ$ being a finite set  or even to restrict the choice of $N$ to those indexes used in some pattern of $F$. Given a set $F\subset\NN_{A^\NZ}$ we will denote as $M_F$ the set of all indexes of $\NZ$ that are used in some pattern of $F$, that is,
\begin{equation}\label{eq:M_F_defn}M_F:=\{g\in\NZ:\  \exists N\subset \NZ\text{ s.t. } g\in N\text{ and } A^N\cap F\neq\emptyset \}.\end{equation}

We remark that, though \eqref{eq:X_F} is the classical definition, one needs to take in mind that  
it implies that $X_F$ is strongly dependent as on the geometrical structure of $\NZ$ as on its algebraic structure.  In this sense, for the classical cases of monoids $\N^d$ or $\Z^d$, the definition given in \eqref{eq:X_F} can be intuitively thought as though to decide if $\x\in X_F$ one can look along $\x$ and check whether or not some pattern of $F$ appears there. However, what one really is doing is looking for the finite patterns that appear in each $\s^g(\x)$ (which in the topology of cylinders means to know the family of all the open sets into where each $\s^g(\x)$ lies). In the specific case of $\N^d$ and $\Z^d$ we have any $\s^g(\x)$ being easily identifiable inside $\x$ which implies that we can think this procedure as moving a fixed-width window along $\x$. But, in the general case this approach does not make sense. Example \ref{ex:shift_spaces_SFT} captures the previous discussion.\\

\begin{ex}\label{ex:shift_spaces_SFT} Consider the monoid $\N$ with the usual sum, and the monoid $\N^*$ with the usual product. Define on $\N$ the metric $d_\N(g,h):=|g-h|$, and define on $\N^*$ the metric $d_{\N^*}(g,h):=|\ln(g)-\ln(h)|$. Note that each of these metrics turns the respective monoid in a metric monoid. Furthermore, we can use each metric to define total orders on the monoids as follows: for $a,b\in\N$ we say $a\le b$ if and only if $d_\N(0,a)\le d_\N(0,b)$; and for $a,b\in\N^*$ we say $a\le b$ if and only if $d_{\N^*}(1,a)\le d_{\N^*}(1,b)$. Hence, we  have that $\N$ and $\N^*$ are geometrically represented as lattices where the nearest neighbors of $x_g$ are $x_{g-1}$ and  $x_{g+1}$.

\begin{figure}[H]
\centering
\includegraphics[width=.8\linewidth=1.0]{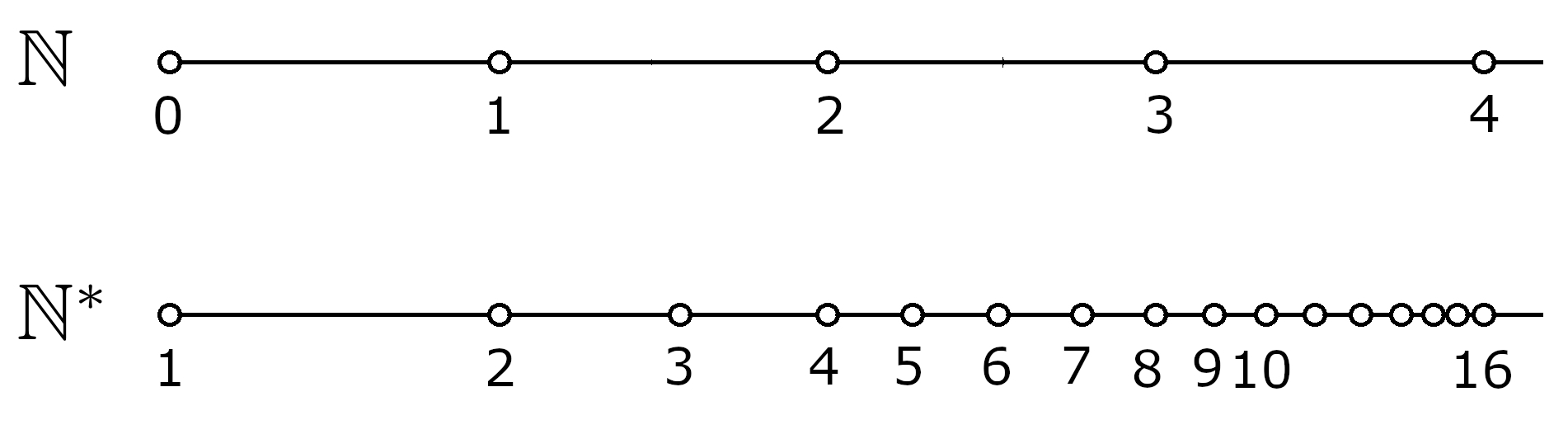}
\caption{At the top the geometric representation of the lattice $\N$ with the usual sum, where all the elements are uniformly spaced. At the bottom the geometric representation of the lattice $\N^*$ with the usual product, where the elements of the form $2^ng$ are uniformly spaced.}\label{Fig_Ex_Lattices}
\end{figure}

Now, let $A:=\{a,b\}$ and in $A^\N$ consider the set of forbidden words $F:=\{(x_0x_1): x_0\neq x_1\}$, and in $A^{\N^*}$ consider the set of forbidden words $F^*:=\{(x_1x_2): x_1\neq x_2\}$.
Note that in both cases was forbidden the words $(ab)$ and $(ba)$ on the positions indexed by the identity element and its nearest neighbor, that is, $M_F=\{0,1\}$ and $M_{F^*}=\{1,2\}$.

Let $\Lambda:=X_F$ and $\Lambda^*:=X_{F^*}$ and observe that while $\Lambda$ consists just of  two constant sequences, $(x_0x_1x_2x_3...)=(aaaa...)$ and  $(y_0y_1y_2y_3...)=(bbbb...)$, the shift $\Lambda^*$ contains infinitely many sequences. In fact, $\x\in A^{\N^*}$ belongs to $\Lambda^*$ if and only if for all $g\in A^{\N^*}$ we have $\big(\s^g(\x)\big)_{M_{F^*}}=(x_{g1},x_{g2})\notin F$. In other words,  $\x\in \Lambda^*$ if and only if for all $g\in \N^*$ it follows that $x_{g2^n}=x_g$ for all $n\geq 0$. Hence we get that are allowed in $\Lambda^*$ sequences where $x_g\neq x_h$ for any $g,h\in\N^*$ such that $g\neq  h2^n$ for all $n\in\Z$. 

We notice that this difference between $\Lambda$ and $\Lambda^*$ is due to the fact that while $(\N,d_\N)$ is conservative,  $(\N^*,d_{\N^*})$ is not.

\end{ex}

\begin{defn}\label{defn:shift_space} We say that a set $\Lambda\subset A^\NZ$ is a {\em shift space} if  $\Lambda=X_F$ for some  $ F\subset\NN_{A^\NZ}$ .
\end{defn}

\begin{rmk}\label{rmk:shift_space-closed_invariant}
Equivalently we have that $\Lambda\subset A^\NZ$ is a shift space if and only if it is closed with respect to the topology of $A^\NZ$ and invariant under any $g$-shift map (the proof is the same as that given in \cite[exercises 1.21 and 1.23]{Ceccherini-Silberstein--Coornaert}). 
\end{rmk}

We consider, on a given shift space $\Lambda$, the topology induced from $A^\NZ$ having as a basis the sets $$\big[(a_i)_{i\in D}\big]_\Lambda:=\big[(a_i)_{i\in D}\big]_{A^\NZ}\cap\Lambda.$$

For each nonempty finite $N\subset \NZ$ we define $W_N(\Lambda)$ as the set of all patterns of $A^N$ that appear in some element of $\Lambda$, that is, 

$$W_N(\Lambda):=\{(w_i)_{i\in N}\in A^N:\ \exists \ \x\in\Lambda \text{ s.t. } x_i=w_i\ \forall i\in N\}.$$

We notice that, with this notation, the set $W_{\{1\}}(\Lambda)$ contains all the symbols of $A$ that appear in some sequence of $\Lambda$.\\

The {\bf language} of $\Lambda$ will be defined as the set $W(\Lambda)$ of all possible finite patterns in elements of $\Lambda$:

$$W(\Lambda):=\displaystyle \bigcup_{\tiny\begin{array}{c}N\subset \NZ\\ \#N<\infty\end{array}} W_N(\Lambda).$$

Note that $W_N(A^\NZ)=A^N$ and $W(A^\NZ)=\NN_{A^\NZ}$.\\

We notice that we are using the name language in a more general way than it is used when the lattice is $\Z$ with usual sum. In that case, besides an empty word to be included in the language, the patterns are taken modulo translations. In fact, whenever $\NZ$ is a group, it is true that two sets of forbidden patterns that are equal modulo translation will generate the same shift space, and that if $N,M\subset\NZ$ are one a translation of the other, then  $W_N(\Lambda)=W_M(\Lambda)$ modulo translation.

 \begin{defn}\label{defn:complete_F}  We will say that a set of forbidden patterns $F\subset\NN_{A^\NZ}$ is {\bf complete} if and only if for all $w\notin W(X_F)$, there exists $u\in F$ which is a subpattern of $w$. 
 \end{defn}
 
Note that $F$ to be complete is equivalent to say that if $w\in \NN_{A^\NZ}$ does not contain any subpattern which belongs to $F$, then $w\in W(X_F)$. For example, if we consider $A=\{0,1\}$ and $\Z$ with the usual sum, and take $F:=\{(x_0x_1x_2)\in \NN_{A^\Z}:\ x_0=x_2=1,\ x_1\in A\}$, then it follows that $F$ is not complete since the pattern $(x_0x_2)$ with $x_0=x_2=1$ does not belong to the language of $X_F$ in spite of it does not contain any subpattern which belongs to $F$ (we recall that in our context we are allowed to take patterns which are not defined on consecutive indexes, which is not done in the classical context). 
On the other hand, given any shift space $\Lambda\subset A^\NZ$, we can always take $F:=\NN_{A^\NZ}\setminus W(\Lambda)$ which is complete and such that $X_F=\Lambda$. In particular, it is easy to check that given $F\in \NN_{A^\NZ}$ with $M_F$ finite, we can always find $\hat F\in \NN_{A^\NZ}$  complete with $M_{\hat F}$ finite, and such that $X_{\hat F}=X_F$.\\

The following result is a version of Proposition \ref{prop:basic_properties}, stated for general shift spaces. Its proof is left to the reader. 

\begin{prop}\label{prop:basic_properties_general} Let  $\Lambda\subset A^\NZ$ be a shift space with  $\sharp W_{\{1\}}(\Lambda)\geq 2$. It follows that:
\begin{enumerate}

\item\label{prop:general_basic_properties_Hausdorff} $\Lambda$ is Hausdorff and totally disconnected;

\item\label{prop:general_basic_properties_compact} $\Lambda$ is compact if and only if $W_{\{1\}}(\Lambda)$ is finite;

\item\label{prop:general_basic_properties_firstcountable} If $\NZ$ is countable then $\Lambda$ is first countable (and metrizable);

\item\label{prop:general_basic_properties_secondcountable} If $\NZ$ and $W_{\{1\}}(\Lambda)$  are countable, then $\Lambda$ is second countable.

\item\label{prop:general_basic_properties_secondcountable2} If $\Lambda$ is second countable, then $W_{\{1\}}(\Lambda)$ is countable.

\end{enumerate}
\end{prop}

\qed

Note that, for general shift spaces, we cannot assure that $\NZ$ to be uncountable implies in $\Lambda$ to be non-first countable or non-second countable (for example, if $\Lambda$  contains only a countable number of constant sequences, it will be first countable and second in spite of the cardinality of $\NZ$).

\section{Sliding block codes and generalized sliding block codes}\label{sec:SBCs_and_GSBCs}

Sliding block codes play a main roll in the study topological invariants of shift spaces. 

\begin{defn}\label{SBC}
Let $A$ and $B$ be two alphabets and let  $\Lambda\subset A^\NZ$ be a shift space. A map $\Phi:\Lambda \to B^\NZ$ is a {\bf sliding block code (SBC)} if there exist a finite set $M\subset\NZ$ and a map (local rule) $\phi:W_M(\Lambda)\to B$ such that for all $\x\in\Lambda$ and $g\in\NZ$ we have $(\Phi(\x))_g=\phi\big(\s^g(\x)_M\big)=\phi\big((x_{gi})_{i\in M}\big)$.
\end{defn}


We remark that in the classical case of $\NZ$ being the lattice $\N^d$ or $\Z^d$ with the usual sum we have that  if $M\subset\NZ$ consists of $k$ nearest neighbors of the identity, then for any $g\in\NZ$ the neighborhood $gM$ of $g$ will consist of $k$ nearest neighbors of $g$. However, when considering the general case of $\NZ$ being any monoid it could not hold. For example if $\NZ$ is the monoid $\N^*$ with the usual product, then  taking the $k$ nearest neighbors of the identity, $M:=\{1,2,3,...,k\}$, for any $g\neq 1$ we will have $gM=\{g1,g2,g3,...,gk\}$ which are not the $k$ nearest neighbors of $g$. This discussion is captured in Example \ref{Ex:SBC_on_N*} below. 

\begin{ex}\label{Ex:SBC_on_N*}
Let $A:=\{0,1\}$, and consider the two following monoids: $\N$ with the usual sum; and $\N^*$ with the usual product. Let $M:=\{0,1\}\subset\N$ and $M^*:=\{1,2\}\subset\N^*$. 

Consider the sliding block codes $\Phi:A^\N\to A^\N$ and $\Phi^*:A^{\N^*}\to A^{\N^*}$, whose local rules are $\phi:W_M(A^\N)\to A$ given by $\phi(x_0,x_1):=x_0+x_1 {\tiny\text{ (mod 2)}}$, and 
$\phi^*:W_{M^*}(A^{\N^*})\to A$ given by $\phi^*(x_1,x_2):=x_1+x_2 {\tiny\text{ (mod 2)}}$.

Note that, the local rules of $\phi$ and $\phi^*$ are both defined on the positions indexed by the identity element of the correspondent monoid and their nearest neighbor. However, for any $g\in\N$ and $\x\in A^\N$ we have $$\left(\Phi(\x)\right)_g=
\phi(\x_{g+\{0,1\}})=\phi(\x_{\{g,g+1\}})=x_g+x_{g+1} {\tiny\text{ (mod 2)}},$$
while 
for any $g\in\N^*$ and $\x\in A^{\N^*}$ we have $$\left(\Phi^*(\x)\right)_g=
\phi^*(\x_{g\{1,2\}})=\phi^*(\x_{\{g,2g\}})=x_g+x_{2g} {\tiny\text{ (mod 2)}}.$$

\end{ex}

\begin{figure}[H]
\centering
\includegraphics[width=.9\linewidth=1.0]{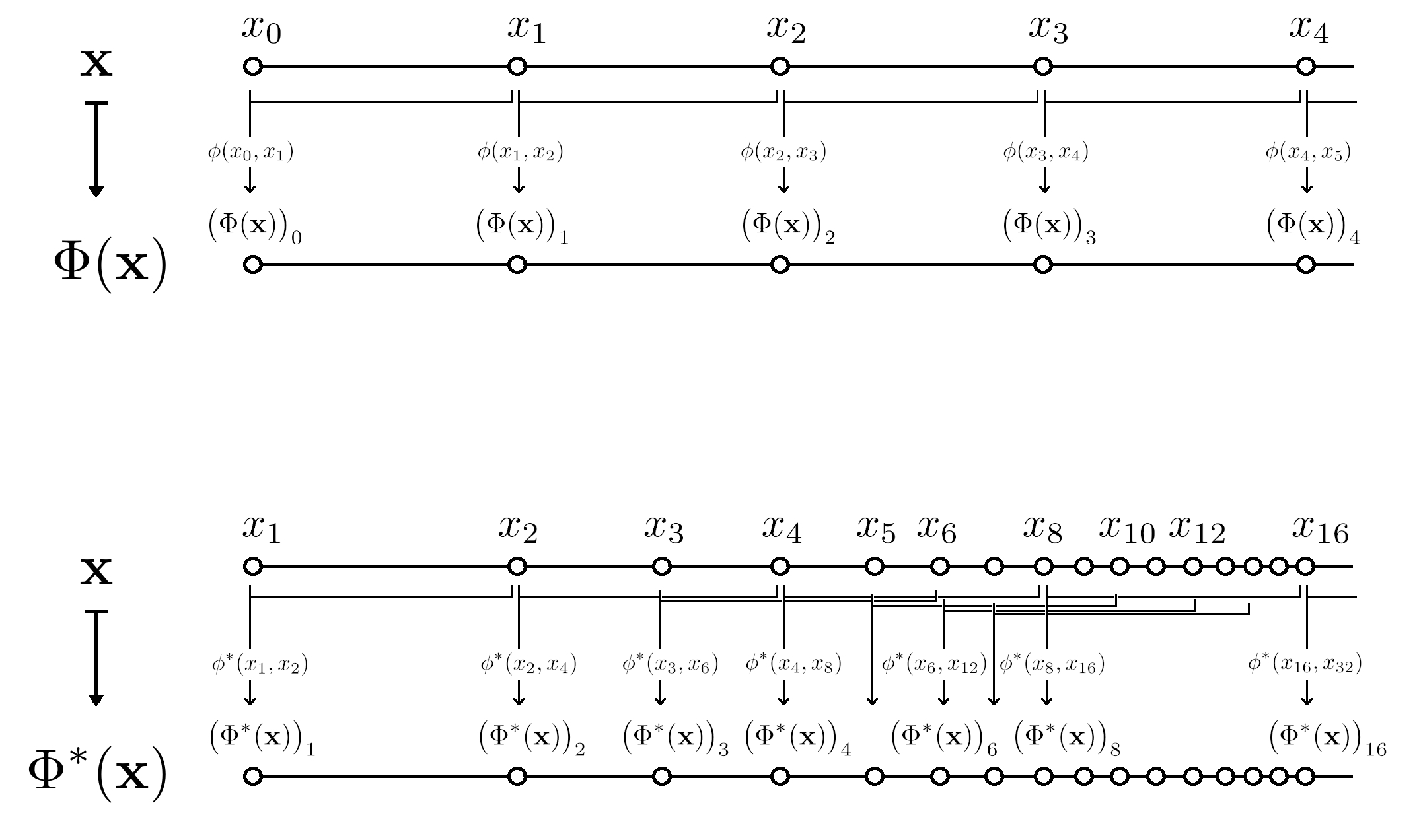}
\caption{At the top the sliding block code $\Phi$: Its local rule $\phi$ acts on the nearest neighbors $x_i$ and $x_{i+1}$. At the bottom the sliding block code $\Phi^*$: Its local rule $\phi^*$ acts on  $x_i$ and $x_{2i}$. Note that in terms of the metrics on $\N$ and $\N^*$,  both local rules, $\phi$ and $\phi^*$, act on constant-width windows.}\label{Fig_Ex_SBC}
\end{figure}

A generalization of sliding block codes was proposed in \cite{Romero_et_Al2006} for maps between shift spaces on the lattice $\Z^d$, and in \cite{SobottkaGoncalves2017}  it was formulated for maps between shifts on general monoids. Roughly speaking, $\Phi$ is said a generalized sliding block code\footnote{In \cite{Romero_et_Al2006} this class of maps was not named, although it was used the expression ``generalized local rules''. In \cite{SobottkaGoncalves2017} this class was designated as ``generalized sliding block codes'' while in \cite{Campos_et_Al2021} this class of maps is referred as ``extended sliding block codes'' (ESBC).} if for any $\x\in\Lambda$ and $g\in\NZ$ there exists a finite $N\subset\NZ$ (which is function of configuration $\x$ around $x_g$) such that  $(\Phi(\x))_g$ is function of the pattern $\x_{gN}$. Next we present the rigorous definition of generalized sliding block codes as given in \cite{SobottkaGoncalves2017}.

\begin{defn}\label{GSBC}
Let $A$ and $B$ be two alphabets and let  $\Lambda\subset A^\NZ$ be a shift space. A map $\Phi:\Lambda \to B^\NZ$ is a {\bf generalized sliding block code (GSBC)} if there exists $\{C_b\}_{b\in B}$ a partition of $\Lambda$ where each nonempty $C_b$ is a union of cylinders of $\Lambda$, such that
\begin{equation}\label{eq:LR_block_code}\bigl(\Phi(\x)\bigr)_g=\sum_{b\in B}b\mathbf{1}_{C_b}\circ\sigma^g(\x),\quad \forall \ \x\in\Lambda,\ \forall \ g\in\NZ, \end{equation} where  $\mathbf{1}_{C_b}$ is the
characteristic function of the set $C_b$ and $\sum$ stands for the symbolic sum.
\end{defn}

 
 \begin{rmk}\label{rmk:finitely_defined_set} The sets $C_b$ that are used in the definition of GSBC are called {\bf finitely defined sets}.  Roughly, a set $\mathcal{S}$ is a finitely defined set, if for all $\x\in A^\NZ$ there exists a finite $D\subset\NZ$ such that one only needs to know $\x_D$ to decide whether $\x$ belongs to $\mathcal{S}$ or not. Finitely defined sets were originally introduced in \cite{GSS,GSS1} and developed in \cite{ZS2020,GS2019} where they were used to study continuous shift-commuting maps between alternative types of symbolic dynamical systems where it were considered other topologies than the prodiscrete one. In the context of  the prodiscrete topology, a set $\mathcal{S}\subset A^\NZ$ is a finitely defined set if and only if it is a clopen set. 
\end{rmk}

\begin{rmk}\label{rmk:form_for_shift_Commuting_maps}
Note that for each $b\in B$ we have $C_b=\Phi^{-1}([b_1]_{B^\NZ})$ where $[b_1]_{B^\NZ}$ is the cylinder of $B^\NZ$ which fixes the symbol $b$ at the position $1$ (the identity of $\NZ$). Furthermore, any shift-commuting map $\Phi$ can be written in the form \eqref{eq:LR_block_code} (however, if $\Phi$ is not continuous, then some set $C_b$ will not be open). 
\end{rmk}

\begin{rmk}\label{rmk:finite_coordinates} A map $\Phi$ is an SBC if and only if it satisfies \eqref{eq:LR_block_code} and each nonempty $C_b$ can be written as the union of cylinders whose coordinates belong to a same finite set $M\subset\NZ$. In particular, whenever the alphabet $A$ is finite, we have that any GSBC is an SBC (due to the compactness of the shift space). 
\end{rmk}

By recalling that $A^\NZ$ can be endowed with the prodiscrete uniform
structure \cite[Section 1.9]{Ceccherini-Silberstein--Coornaert}, we can also define SBCs and GSBCs through the following theorems which are generalizations of the classical Curtis-Hedlund-Lyndon Theorem.

\begin{theo}[Theorem 1.9.1 in \cite{Ceccherini-Silberstein--Coornaert}] \label{theo:hedlund_classic} A map $\Phi:\Lambda\subset A^\NZ \to B^\NZ$ is a {\bf sliding block code} if, and only if, it is {\bf uniformly continuous} and {\bf commutes with all $g$-shift maps}.
\end{theo}

\qed

\begin{theo}[Theorem 5 in \cite{SobottkaGoncalves2017}\footnote{A first version of this result, concerning shift spaces on the lattice $\Z^d$, was given in \cite[Theorem 4]{Romero_et_Al2006}.}]\label{theo:hedlund_general} A map $\Phi:\Lambda\subset A^\NZ \to B^\NZ$ is a {\bf generalized sliding block code} if, and only if, it is {\bf continuous} and {\bf commutes with all $g$-shift maps}.
\end{theo}

\qed

\begin{rmk} The characterizations given in theorems \ref{theo:hedlund_classic} and \ref{theo:hedlund_general}, provide a direct way to check that the classes of SBCs and GSBCs are invariant under compositions of their respective elements.
\end{rmk}

\begin{rmk} Note that a $\bar g$-shift map $\s^{\bar g}$ is not necessarily a sliding block code, since it could not commute with the other $g$-shift maps. In fact, $\s^{\bar g}$ will be a sliding block code on $\Lambda$ if, and only if, for all $\x\in\Lambda$ and $i\in\NZ$ we have $x_{\bar{g}i}=x_{i\bar{g}}$ (this condition holds trivially if $\bar g$ belongs to the center of $\NZ$, that is, if $\bar{g}i=i\bar{g}$ for all $i\in\NZ$). In terms of the expression \eqref{eq:LR_block_code}, when $\s^{\bar g}$ is an SBC, then for each $b\in B(=A)$ we have $C_b=[b_{\bar g}]_{\Lambda}$.

\end{rmk}

Next, we define some subclasses of GSBCs and SBCs which will play fundamental roles in the definition of sofic shifts.\\

\begin{defn}\label{defn:locally-finite-to-one_SBC}
We say that a generalized sliding block code $\Phi:\Lambda\to B^\NZ$ is {\bf locally finite-to-one} if each $C_b=\Phi^{-1}([b_1]_{B^\NZ})$ can be written as the union of a finite number of cylinders. When a locally finite-to-one generalized sliding block code is such that there exists $k\in\N$ such that each $C_b$ is the union of at most $k$ cylinders, we will say that it is {\bf locally bounded finite-to-one} with {\bf order $k$}.
\end{defn}

In the particular case of $\Phi$ being a sliding block code between infinite-alphabet shift spaces, to be  locally finite-to-one does not implies that  it can be written with a local rule $\phi:W_M(\Lambda)\to B$ which is finite-to-one. The next example shows this fact:

\begin{ex}

Let $\Phi:\N^\N \to (\N\times\N)^\N$ be the map given for all $\x\in \N$ and $i\in\N$ by $(\Phi(\x))_i:=(x_i,x_ix_{i+1}+x_i)$. Observe that $\Phi$ is an invertible sliding block code which can be written in the form \eqref{eq:LR_block_code} with the following finitely defined sets: $C_{(0,0)}$ composed only by the cylinder that fixes $0$ in the coordinate $0$; $C_{(a,b)}=\emptyset$, if $a=0$ and $b\neq 0$ or if $a\neq 0$ and $(b-a)/a\notin\N$; and $C_{(a,b)}$ composed only by the cylinder that fixes $a$ in the coordinate $0$ and $(b-a)/a$ in the coordinate $1$, if $a\neq 0$ and $(b-a)/a\in\N$. Hence $\Phi$ is locally bounded finite-to-one, however the local rule $\phi:\N^2\to\N^2$ is not finite-to-one (since $\phi(0,b)=(0,0)$ for all $b\in\N$).

\end{ex}

The next results give some sufficient conditions for $\Phi(\Lambda)$ to be a shift space.

\begin{theo}\label{theo: image_of_shifts-intersection} Suppose $\NZ$ is countable and $\Lambda\subset A^\NZ$ is such that for all nested family of nonempty cylinders $\{W_\ell\}_{\ell\geq 1}$  it follows that $\bigcap_{i\geq 1}W_i\neq\emptyset$. If $\Phi:\Lambda\to B^\NZ$ is a locally finite-to-one generalized sliding block code, then $\Phi(\Lambda)$ is a shift space.

\end{theo}

\begin{proof} Since $\Phi$ is a GSBC, then $\s(\Phi(\Lambda))=\Phi(\s(\Lambda))\subset\Phi(\Lambda)$. Hence, to conclude that $\Phi(\Lambda)$ is a shift space, we only need to prove that it is closed in $B^\NZ$.\\

Let $(\y^i)_{i\geq 1}$ be a sequence in $\Phi(\Lambda)$ which converges to some $\z=(z_g)_{g\in \NZ}\in B^\NZ$. Let us show that $\z\in \Phi(\Lambda)$.

Consider $\{g_\ell\}_{\ell\geq 1}$  an enumeration of $\NZ$, and for each $k\geq 1$ define $$\mathcal{Z}_k:=\big[(z_{g_\ell})_{1\leq\ell\leq k}\big]_{B^\NZ},$$
which is the cylinder that fixes the symbol $z_{g_\ell}$ in the position $g_\ell$ for $1\leq \ell\leq k$.

Observe that $\{\mathcal{Z}_k\}_{k\geq 1}$ is family of nested nonempty cylinders such that $\bigcap_{k\geq 1}\mathcal{Z}_k=\{\z\}$ and, since $\y^i$ converges to $\z$, for each $k\geq 1$ there exists $m\geq 1$ such that $\y^m\in \mathcal{Z}_k$, which implies that for all $k\geq 1$ we have  
\begin{equation}\label{eq:U_ell} \begin{array}{lcl}\emptyset&\neq& \Phi^{-1}(\mathcal{Z}_k)\\\\
&=&\Phi^{-1}\left(\big[(z_{g_\ell})_{1\leq\ell\leq k}\big]_{B^\NZ}\right)\\\\
&=&\Phi^{-1}\left(\bigcap_{1\leq\ell\leq k}
\big[z_{g_\ell}\big]_{B^\NZ}\right)\\\\
&=&\bigcap_{1\leq\ell\leq k}
\Phi^{-1}\big([z_{g_\ell}]_{B^\NZ}\big)=\bigcap_{1\leq\ell\leq k}U_\ell,\end{array}\end{equation}
where $U_\ell:=\Phi^{-1}\big([z_{g_\ell}]_{B^\NZ}\big)$. Furthermore, by denoting $C_{b^{\ell}}:=\Phi^{-1}([{b^{\ell}_1}]_{B^\NZ})$ with $b^{\ell}:=z_{g_\ell}$, it follows that $U_\ell=\s^{g_\ell^{-1}}(C_{b^{\ell}})$ where $\s^{g_\ell^{-1}}(C_{b^{\ell}})$ denotes the translation of $C_{b^\ell}$ to $g_\ell$. Since $\Phi$ is locally finite-to-one then each $C_{b^\ell}$ can be written as a union of finitely many pairwise-disjoint cylinders of $\Lambda$ and, since each $U_\ell\neq\emptyset$, then from \eqref{eq:g-translation} we get that each $U_\ell$ is also the union of finitely many pairwise-disjoint cylinders, say \begin{equation}\label{eq:U_ell-finite}U_\ell=\bigcup_{1\leq i\leq N(\ell)}W^\ell_i.\end{equation} Since \eqref{eq:U_ell} holds for all $k\geq 1$ and \eqref{eq:U_ell-finite} holds for all $\ell\geq 1$, it follows that for each $\ell$ there exists $1\leq j(\ell)\leq  N(\ell)$ such that \begin{equation}\label{eq:U_ell-intersect}\emptyset\neq \bigcap_{1\leq\ell\leq k} W^\ell_{j(\ell)}\subset \bigcap_{1\leq\ell\leq k}U_\ell= \Phi^{-1}(\mathcal{Z}_k),\qquad \forall k\geq 1.\end{equation}
Since $\left\{\cap_{1\leq\ell\leq k} W^\ell_{j(\ell)}\right\}_{k\geq 1}$ is a nested family of nonempty cylinders, by hypothesis on $\Lambda$ it follows that $\bigcap_{\ell\geq 1} W^\ell_{j(\ell)}\neq\emptyset$\sloppy. Thus,
\begin{equation}\label{eq:nonempty_intersect}\emptyset\neq\bigcap_{\ell\geq 1} W^\ell_{j(\ell)}\subset  \bigcap_{\ell\geq 1}U_\ell=\bigcap_{k\geq 1}\bigcap_{1\leq\ell\leq k}U_\ell=\bigcap_{k\geq 1}\Phi^{-1}(\mathcal{Z}_k)=\Phi^{-1}(\z),\end{equation}
and then $\z\in\Phi(\Lambda)$.

\end{proof}

Note that  $\bigcap_{\ell\geq 1} W^\ell_{j(\ell)}$ in \eqref{eq:nonempty_intersect} is a closed set of $\Lambda$ which is fixing infinitely many (but not necessarily all) coordinates. In Theorem \ref{theo: image_of_shifts} we shall assume some conditions that imply $\bigcap_{\ell\geq 1} W^\ell_{j(\ell)}$ fixes all the coordinates of $\NZ$.

\begin{cor}\label{cor:image_Phi-fullshift} Suppose $\NZ$ is countable and  $\Phi:A^\NZ\to B^\NZ$ is a locally finite-to-one generalized sliding block code, then $\Phi(A^\NZ)$ is a shift space.
\end{cor}

\begin{proof} It is direct from the fact that the full shift $A^\NZ$ satisfies the property stated in Theorem \ref{theo: image_of_shifts-intersection}.

\end{proof}

We recall that if some $U_\ell$ in \eqref{eq:U_ell} could not be written as the union of finitely many cylinders (that is, if $\Phi$ were not locally finite-to-one), then we could not assure the existence of cylinders $W^\ell_{j(\ell)}$ satisfying \eqref{eq:U_ell-intersect}. Such fact is encapsulated by Example 1.1. in \cite{Campos_et_Al2021} where it is given a counterexample of Corollary \ref{cor:image_Phi-fullshift} when $\Phi$ fails in being locally finite-to-one. On the other hand, for a shift space $\Lambda$ where there exists a nested family of nonempty cylinders whose intersection of all of its members is empty, it is possible that in spite of some generalized sliding block code $\Phi$ to be locally finite-to-one, we have $\Phi(\Lambda)$ being not closed (due to the fact that the non-emptyness in \eqref{eq:nonempty_intersect} could fail - see example  below).

\begin{ex}\label{ex:Phi(Lambda)-nonshift}
Consider the lattice $\N$ with usual sum, and $\Lambda\subset (\N^*)^\N$ being the smallest shift space which, for $n\geq 1$, contains configurations $\y^n\in(\N^*)^\N$ in the form $y_0^n=n$, $y_i^n=i$ for all $i=1,...,n-1$ and $y_j^n=n$ for all $j\geq n$. 

Let $\Phi:\Lambda\to B^\NZ$ be the classical shift map, that is, $\Phi\big((y_i)_{i\in\N}\big):=\s\big((y_i)_{i\in\N}\big)=(y_{i+1})_{i\in\N}$. 
Observe that the sequence $\z=(z_i)_{i\in\N}$ with $z_i=i+1$ belongs to $\Lambda$, but it does not belong neither to $\Phi^{-1}(\Lambda)$ nor to $\Phi(\Lambda)$. However, $\Phi(\y^n)\to\z$ as $n\to\infty$, which means that $\Phi(\Lambda)$ is not closed.

In particular, observe that taking $\mathcal{Z}_k:=[(z_\ell)_{0\leq \ell\leq k}]_{B^\NZ}$, for all $k\geq 0$, it follows that $\Phi^{-1}(\mathcal{Z}_k)=[(z_{\ell-1})_{1\leq \ell\leq k+1}]_{\Lambda}$\sloppy, which, from the definition of $\Lambda$, contains only the points $\y^n$ with $n\geq k+1$. Hence, $\bigcap_{k\geq 1}\Phi^{-1}(\mathcal{Z}_k)=\emptyset$.

\end{ex}

Note that to assure that $\Phi(\Lambda)$ is closed we do not need all nested family of nonempty cylinders $\Lambda$ having nonempty intersection, but only that such property works for any family $\{W^\ell_{j(\ell)}\}_{\ell\geq 1}$ of cylinders of $\Lambda$ which are defined from \eqref{eq:U_ell} and \eqref{eq:U_ell-finite}. Thus, the closedness of $\Phi(\Lambda)$ depends on the relationship between the local rule of $\Phi$ and the structure of $\Lambda$. In \cite{Campos_et_Al2021} it was presented sufficient conditions as on $\Lambda$ as on the local rule of $\Phi:\Lambda\subset A^\Z\to \Phi(\Lambda)\subset B^\Z$, for  $\Phi(\Lambda)$ to be a shift space \cite[Theorems 3.1 and 3.2]{Campos_et_Al2021} and for $\Phi^{-1}$ (when $\Phi$ is invertible) to be a generalized sliding block code too \cite[Theorem 3.3]{Campos_et_Al2021}. Theorem below presents some conditions which allow to get some results like those in \cite{Campos_et_Al2021}.

\begin{theo}\label{theo: image_of_shifts} Let $\Lambda\subset A^\NZ$ be a shift space and $\Phi:\Lambda\to B^\NZ$ be a locally finite-to-one generalized sliding block code, and 
for each $b\in B$ let $C_b:=\Phi^{-1}([b_1]_{B^\NZ})$.  If at least one of the following conditions holds:
\begin{enumerate}
\item[(A1)] $W_{\{1\}}(\Lambda)$ is finite;

\item[(A2)] $\NZ$ is countable and each $C_b$ can be written as the union of cylinders defined on the coordinate $1$;

\item[(A3)] $\NZ$ is a countable group  and exists $n\in\NZ$ such that each $C_b$ can be written as the union of cylinders defined on the coordinate $n$;

\end{enumerate}
then it follows that:

\begin{enumerate}
\item $\Phi$ is a closed map;

\item $\Phi(\Lambda)$ is a shift space;

\item If $\Phi$ is injective, then $\Phi^{-1}:\Phi(\Lambda)\to\Lambda$ is also a generalized sliding block code (but not necessarily locally finite-to-one).
\end{enumerate}

\end{theo}

\begin{proof}\phantom\\

\begin{enumerate}

\item If $W_{\{1\}}(\Lambda)$ is finite, then $\Lambda$ is compact. Since $B^\NZ$ is Hausdorff and $\Phi$ is continuous, it follows that $\Phi$ is a closed map (due to the Closed Map Lemma).\\

Now, suppose (A2) or (A3) holds, and let $S\subset \Lambda$ be a closed set. To check that $\Phi(S)$ is closed we will proceed as in Theorem \ref{theo: image_of_shifts-intersection}. Let $(\y^i)_{i\geq 1}$ be a sequence in $\Phi(S)$ converging to $\z=(z_g)_{g\in \NZ}\in B^\NZ$. Let us check that $\z\in \Phi(S)$.
Let $\{g_\ell\}_{\ell\geq 1}$ be  an enumeration of $\NZ$, and for each $k\geq 1$ let $\mathcal{Z}_k:=\big[(z_{g_\ell})_{1\leq\ell\leq k}\big]_{B^\NZ}$. Then, following as in \eqref{eq:U_ell} and \eqref{eq:U_ell-finite}, we have

\begin{equation}\label{eq:U_ell_2}\begin{array}{lcl}\emptyset&\neq& S\cap \Phi^{-1}(\mathcal{Z}_k)
=S\cap\bigcap_{1\leq\ell\leq k}
\Phi^{-1}\big([z_{g_\ell}]_{B^\NZ}\big)\\\\
&=&S\cap\bigcap_{1\leq\ell\leq k}U_\ell
=S\cap\bigcap_{1\leq\ell\leq k}\s^{g_\ell^{-1}}(C_{b^{\ell}})=S\cap\bigcap_{1\leq\ell\leq k}\bigcup_{1\leq i\leq N(\ell)}W^\ell_i,\end{array}\end{equation}

where $b^{\ell}:=z_{g_\ell}$ and each $W^\ell_i$ is a cylinder of $\Lambda$. 

Recall that each $C_{b^{\ell}}$ can be written as the union of $N(\ell)$ cylinders. Thus,  
 if (A2) holds, then each of these cylinders can be taken defined on the coordinate 1, and so, each cylinder $W^\ell_i$ composing $\s^{g_\ell^{-1}}(C_{b^{\ell}})$ is defined on the coordinate $g_\ell$. On the other hand, if (A3) holds, then there exists $n\in \NZ$ which is a coordinate used in each of the $N(\ell)$ cylinders that compose $C_{b^{\ell}}$. Therefore, since $\NZ$ is a group, given any $h\in\NZ$, there exists $g_\ell$ such that $g_\ell n=h$, and so, each $W^\ell_i$ composing $\s^{g_\ell^{-1}}(C_{b^{\ell}})$ is defined on the coordinate $h$. Hence, we have that as (A2) as (A3) implies that there exists $K$ such that any coordinate of $\NZ$ will be eventually fixed in $\Phi^{-1}(\mathcal{Z}_k)$ for all $k\geq K$. Now, using the same argument used to obtain \eqref{eq:U_ell-intersect}, we get that for each $\ell$ there exists $1\leq j(\ell)\leq N(\ell)$ such that
$$\emptyset\neq S\cap\bigcap_{1\leq\ell\leq k} W^\ell_{j(\ell)}\subset S\cap\bigcap_{1\leq\ell\leq k}U_\ell= S\cap\Phi^{-1}(\mathcal{Z}_k),\qquad \forall k\geq 1.$$
 
Hence, we can take a sequence $(\x^k)_{k\geq 1}\in S$ where $\x^k\in S\cap\bigcap_{1\leq\ell\leq k} W^\ell_{j(\ell)}$. It follows that $(\x^k)_{k\geq 1}$ converges to some $\x$ in $A^\NZ$, and since $S$ is closed we conclude that $\x\in S$. Finally, from the continuity of $\Phi$, we have $\Phi(\x)=\lim_{k\to\infty}\Phi(\x^k)=\z$.

\item  From the previous item, under (A1), (A2) or (A3) we have that $\Phi$ is a closed map, and then $\Phi(\Lambda)$ is closed in $B^\NZ$. On the other hand, $\Phi$ commutes with any $g$-shift map, and then $\Phi(\Lambda)$ is invariant under any $g$-shift map.

\item Suppose $\Phi:\Lambda\to\Phi(\Lambda)$  is invertible. It is direct that $\Phi^{-1}$ commutes with any $g$-shift map. On the other hand, if (A1), (A2) or (A3) holds, then $\Phi$ is also an open map (as consequence of it to be an invertible closed map), which means that $\Phi^{-1}$ is continuous (in particular, if (A1) holds, then $W_{\{1\}}(\Lambda)$ and $W_{\{1\}}(\Phi(\Lambda))$ are finite and $\Phi^{-1}$ is necessarily a locally bounded finite-to-one sliding block code).

\end{enumerate}

\end{proof}

We remark that theorems \ref{theo: image_of_shifts-intersection} and \ref{theo: image_of_shifts} do not exhaust all possible conditions under which $\Phi(\Lambda)$ is a shift space. In Section \ref{sec:HBCs_and_HBSs} we will study a particular class of generalized sliding block codes for which more general conditions can be obtained.\\

Note that the classes of locally finite-to-one maps and locally bounded finite-to-one maps are both closed for compositions. Furthermore, while in the infinite alphabet case the class of locally bounded finite-to-one SBCs, the class of locally finite-to-one SBCs, the class of SBCs, and the class of GSBCs do not coincide (see Figure \ref{Fig_GSBC-SBC}), when the alphabet is finite all four classes collapses in the single class of locally bounded finite-to-one SBCs.

\begin{figure}[H]\
\centering
\includegraphics[width=0.8\linewidth=1.0]{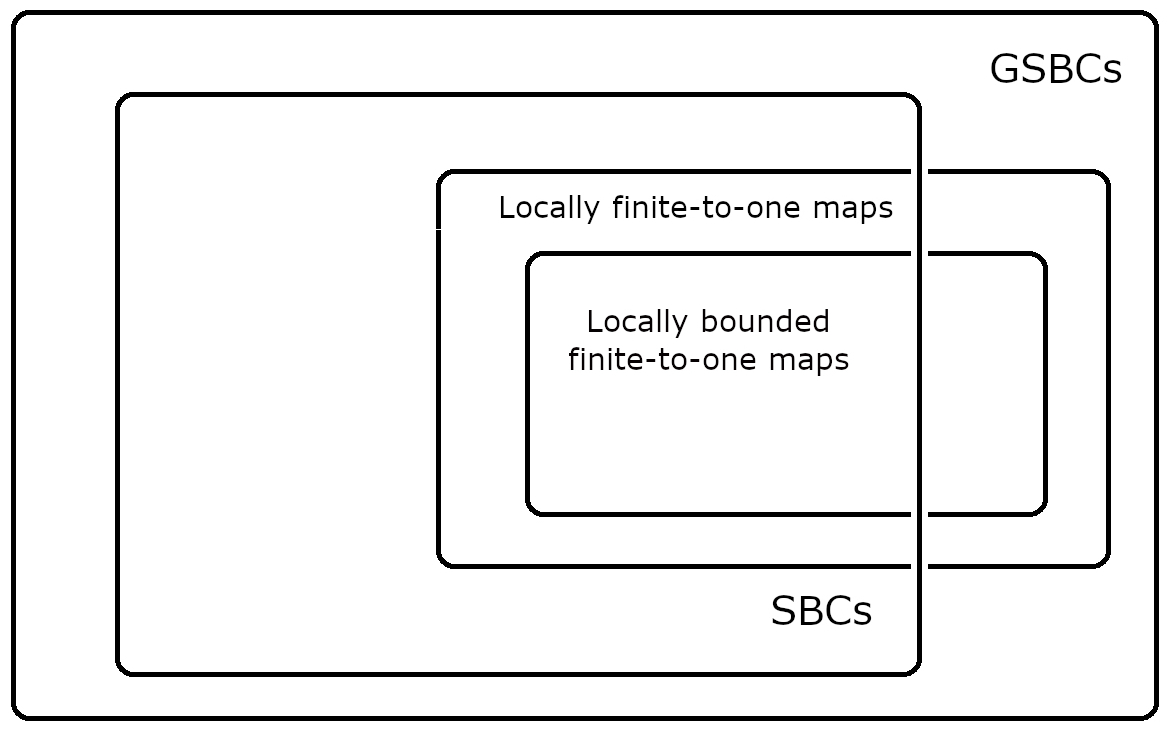}
\caption{The relationship between classes of continuous shift-commuting maps.}\label{Fig_GSBC-SBC}
\end{figure}

We recall that a topological dynamical system is a pair $(X,T)$ where $X$ is a topological space and $T:X\to X$ is a continuous map. In the context of shift spaces, we have that $\Lambda\subset A^\NZ$ with all the actions $\s^g$ conform a topological dynamical system, and generalized sliding block codes play the role of topological conjugacies, topological factors and topological extensions between the shift spaces.

\begin{defn}\label{defn:conjugacy}
We say that two shift spaces $\Lambda\subset A^\NZ$ and $\Gamma\subset B^\NZ$ are {\bf topologically conjugate} (or simply {\bf conjugate}) if and only if there exists an invertible generalized sliding block code $\Phi:\Lambda\to\Gamma$ such that $\Phi^{-1}$ is also a generalized sliding block code (in such a case, we say $\Phi$ is a {\bf (topological) conjugacy} between $\Lambda$ and $\Gamma$).

We say that the shift space $\Gamma$ is a {\bf topological factor} (or simply a {\bf factor}) of $\Lambda$ if and only if there exists an onto generalized sliding block code $\Phi:\Lambda\to\Gamma$ (in such a case, we say $\Phi$ is a {\bf (topological) factor map} from $\Lambda$ to $\Gamma$). When $\Gamma$ is a factor of $\Lambda$ we can equivalently to say that $\Lambda$ is a {\bf topological extension} (or simply an {\bf extension}) of $\Gamma$.

Whenever the map $\Phi$ (together its inverse, when it is the case) is a sliding block code, we will say that the above objects are {\bf uniform/uniformly} (e.g., $\Phi:\Lambda\to\Gamma$ is a uniform conjugacy and $\Lambda$ and $\Gamma$ are uniformly conjugate).
\end{defn}

\section{Higher block codes and higher block shifts}\label{sec:HBCs_and_HBSs}

We recall here the definition of the {\bf $N^{th}$-higher block shift}, which is given in \cite[Definition 1.4.1]{LindMarcus} for finite-alphabet shifts on the lattice $\N$ or $\Z$ with the usual sum:  For $N\in \N^*$, the {\bf $N^{th}$-higher block code} is the sliding block code $\Phi^{[N]}:A^\NZ \to (A^N)^\NZ$ given by \begin{equation*}\label{eq:Nth_higher_block_code}\Phi^{[N]}\big((x_i)_{i\in\NZ}\big)=(x_i...x_{N-1+i})_{i\in\NZ};\end{equation*} and 
 given a shift space $\Lambda\subset A^\NZ$, the {\bf $N^{th}$-higher block presentation} of $\Lambda$ is $\Lambda^{[N]}\subset (A^N)^\NZ$ given by
\begin{equation}\label{eq:Nth_higher_block_presentation}\Lambda^{[N]}:=\Phi^{[N]}(\Lambda)=\{(w_{i}...w_{N-1+i})_{i\in\NZ}:\ (w_i)_{i\in\NZ}\in\Lambda\},\end{equation}
which is also referred as the {\bf $N^{th}$-higher block shift} of $\Lambda$.

For shift spaces over infinite alphabets or on other lattices we need a definition which encapsulates all ways one can encode sets of coordinates as the symbols of a new shift space:

\begin{defn}\label{defn:N-higher_block_shift} Let  $\Lambda\subset A^\NZ$ be any shift space and $\bfN$ be a partition of $\Lambda$ by cylinders. Let $\bfM_\bfN:=\{M\subset\NZ:\ \exists a_i\in A\ s.\ t.\ [(a_i)_{i\in M}]_{\Lambda}\in \bfN\}$, that is, $M\in \bfM_\bfN$ if and only if it is a set of $\NZ$ which contains exactly all the coordinates of some cylinder of $\bfN$.

Given $\x\in \Lambda$, let $Z_\x\in \bfN$ be the cylinder which contains $\x$, and $M_\x\in\bfM_\bfN$ be the set off all coordinates used in the definition of $Z_\x$. Let $$A^{[\bfN]}:=\{(a_i)_{i\in M}\in\NN_{A^\NZ}:\ [(a_i)_{i\in M}]_{\Lambda}\in\bfN\},$$ 
and define the {\bf $\bfN$-higher block code} as the map $\Phi^{[\bfN]}:\Lambda \to (A^{[\bfN]})^\NZ$ given by 
$$\Phi^{[\bfN]}(\x):=\left(\x_{gM_{\s^g(\x)}}\right)_{g\in\NZ}\qquad(where\ \x_{gM_{\s^g(\x)}}\equiv(x_{gi})_{i\in M_{\s^g(\x)}}\in A^{M_{\s^g(\x)}},\ for\ all\ g\in\NZ).$$

Hence, the {\bf $\bfN$-higher block presentation} of $\Lambda$ is
 $\Lambda^{[\bfN]}\subset (A^{[\bfN]})^\NZ$  given by $$\Lambda^{[\bfN]}:=\Phi^{[\bfN]}(\Lambda)$$
 If  the $\bfN$-higher block presentation of $\Lambda$ is a shift space, then it is said to be the {\bf $\bfN$-higher block shift} of $\Lambda$.
 
\end{defn}

Note that the $N^{th}$-higher block code given in \cite[Definition 1.4.1]{LindMarcus} for shift spaces on the lattices $\N$ and $\Z$ corresponds to a $\bfN$-higher block code where $\bfN$ is the family of all cylinders on the coordinates $\{0, 1, ..., N-1\}$. In fact, in this case we have $\bfM_\bfN=\big\{\{0,1,...,N-1\}\big\}$ which implies $M_\x=\{0,1,...,N-1\}$ for all $\x\in A^\NZ$, $A^{[\bfN]}=A^{\{0,1,...,N-1\}}$, and then $\Phi^{[\bfN]}(\x)=(\x_{g+M_{\s^g(\x)}})_{g\in\NZ}=(x_gx_{g+1}...x_{g+N-1})_{g\in\NZ}$.

The next result characterizes the higher block codes as generalized sliding block codes and gives some sufficient conditions under which a shift and its higher block presentation are uniformly conjugated  (generalizing \cite[Example 1.5.5]{LindMarcus} and \cite[Example 1.5.10]{LindMarcus} stated for classical shifts).

\begin{rmk}\label{rmk:higher_block_presentation} For any $\bfN$ partition of $\Lambda$ by cylinders, it follows that  $\Phi^{[\bfN]}:\Lambda \to (A^{[\bfN]})^\NZ$ is a locally bounded finite-to-one GSBC with order 1. In fact, just observe that  $\Phi^{[\bfN]}$  can be written in the form of \eqref{eq:LR_block_code} with $C_b=[(\beta_i)_{i\in M}]_\Lambda$ for each $b:=(\beta_i)_{i\in M}\in A^{[\bfN]}$. 

Furthermore, $\bfM_\bfN$ being finite means that $L=\bigcup_{M\in\bfM_\bfN}M$ is finite and, since any cylinder of $\bfN$ is defined on coordinates contained in $L$, from Remark \ref{rmk:finite_coordinates} we get that $\bfM_\bfN$ being finite is equivalent to $\Phi^{[\bfN]}$ be an SBC.
\end{rmk}

Given a point $\mathbf{b}=(b_g)_{g\in\NZ}\in(A^{[\bfN]})^\NZ$,  recall that for each $g\in\NZ$ we have $b_g=(\beta^g_i)_{i\in M^g}\in A^{[\bfN]}$ for some $M^g\in\bfM_\bfN$. We will say that $\mathbf{b}\in(A^{[\bfN]})^\NZ$ holds the {\bf overlapping condition} if, and only if, it is such that
 \begin{equation}\label{eq:overlapping_condition}  g,h\in\NZ,\  m\in M^g, n\in M^h, \text{ such that } gm=hn\qquad\Longrightarrow\qquad\beta^g_{m}=\beta^{h}_{n}.\end{equation}

We notice that whenever  $1\in\bigcap_{M\in\bfM_\bfN}M$, the overlapping condition is equivalent to have for all $g\in\NZ$ and $m\in M^g$ that $\beta^g_m=\beta^{gm}_1$ (the proof is left to the reader).

\begin{lem}\label{lem:overlapping_condition} Let $\bfN$ be a partition of $\Lambda\subset A^\NZ$ by cylinders. If $\mathbf{b}\in(A^{[\bfN]})^\NZ$ belongs to $\Lambda^{[\bfN]}$, then it verifies the overlapping condition.
\end{lem} 

\begin{proof}
Let $\mathbf{b}=(b_g)_{g\in\NZ}=\big((\beta^g_i)_{i\in M^g}\big)_{g\in\NZ} \in(A^{[\bfN]})^\NZ$.
We have that 
$\mathbf{b}\in \Lambda^{[\bfN]}=\Phi^{[\bfN]}(\Lambda)$ if and only if there exists a sequence
$\x:=(\gamma_g)_{g\in \NZ}\in \Lambda$ such that $\Phi^{[\bfN]}(\x)=\mathbf{b}$, that is,  
$$\Phi^{[\bfN]}\big((\gamma_g)_{g\in \NZ})=\big((\gamma_{gi})_{i \in M_{\s^g(\x)}}\big)_{g\in \NZ}
=\big((\beta^g_i)_{i\in  M^g}\big)_{g\in\NZ}.$$

Thus, from the last equality above, we get that if $\mathbf{b}\in \Lambda^{[N]}$, then for all $g\in \NZ$ and $i\in M_{\s^g(\x)}=M^g$ we have $\gamma_{gi}=\beta^g_i$. Hence, for all $h\in \NZ$ and $j\in M^h$ such that $gi=hj$, it follows that $\beta^g_i=\gamma_{gi}=\gamma_{hj}=\beta^h_j$.

\end{proof}

\begin{cor}\label{cor:full_shift-non_overlapping_condition} Let $\bfN$ be a partition of $A^\NZ$ by cylinders. We have that  $\mathbf{b}\in(A^{[\bfN]})^\NZ$ belongs to $(A^\NZ)^{[\bfN]}$ if, and only if, it verifies the overlapping condition.
\end{cor}

\begin{proof}
We only need to check that in the case of the full shift, the overlapping condition is a sufficient condition. Let  $\mathbf{b}=(b_g)_{g\in\NZ}=\big((\beta^g_i)_{i\in M^g})\big)_{g\in\NZ}\in(A^{[\bfN]})^\NZ$ a point satisfying the overlapping condition and define $\x=(\gamma_h)_{h\in\NZ}\in A^\NZ$ where $\gamma_h=\beta^g_i$ if $h=gi$, and $\gamma_h$ is any symbol of $A$ if $h\neq gi$ for all $g\in\NZ$ and $i\in M^g$. Since $\mathbf{b}$ satisfies the overlapping condition, it follows that $\x$ is well defined. Now, to check that $\Phi^{[\bfN]}(\x)=\mathbf{b}$ observe that for any $g\in\NZ$, taking $M^g\in\bfM_\bfN$ (the set of coordinates where $b_g$ is defined) we have $\big(\s^g(\x)\big)_{M^g}=(\gamma_{gi})_{i\in M^g}=(\beta^g_i)_{i\in M^g}$, which implies that $\s^g(\x)$ belongs to the cylinder $[(\beta^g_i)_{i\in M^g}]_{A^\NZ}$ and so $\big(\Phi^{[\bfN]}(\x)\big)_g=(\beta^g_i)_{i\in M^g}=b_g$.

\end{proof}

\begin{cor}\label{cor:full-shift_higher_block_presentation} Let $\bfN$ be any partition of $A^\NZ$ by cylinders. Then $(A^\NZ)^{[\bfN]}$ is a shift space.
\end{cor}

\begin{proof} 
If $\NZ$ is countable, then the result follows directly for Corollary \ref{cor:image_Phi-fullshift}. To prove the general case, define $P_O\subset\NN_{(A^{[\bfN]})^\NZ}$ as 
\begin{equation}\label{eq:P_O}P_O:=\{ b_gb_h \in \NN_{(A^{[\bfN]})^\NZ}: b_g=(\beta^g_i)_{i\in L}, \ b_h=(\beta^h_i)_{i\in M},\ \exists \ell\in L, m\in M\ s. t.\ g\ell=hm\ and\ \beta^g_\ell \neq \beta^h_m\},\end{equation}
that is, $P_O$ is composed by all patterns of $\mathcal{N}_{A^{[\bfN]}}^2$ for which the overlapping condition fails.
Hence, the shift space $X_{P_O}\subset (A^{[\bfN]})^\NZ$ is composed for all sequences for which the overlapping condition holds, and then, from Corollary \ref{cor:full_shift-non_overlapping_condition} we have that $X_{P_O}=(A^\NZ)^{[\bfN]}$.

\end{proof}

We remark that without assuming more properties for $\NZ$ and for $\bfN$ we cannot get any result about the cardinality of $M_{P_O}$. In fact, if $\bfN$ is such that $\bfM_\bfN$ contains infinitely many coordinates or if $\NZ$ is not right cancellative, then it is possible that $M_{P_O}$ will have infinitely many elements. In the last part of Section \ref{sec:SFTs} we examine some conditions under which $M_{P_O}$ is finite.\\

The next proposition gives sufficient conditions under which the higher block presentation of a shift space is also a shift space. For classical shifts (over finite alphabets and on the lattices $\N$ or $\Z$), such result always holds (directly from \cite[Proposition 1.4.3]{LindMarcus} or, in our framework, as consequence of the fact that classical shift spaces hold both conditions (B1) and (B3) in the proposition below). However, in general it is not true that a higher block presentation is a shift space. For instance, observe that the map $\Phi$ in Example \ref{ex:Phi(Lambda)-nonshift} corresponds to the higher block code $\Phi^{[\bfN]}$ where $\bfN$ is the collection of all cylinders defined on the second coordinate of the lattice $\N$, and so for that shift $\Lambda$ we have that $\Lambda^{[\bfN]}$ is not a shift space.

We also notice that the next proposition uses the specific structure of the generalized sliding block code $\Phi^{[\bfN]}$ to find other conditions than those given by Theorem \ref{theo: image_of_shifts}.

\begin{prop}\label{prop:higher_block_shift} Given a shift space $\Lambda\subset A^\NZ$ and  $ \bfN$ a partition of $\Lambda$ by cylinders, if at least one of the following conditions holds:
\begin{enumerate}

\item[(B1)] $W_{\{1\}}(\Lambda)$ is finite;

\item[(B2)] $\Lambda=A^\NZ$;

\item[(B3)] $\NZ$ is countable and $\Lambda$ is such that for all nested family of nonempty cylinders $\{W_\ell\}_{\ell\geq 1}$ it follows that $\bigcap_{i\geq 1}W_i\neq\emptyset$;

\item[(B4)] $\NZ$ is a group and $\bigcap_{M\in\bfM_\bfN}M\neq\emptyset$;

\item[(B5)] $1\in\bigcap_{M\in\bfM_\bfN}M$;

\end{enumerate}
then $\Lambda^{[\bfN]}$ is also a shift space.

In particular, if (B4) or (B5) holds, then $\Phi^{[\bfN]}$ restricted to its image is invertible and $\left(\Phi^{[\bfN]}\right)^{-1}$ is an SBC (but in general it is is not locally finite-to-one).

\end{prop}

\begin{proof}then $\Lambda^{[\bfN]}$ is also a shift space.

If condition (B1), (B2) or (B3) holds, since $\Phi^{[\bfN]}$ is a locally bounded finite-to-one GSBC, the result comes directly from Theorem \ref{theo: image_of_shifts}, Corollary \ref{cor:full-shift_higher_block_presentation} or Theorem \ref{theo: image_of_shifts-intersection}, respectively.\\

Now suppose (B4) or (B5) holds. If (B4) holds, then take any $n\in\bigcap_{M\in\bfM_\bfN}M$, while if (B5) holds, then take $n=1$. 

Note that $(b_g)_{g\in \NZ}=\big((\beta^g_i)_{i\in M^g}\big)_{g\in\NZ}\in \Lambda^{[\bfN]}=\Phi^{[\bfN]}(\Lambda)$ means that there exists $\x=(\gamma_g)_{g\in \NZ}\in \Lambda$ such that for each $g\in\NZ$ we have $b_g=(\gamma_{gi})_{i\in M_{\s^g(\x)}}\in A^{[\bfN]}$ with $[(\gamma_{gi})_{i\in M_{\s^g(\x)}}]_{\Lambda}\in\bfN$. Observe that for each $g\in\NZ$ the symbol $\gamma_g$ is the symbol at coordinate $n$ in $b_{gn^{-1}}$ (where $n^{-1}$ stands for the inverse of $n$ if (B4) holds, and it stands for 1 if (B5) holds). Hence, $\Phi^{[\bfN]}$ is injective and we can define $\left(\Phi^{[N]}\right)^{-1}:\Lambda^{[\bfN]}\to\Lambda$ as the map with local rule on the neighborhood $\{n^{-1}\}$, that is, the map which takes $(b_g)_{g\in \NZ}$ to  $(\gamma_g)_{g\in \NZ}$ where $\gamma_g$ is the symbol at coordinate $n$ in $b_{gn^{-1}}$, that is, $\gamma_g=\beta^{gn^{-1}}_n$.

Let $P_O\subset\NN_{(A^{[\bfN]})^\NZ}$ be the set of forbidden patterns given in  \eqref{eq:P_O}.  Let $F\subset\NN_{A^\NZ}$ be a set of forbidden patterns such that $\Lambda=X_F$ and  define $P_F\subset\NN_{(A^{[\bfN]})^\NZ}$ as follows:

\begin{equation}\label{eq:P_F}(b_g)_{g\in K}=\big((\beta^g_i)_{i\in M^g}\big)_{g\in K}\in  P_F\qquad\Longleftrightarrow\qquad (\beta_n^{hn^{-1}})_{h\in Kn}\in F.\end{equation}

We will show that $\Lambda^{[\bfN]}= X_{P_O\cup P_F}$. First, let us check that $\Lambda^{[\bfN]}\subset X_{P_O\cup P_F}$. Given $\mathbf{b}=(b_g)_{g\in\NZ}=\big((\beta^g_i)_{i\in  M^g}\big)_{g\in\NZ}\in \Lambda^{[\bfN]}$, let $\x=(\gamma_g)_{g\in\NZ}\in\Lambda$ such that $\Phi^{[\bfN]}\big((\gamma_g)_{g\in \NZ})=\big((\gamma_{gi})_{i \in M_{\s^g(\x)}}\big)_{g\in \NZ}=\big((\beta^g_i)_{i\in  M^g}\big)_{g\in\NZ}=\mathbf{b}$\sloppy. From Lemma \ref{lem:overlapping_condition}, $\mathbf{b}$ satisfies the overlapping condition, that is, for any $g,h\in\NZ$ we have $(b_gb_h)\notin P_O$. Furthermore, since $\x\in \Lambda$, it follows that for all finite $L\subset \NZ$ and for all $g\in\NZ$ we have $\big(\s^g(\x)\big)_L=(\gamma_{gi})_{i\in L}\notin F$. Hence, if by contradiction we suppose there exists $K$ such that $(b_g)_{g\in K}\in P_F$, then $(\beta_n^{hn^{-1}})_{h\in Kn}=(\gamma_{h})_{h\in Kn}\in F$ which contradicts that $\x\in\Lambda$. Thus, $(b_g)_{g\in K}\notin P_F$ for all $K\subset\NZ$, that is,  $\mathbf{b}\in X_{P_F}$. Thus, we have $\mathbf{b}\in X_{P_O\cup P_F}$.

For the opposite inclusion, take $\mathbf{b}\in X_{P_O\cup P_F}$ and define $\x:=(\gamma_g)_{g\in\NZ}\in A^\NZ$ with $\gamma_g:=\beta^{gn^{-1}}_n$. Note that, $\x$ is defined in the same way that an inverse image of a point of $\Lambda^{[\bfN]}$ by $\Phi^{[\bfN]}$. We need to prove that $\x$ belongs to $\Lambda$ and that $\Phi^{[\bfN]}(\x)=\mathbf{b}$. In fact, since $\mathbf{b}\in X_{P_F}$ then for all finite $K\subset \NZ$ we have $(b_g)_{g\in K}=\big((\beta^g_i)_{i\in M^g}\big)_{g\in K}\notin  P_F$ and so $(\beta_n^{hn^{-1}})_{h\in Kn}\notin F$. Hence, given any $L\subset\NZ$, we can take $K=Ln^{-1}$, and then $(\gamma_h)_{h\in L}=(\beta_n^{hn^{-1}})_{h\in Kn}\notin F$. Thus, $\x\in\Lambda$. To check that $\Phi^{[\bfN]}(\x)=\mathbf{b}$ note that for all $g\in\NZ$ we have $$\big(\s^g(\x)\big)_{M^g}=(\gamma_{gi})_{i\in M^g}=(\beta^{gin^{-1}}_n)_{i\in M^g}=(\beta^g_i)_{i\in M^g},$$ where the last equality above comes from the fact that $\mathbf{b}\in X_{P_O}$  (that is, it satisfies the overlapping condition). Therefore, $\big(\Phi^{[\bfN]}(\x)\big)_g=(\beta^g_i)_{i\in M^g}=b_g$,
and we conclude that $\mathbf{b}\in\Lambda^{[\bfN]}$.

\end{proof}

We remark that conditions given in (B4) and (B5) are more general than those stated in conditions (A2) and (A3) of Theorem \ref{theo: image_of_shifts} since here we did not assume $\NZ$ being countable. Furthermore, (B4) and (B5) are sufficient conditions under which a shift and its higher block presentation are uniformly conjugated  (generalizing \cite[Example 1.5.5]{LindMarcus} and \cite[Example 1.5.10]{LindMarcus} stated for classical shifts).\\

Now, consider the general case of $\Lambda\subset A^\NZ$ being a shift space, and $\bfN$ being a partition of $\Lambda$ by cylinders, without any additional assumption on $\Lambda$, $\NZ$ or $\bfN$. Let $F\subset\NN_{A^\NZ}$, such that $X_F=\Lambda$, and define $D_F\subset\NN_{(A^{[\bfN]})^\NZ}$ as follows:

\begin{equation}\label{eq:D_F}\begin{array}{c}(b_g)_{g\in K}=\big((\beta^g_i)_{i\in M^g}\big)_{g\in K}\in  D_F\\\\ \Updownarrow\\\\ \forall g\in K,\ \exists\ \emptyset\neq N^g\subset  M^g,\ s.\ t.\ 
 (\gamma_h)_{h\in L} \in F,\\\\ where\ L:=\bigcup_{g\in K}gN^g\ and\ \gamma_h=\beta^g_n\ with\ gn=h.\end{array}\end{equation}
From its definition, one can easily check that $\Lambda^{[\bfN]}\subset X_{D_F}$, and thus  $\Lambda^{[\bfN]}\subset X_{P_O\cup D_F}$. However, the opposite inclusion is not true in general. First, if neither of conditions in Proposition \ref{prop:higher_block_shift} hold and $F$ is not complete in the sense of Definition \ref{defn:complete_F}, then it is possible that some restrictions of $F$ are not captured by $D_F$ and so  $X_{P_O\cup D_F}$ contains sequences that do not belong to $\Lambda^{[\bfN]}$ (see Example \ref{ex:F_non_complete} below). Moreover, even when $F$ is complete, it remains open whether without additional assumptions on $\Lambda$ or $\NZ$ (as those made in Proposition \ref{prop:higher_block_shift}) we have $\Lambda^{[\bfN]}= X_{P_O\cup D_F}$. To make clear what is the difficulty in proving that $\Phi^{[\bfN]}$ is onto on $ X_{P_O\cup D_F}$, let us prove the case where it does work, that is, the case (B3) of Proposition \ref{prop:higher_block_shift}.

\begin{prop}\label{prop:higher_block_shift-NZ_countable} Suppose $\NZ$ is countable, and $\Lambda\subset A^\NZ$ is a shift space such that for all nested family of nonempty cylinders $\{W_\ell\}_{\ell\geq 1}$ it follows that $\bigcap_{i\geq 1}W_i\neq\emptyset$. If $ \bfN$ is a partition of $\Lambda$ by cylinders and $F\subset\NN_{A^\NZ}$ is a complete set of forbidden patterns such that $\Lambda=X_F$, then $\Lambda^{[\bfN]}= X_{P_O\cup D_F}$, where $D_F$ is given by \eqref{eq:D_F}.
\end{prop}

\begin{proof}
Since the inclusion $\Lambda^{[\bfN]}\subset X_{P_O\cup D_F}$ is direct, we only need to prove that $\Phi^{[\bfN]}$ is onto on $ X_{P_O\cup D_F}$. Let $\mathbf{b}=(b_g)_{g\in\NZ}=\big((\beta^g_i)_{i\in M^g})\big)_{g\in\NZ}\in X_{P_O\cup D_F}$. Let $T\subset \NZ$ be the set of all coordinates $h\in\NZ$ for which there exist $g\in \NZ$ and $m\in M^g$ such that $h=gm$. Since $\mathbf{b}$ satisfies the overlapping condition, then for each $h\in T$ we can define $\gamma_h:=\beta^g_m \in A$. Observe that it is possible that $\NZ\setminus T\neq\emptyset$, and so  $\gamma_\ell$ is not defined for all $\ell\in \NZ\setminus T$.
Note that, for any choice of symbols $\gamma_\ell$ for $\ell\in \NZ\setminus T$, the sequence $\x:=(\gamma_j)_{j\in\NZ}\in A^\NZ$ obtained is such that $\big(\s^g(\x)\big)_{M^g}=(\gamma_{gm})_{m\in M^g}=(\beta^g_m)_{m\in M^g}$. Hence, we can extend $\Phi^{[\bfN]}$ for such sequences $\x$ independently of whether them belong to $\Lambda$ or not, that is, we can compute $\big(\Phi^{[\bfN]}(\x)\big)_{g}=(\beta^g_m)_{m\in M^g}=b_g$. Thus, we only need to prove that it is possible to define $\gamma_\ell$ for all $\ell\in \NZ\setminus T$ in some way that the point of $(\gamma_g)_{g\in\NZ}\in A^\NZ$ obtained belongs to $\Lambda$. 
Since $\NZ$ is countable, we can take $\{h_i\}_{i\geq 1}$ any enumeration of $T$, and for $k\geq 1$ define the cylinder $Z_k:=[(\gamma_{h_i})_{1\leq i\leq k}]_\Lambda$. Since $\mathbf{b}$ does not contain any forbidden pattern of $D_F$, from the definition of $D_F$ given by \eqref{eq:D_F} and from the definition of $(\gamma_h)_{h\in T}$ we have that $(\gamma_{h_i})_{1\leq i\leq k}\notin F$. Furthermore, since $F$ is complete, then $Z_k$ is not empty. Hence,  $\{Z_k\}_{k\geq 1}$ is a nested family of nonempty cylinders of $\Lambda$ and then $\bigcap_{k\geq 1} Z_k\neq\emptyset$. Finally, taking any  $\x\in \bigcap_{k\geq 1} Z_k\subset\Lambda$ it follows that $x_h=\gamma_h$ for $h\in T$ and then $\Phi^{[\bfN]}(\x)=\mathbf{b}$.

\end{proof}

\begin{ex}\label{ex:F_non_complete}
Let $A=\{0,1\}$ and $\N^*$ with the usual product. Define $F:=\{(x_2x_5x_{10})\in\NN_{A^{\N^*}}:\ x_2=x_{10}=1,\ x_5\in A\}$ which is not complete. Let $\Lambda:=X_F$ and let $\bfN$ be the collection of all cylinders of $\Lambda$ on the coordinates 2 and 3. Hence  $A^{[\bfN]}:=\left\{\left(\begin{array}{c}\beta_2\\ \beta_3\end{array}\right): \beta_2,\beta_3\in A\right\}$, and thus $(A^{[\bfN]})^{\N^*}=\left\{\left(\begin{array}{c}\beta^i_2\\ \beta^i_3\end{array}\right)_{i\in\N^*}:\ \left(\begin{array}{c}\beta^i_2\\ \beta^i_3\end{array}\right)\in A^{[\bfN]}\ \forall i\in\N^* \right\}$\sloppy. Observe that all $g\in \N^*$ and $n\in\{2,3\}$ we never have $gn=5$, and since $5$ is a coordinate that appears in all the forbidden words of $F$, it implies that $D_F$ given in \eqref{eq:D_F} is the empty set. Thus $X_{P_O\cup D_F}=X_{P_O}=(A^{\N^*})^{[\bfN]}$. On the other hand,  
 for all  $\x=(x_i)_{i\in\N^*}\in \Lambda$ we have
$\Phi^{[\bfN]}(\x)=\left(\begin{array}{c}x_{2i}\\x_{3i}\end{array}\right)_{i\in\N^*}=\left(\begin{array}{c}\beta^i_2\\ \beta^i_3\end{array}\right)_{i\in\N^*}$, and then we always have $\left(\begin{array}{c}\beta^1_2\\ \beta^1_3\end{array}\right)$ and $\left(\begin{array}{c}\beta^5_2\\ \beta^5_3\end{array}\right)$ such that $\beta^1_2\neq 1$ or $\beta^5_2\neq 1$, and then $\Lambda^{[\bfN]}\subsetneq (A^{\N^*})^{[\bfN]}$.

\end{ex}

\section{Shifts of finite type}\label{sec:SFTs}

In the classical theory of symbolic dynamics over finite alphabets, an important role is played by shifts of finite type. These shift spaces are defined in the classical framework as the shift spaces that can be obtained by forbidding only a finite number of words in the language.  This definition falls apart when one considers infinite alphabets. In fact, even finite-alphabet shift space cannot fulfill such definition if one considers them as subspaces of some infinite-alphabet shift. For instance, according to this definition the shift space $\{0,1\}^\Z$ is never a shift of finite type if we consider it as a subset of $\N^\Z$. Although this problem can be solved by considering that to ``be a shift of finite type'' is a property relative to the universe where we are working, the important features of shifts of finite type (shadowing property, ergodicity, informational content, etc.) are not relative properties. In fact, there is not any change in the dynamical behavior of $\{0,1\}^\Z$ whether it is considered or not a subset of some infinite-alphabet shift. 

In this section we propose an alternative definition for shifts of finite type, which is independent of the context, and coincides with the classical definition when considering finite-alphabet shift spaces. The alternative definition we are going to propose, lays on the simple observations that in the context of finite-alphabet shifts, $F\subset\NN_{A^\NZ}$ to be finite is equivalent to $M_F$ given in \eqref{eq:M_F_defn} to be finite. Thus, we can define:

\begin{defn}\label{defn:SFT} A shift space $\Lambda\subset A^\NZ$ is said to be a {\bf shift of finite type (SFT)} if there exists $F\subset\NN_{A^\NZ}$, with $M_F$ finite, such that $\Lambda=X_F$. A shift of finite type $\Lambda$ is said to be an {\bf $M$-step shift} or to have a {\bf memory $M$} for some finite $M\subset \NZ$ if $M_F\subset M$.

\end{defn}

Note that $\Lambda$ is an $M$-step SFT if and only if there exists $F\subset A^M$ such that $\Lambda=X_F$.

While Definition \ref{defn:SFT} coincides with the traditional one when considering finite-alphabet shifts, it allows to consider them as subset of infinite-alphabet shifts. On the other hand, for infinite-alphabet shift spaces, this definition includes several shift spaces over infinite alphabets that behave similar to the classical SFTs over finite alphabets (e.g., countable topological Markov chains \cite{Gurevic1969,Sarig1999} and shift spaces with the shadowing property \cite{DarjiGoncalvesSobottka2020,GoodMeddaugh2020}). An SFT may also be referred as: finite-step shift, finite-order shift, shift with finite memory or shift of bounded type. Such names allude the finiteness of number of coordinates used to define the forbidden words  (see, for instance, \cite{DarjiGoncalvesSobottka2020,LindMarcus,MeddaughRaines2020}). Note that when the lattice is $\N$ or $\Z$, the forbidden words are usually defined on consecutive coordinates starting at coordinate zero and so the ``step'', ``memory'' or ``order'' of an SFT can be taken as single numbers $m\in\N$ related to the maximum quantity of consecutive coordinates needed to write any forbidden word, that is,  a shift  $\Lambda\subset A^\Z$ has step $m\in\N$ if and only if it is an $M$-step shift for $M=\{0,1,2,...,m\}$.\\

It is easy to check that the intersection of any family of shift spaces over the same alphabet will be a shift space (whose set of forbidden patterns correspond to the union of the sets of forbidden patterns of each shift space in the family). In particular, if the family is composed by $L$-step shifts, then the intersection of its members will also be an $L$-step shift . However, it is not true in general that the union of infinitely many shift spaces is a shift space. In fact, if for $k\geq 1$ we consider $\Lambda_k\subset \{0,1\}^\N$ being the shift space given by $\Lambda_k:=\{\s^n(\x): \x\text{ is the sequence such that }x_0=x_{ik}=1 \text{ and } x_j=0\ \forall i\in\N \text{ and } j\neq ik, n\in\N\}$\sloppy, then $\bigcup_{k\geq 1}\Lambda_k$ is not a shift space, since the sequence with only $0$s belongs to the closure of $\bigcup_{k\geq 1}\Lambda_k$ but does not belong to $\bigcup_{k\geq 1}\Lambda_k$. It is easy to check that a sufficient condition for a union of shift spaces being a shift space is that the alphabets of such shift spaces do not share any symbol. The following lemma gives a sufficient condition for a union of SFTs to be an SFT when considering the one-dimensional lattices $\N$ or $\Z$ with the usual sum.

\begin{lem}\label{lem:union_of_SFTs} Let $\NZ$ be $\N$ or $\Z$ with the usual sum, and $I\subseteq \N$. Let $\{A_k\}_{k\in I}$ be a disjoint family of nonempty sets. Let $\ell\in\N$ and for each $k\in I$ let $\Lambda_k\subset A_k^{\NZ}$ be an $\ell$-step SFT. Then $\Lambda:=\bigcup_{k\in I}\Lambda_k$ is an $m$-step SFT where $m=\max\{1,\ell\}$.
\end{lem}

\begin{proof}
Let $A:=\bigcup_{k\in I}A_k$ and $\Lambda:=\bigcup_{k\in I} \Lambda_k\ \subset\ A^\NZ$. For each $k\in I$, let $F_k\subset A_k^{\ell+1}$ such that $\Lambda_k=X_{F_k}$ (such $F_k$ exists since $\Lambda_k$ is an $\ell$-step shift).

It follows that $\Lambda=X_F$ with $F=\{(x_0x_1):\ x_0\in A_i\text{ and } x_1\in A_j \text{ for all }i\neq j\}\cup\bigcup_{k\in I} F_k$, and then $M_F\subset\{0,1,...,m\}$ where $m=\max\{1,\ell\}$.

\end{proof}

Note that for other lattices it could not be true that even the union of finitely many $L$-steps shifts over distinct alphabets results in an $M$-step shift. For instance, if we consider the lattice $\N^*$ with the usual product, the alphabets $A_k:=\{a_k,b_k\}$ for $k=1,2$ and the $\{1,2\}$-step shift spaces $\Lambda_k:=X_{F_k}$ where $F_k:=\{(x_1,x_2):\ (x_1,x_2)=(a_k,a_k)\}$, then $\Lambda=\Lambda_1\cup\Lambda_2$ is not an $M$-step shift for any finite $M\subset\NZ$, since using the analysis made in Example \ref{ex:shift_spaces_SFT} any $F$ such that $X_F=\Lambda$ should forbid sequences such that $x_g\in A_1$ and  $x_h\in A_2$ for all co-primes $g,h\in\N^*$.\\

The next lemma gives sufficient conditions under which a shift of finite type holds the condition stated in Theorem \ref{theo: image_of_shifts-intersection}.

\begin{lem}\label{lem:SFTs_finite-tersect} Suppose $\NZ$ is countable and cancellative (that is, left and right cancellative), and let $\Lambda\subset A^\NZ$ be a shift of finite type. If $\{W_\ell\}_{\ell\geq 1}$ is a nested family of nonempty cylinders of $\Lambda$, then $\bigcap_{i\geq 1}W_i\neq\emptyset$.
\end{lem}

\begin{proof} Since $\NZ$ is countable, we can identify it with the set $\N^*$.
Since $\Lambda$ is an SFT, there exists a finite set of coordinates $M\subset\NZ$, say with $\sharp M=m$, such that $\y\in A^\NZ$ belongs to $\Lambda$ if and only if $\y_{gM}\in W(\Lambda)$ for all $g\in \NZ$. 

Let $\{W_\ell\}_{\ell\geq 1}$ be a nested family of nonempty cylinders of $\Lambda$. Without loss of generality we can consider that there exists $T:=\{g_i\}_{i\geq 1}\subset \NZ$ and $(a_{g_i})_{i\geq 1}\in A^T$ such that for all $\ell\geq 1$ we have   $W_\ell=[(a_{g_i})_{i\leq \ell}]_{\Lambda}$. For each $\ell\geq1$ define $U_\ell:=[(a_{g_i})_{i\leq \ell}]_{A^\NZ}$. Since $\{U_\ell\}_{\ell\geq 1}$ is a nested family of nonempty cylinders of $A^\NZ$, then $\bigcap_{\ell\geq 1}U_\ell=\{\x\in A^\NZ: x_{g_i}=a_{g_i}\ \forall i\geq 1\}\neq\emptyset$.

Let us check that $\bigcap_{\ell\geq 1}W_\ell=\{\x\in \Lambda: x_{g_i}=a_{g_i}\ \forall i\geq 1\}\neq\emptyset$. For any $N\subset \NZ$ denote $\breve{N}:=N\cap T$, and for $g\in\NZ$ denote $N^g:=gN$ and $\breve{N}^g:=N^g\cap T$. Note that since $M$ is finite and $\NZ$ is  cancellative, for each $g$, there exists at most $K= m^2$ elements of $\NZ$, say $h_1^g,...,h_{k(g)}^g$ with $k(g)\leq K$, such that $M^{h_i^g}\cap M^g\neq\emptyset$ for all $i\leq k(g)$ (in particular we can always consider $h_1^g:=g$).

Fix any $c\in A$ and take $\x^0\in \bigcap_{\ell\geq 1}U_\ell$ such that $x^0_i=c$ for all $i\in \NZ\setminus T$. Define $\prescript{}{1}N:=\bigcup_{i\leq k(1)}M^{h_i^1}$. Observe that $\x^0_{\prescript{}{1}{\breve{N}}}=(a_j)_{j\in \prescript{}{1}{\breve{N}}}$ belongs to the language of $\Lambda$, and then there exists a choice of $(b_j)_{j\in M^1\setminus \breve{M}^1}$, such that $(\alpha_j)_{j\in \prescript{}{1}{N}}$ defined by 
$$\alpha_j:=\left\{\begin{array}{lcl}
b_j&,& j\in M^1\setminus \breve{M}^1\\\\\
x^0_j&,& j\in  \prescript{}{1}{N}\setminus(M^1\setminus \breve{M}^1)\end{array},\right.$$
belongs to the language of $\Lambda$.
Now take  $\x^1\in \bigcap_{\ell\geq 1}U_\ell$ such that
$$x^1_j:=\left\{\begin{array}{lcl}
\alpha_j&,& j\in \prescript{}{1}{N}\\\\\
x^0_j&,& j\in  \NZ\setminus\prescript{}{1}{N}\end{array}.\right.$$
Define $\prescript{}{2}N:=\bigcup_{i\leq k(2)}M^{h_i^2}$ and  it follows that 
$\x^1_{\prescript{}{2}{\breve{N}}\cup M^1}$ belongs to the language of $\Lambda$, and then there exists a choice of $(b_j)_{j\in M^2\setminus (\breve{M}^2\cup M^1) }$, such that $(\alpha_j)_{j\in \prescript{}{2}{N}}$ defined by 
$$\alpha_j:=\left\{\begin{array}{lcl}
b_j&,& j\in M^2\setminus (\breve{M}^2\cup M^1)\\\\\
x^1_j&,& j\in  \prescript{}{2}{N}\setminus(M^2\setminus (\breve{M}^2\cup M^1))\end{array},\right.$$
belongs to the language of $\Lambda$.
Hence, we define 
$$x^2_j:=\left\{\begin{array}{lcl}
\alpha_j&,& j\in \prescript{}{2}{N}\\\\\
x^1_j&,& j\in  \NZ\setminus\prescript{}{2}{N}\end{array}.\right.$$
From here we proceed recursively by defining $\x^\ell\in \bigcap_{\ell\geq 1}U_\ell$ given by
$$x^\ell_j:=\left\{\begin{array}{lcl}
\alpha_j&,& j\in \prescript{}{\ell}{N}\\\\\
x^{\ell-1}_j&,& j\in  \NZ\setminus\prescript{}{\ell}{N}\end{array},\right.$$
where $\prescript{}{\ell}N:=\bigcup_{i\leq k(\ell)}M^{h_i^\ell}$ and 
$$\alpha_j:=\left\{\begin{array}{lcl}
b_j&,& j\in M^\ell\setminus (\breve{M}^\ell\cup\bigcup_{1\leq i\leq\ell-1} M^i)\\\\\
x^1_j&,& j\in  \prescript{}{\ell}{N}\setminus(M^\ell\setminus (\breve{M}^\ell\cup\bigcup_{1\leq i\leq\ell-1} M^i))\end{array},\right.$$
 with $(b_j)_{j\in M^\ell\setminus (\breve{M}^\ell\cup\bigcup_{1\leq i\leq\ell-1} M^i)}$ chosen in a way that $(\alpha_j)_{j\in \prescript{}{\ell}{N}}$ belongs to the language of $\Lambda$.
 
 From its construction, it follows that the sequence $(\x^\ell)_{\ell\geq 0}$ converges to some $\x\in A^\NZ$. In fact, for any $i\in\NZ$ there exists some $L\geq 0$ such that $x^\ell_i$ is fixed for all $\ell\geq L$. Since $\x^\ell_T=(a_{g})_{g\in T}$ for all $\ell\geq 0$, then $\x\in \bigcap_{\ell\geq 1}U_\ell$. Furthermore, for all $g\in\NZ$ there exists $L\geq 0$ such that $\x_{gM}=\x^\ell_{gM}=\x^\ell_{M^g}\in W(\Lambda)$ for all $\ell\geq L$, which means that $\x\in\Lambda$. Hence, we conclude that $\x \in \Lambda\cap\bigcap_{\ell\geq 1}U_\ell=\bigcap_{\ell\geq 1}\Lambda\cap U_\ell=\bigcap_{\ell\geq 1}W_\ell$.

\end{proof}

Hence, as direct consequence of  Theorem  \ref{theo: image_of_shifts-intersection} and Lemma \ref{lem:SFTs_finite-tersect}, we have the following corollary.

\begin{cor}\label{cor:image_of_SFT_by_Phi} Suppose $\NZ$ is countable and cancellative, and let $\Lambda\subset A^\NZ$ be a shift of finite type and $\Phi:\Lambda\to B^\NZ$ be a locally finite-to-one generalized sliding block. Then $\Phi(\Lambda)$ is a shift space.
\end{cor}

\qed

Note that SFTs on the lattices $\N^d$ or $\Z^d$ with the usual sum are particular cases of the previous corollary. We also notice that this corollary assures that $\Phi(\Lambda)$ is a shift space, but does not assure that it is an SFT (in fact, $\Phi(\Lambda)$ is a weakly sofic shift - see Section \ref{sec:sofic}). For the particular case of higher block codes, we shall present in Proposition \ref{prop:higher_block_shift_SFT-2} some sufficient conditions under which the image of an SFT is an SFT.

From Definition \ref{defn:SFT}, we have that $\Lambda\subset A^\NZ$ being an SFT means that to decide if any given $\x\in A^\NZ$ belongs or not to $\Lambda$, we need just  checking for each $g\in\NZ$ if $(x_{gi})_{i\in M}$ belongs or not to $F$. Such procedure could be described as follows: For each $g\in\NZ$ one assigns the value 0 if $(x_{gi})_{i\in M}\in F$ and the value 1 if $(x_{gi})_{i\in M}\notin F$; thus we conclude that $\x\in \Lambda$ if and only if for all $g$ we have assigned the value 1. Hence, since we can always take $F$ being complete in the sense that any pattern which is not in the language contains a subpattern of $F$, the previous discussion can be written as the following theorem which provides an alternative definition for SFTs:

\begin{theo}\label{theo:SBC_and_SFT} Let $\Lambda\subset A^\NZ$  be shift space and let $\mathbf{1}\in \{0,1\}^\NZ$ be the configuration which assigns $1$ to each site of $\NZ$. Then, $\Lambda$ is a shift of finite type if and only if there exists a sliding block code $\Phi:A^\NZ\to \{0,1\}^\NZ$ such that $\Lambda=\Phi^{-1}(\mathbf{1})$.
\end{theo}

\qed

This alternative definition of SFTs directly prove that:

\begin{cor}\label{cor:inverse_image_of_SFT}
If $\Phi:A^\NZ\to B^\NZ$ is a sliding block code and $\Gamma\subset B^\NZ$ is an SFT, then $\Lambda:=\Phi^{-1}(\Gamma)$ is also an SFT.
\end{cor}

\begin{proof} Since $\Gamma$ is an SFT, there exists a sliding block code $\Psi:B^\NZ\to \{0,1\}^\NZ$ such that $\Gamma=\Psi^{-1}(\mathbf{1})$. Hence, taking $\Psi\circ\Phi:A^\NZ\to \{0,1\}$ it follows that $\Lambda=(\Psi\circ\Phi)^{-1}(\mathbf{1})$, which means that $\Lambda$ is an SFT.

\end{proof}

Most of the results for finite-alphabet SFTs on the lattice $\N$ or $\Z$ hold for infinite-alphabet SFTs on these lattices (see for instance propositions 2.2.5 and 2.3.5 in \cite{DarjiGoncalvesSobottka2020}). In what follows, we state the general version of Theorem 2.1.10 in \cite{LindMarcus} which will be useful to understand some concepts throughout this work. The proof of this theorem is analogous to that given in  \cite{LindMarcus}.

\begin{theo}\label{theo:uniform_conjugacy-SFT} Let $\NZ$ be the lattice $\N$ or $\Z$ with usual sum, and let $A$ and $B$ be two alphabets. A shift space of $A^\NZ$ which is uniformly conjugate to an SFT of $B^\NZ$ is itself an SFT.
\end{theo}

\qed

As we will be showed in Lemma \ref{claim:w-sofic_SVL}, without the uniformity assumption on the conjugacy we cannot assure that a shift space conjugated to an SFT is itself an SFT. For shift spaces on the lattice $\N$, such fact can also be deduced from \cite[Prop. 2.3.5]{DarjiGoncalvesSobottka2020} where it was proved that shift space of $A^\N$ has the shadowing property if and only if it is an SFT, together \cite[Prop. 2.1.6]{DarjiGoncalvesSobottka2020} which proves that uniformity of the conjugacy is required in order to ensure that the shadowing property will be preserved.

It remains open the question whether a shift on a general lattice which is uniformly conjugated to an SFT is itself an SFT or not (see Conjecture \ref{conj:SFT-SFT} in the last section). The previous argument of shadowing property cannot be used for general case, since \cite[Prop. 2.3.5]{DarjiGoncalvesSobottka2020} just establishes the relationship between shifts with shadowing property and SFTs when the lattice is $\N$.\\

Now, to end this section let us examine the higher block presentation of shifts of finite type, presenting results with sufficient conditions under which such higher block presentations will be themselves shifts of finite type. To ensure that  higher block presentation of shift of finite is also a shift of finite type, we need some structure on the monoid, on the shift space or on the partition $\bfN$ that, besides make the higher block presentation to be a shift space, also make the set $P_O$ given in \eqref{eq:P_O} be finite or allow us to replace it for some finite set of forbidden patterns. Moreover, such structure shall also able us to translate the set of forbidden patterns that generate the original SFT into a finite set of forbidden patterns over the alphabet $A^{[\bfN]}$.

\begin{prop}\label{prop:higher_block_shift_SFT-2} Let $\Lambda\subset A^\NZ$ be a shift of finite type  and  $ \bfN$ be a partition of $\Lambda$ by cylinders with $\bfM_\bfN$ finite. If at least one of the following conditions holds:
\begin{enumerate}

\item[(C1)] $\Lambda=A^\NZ$ and $\NZ$ can be extended to a group $\mathbb{G}$ with the property that for all $g,h\in\NZ$ we have $g^{-1}h\in\NZ$ or $h^{-1}g\in\NZ$;

\item[(C2)] $\NZ$ can be extended to a countable group $\mathbb{G}$ with the property that for all $g,h\in\NZ$ we have $g^{-1}h\in\NZ$ or $h^{-1}g\in\NZ$;

\item[(C3)] $\NZ$ is a group and $\bigcap_{M\in\bfM_\bfN}M\neq\emptyset$;

\item[(C4)] $1\in\bigcap_{M\in\bfM_\bfN}M$;

\end{enumerate}
then  $\Lambda^{[\bfN]}$ is also a shift of finite type.

Conversely, if condition (C3) or condition (C4) holds, and $\Lambda^{[\bfN]}$ is a shift of finite type, then $\Lambda$ is also a shift of finite type.

\end{prop}

\begin{proof}

Suppose (C1) holds. From Corollary \ref{cor:full-shift_higher_block_presentation}, we have that $(A^\NZ)^{[\bfN]}=X_{P_O}$, where $P_O$ is defined by \eqref{eq:P_O}. We shall show that there exists $F_O\subset\NN_{(A^{[\bfN]})^\NZ}$ with $M_{F_O}$ finite and such that $X_{P_O}=X_{F_O}$. Define
\begin{equation}\label{eq:F_O_1}F_O:=\{ b_1b_h \in \NN_{(A^{[\bfN]})^\NZ}: b_1=(\beta^1_i)_{i\in L}, \ b_h=(\beta^h_i)_{i\in M},\ \exists \ell\in L, m\in M\ s. t.\ \ell=hm\ and\ \beta^1_\ell \neq \beta^h_m\}.\end{equation}
Since $\mathbb{G}$ is a group, for each $\ell,m\in \bigcup_{M\in\bfM_\bfN} M$ there is exactly one $h\in\mathbb{G}$ such that $\ell=hm$. Hence, for each $\ell,m\in \bigcup_{M\in\bfM_\bfN} M$ there are at most one $h\in \NZ$ satisfying $\ell=hm$ (maybe there is none) and thus, since $\bfM_\bfN$ is finite, there are only a finite number of coordinates $h$ in $\NZ$ that can be used in the definition of a pattern in $F_O$, that is, $M_{F_O}$ is finite. Since $F_O\subset P_O$, it follows that $X_{P_O}\subset X_{F_O}$. To check the opposite inclusion let
$\mathbf{b}=(b_j)_{j\in\NZ}=\big((\beta^j_i)_{i\in L^j}\big)_{j\in\NZ}\in X_{F_O}$ and let 
$g,h\in\NZ$, $\ell\in L^g$ and $m\in L^h$ such that $g\ell=hm$. Without loss of generality we  suppose $g^{-1}h=:\hat h\in\NZ$, and define $\mathbf{d}\in X_{F_O}$ given by  
\begin{equation}\label{eq:d}\mathbf{d}=(d_j)_{j\in\NZ}=\big((\delta^j_i)_{i\in M^j}\big)_{j\in\NZ}:=\s^g(\mathbf{b})=(b_{gj})_{j\in\NZ}=\big((\beta^{gj}_i)_{i\in L^{gj}}\big)_{j\in\NZ}.\end{equation}
Therefore, $\ell\in M^1$, $m\in M^{\hat{h}}$ and $\ell=\hat{h}m$, and it follows that 
$$\beta^g_\ell=_{a}\delta^1_\ell=_{b}\delta^{\hat{h}}_m=_{c}\beta^{g\hat{h}}_m=\beta^h_m,$$
where $=_{a}$ and $=_{c}$ come from the definition of $\mathbf{d}$, $=_{b}$ is due $\mathbf{d}\in X_{F_O}$. Thus,
 $\mathbf{b}\in X_{P_O}$, and we conclude that $(A^\NZ)^{[\bfN]}=X_{P_O}=X_{F_O}$.\\

Now suppose (C2) or (C3) holds. We shall prove that $\Lambda^{[\bfN]}=X_P$ for some $P \subset \NN_{(A^{[\bfN]})^\NZ}$ with $M_P$ finite. Since $\Lambda$ is an SFT there exists $F\subset \NN_{A^\NZ}$ complete with $M_F$ is finite and such that $\Lambda=X_F$. 

Recall that if condition (C2) holds, and since $\Lambda$ is an SFT, then any family of nested nonempty cylinders has nonempty intersection (Lemma \ref{lem:SFTs_finite-tersect}). Therefore, from Proposition \ref{prop:higher_block_shift-NZ_countable} it follows that
$\Lambda^{[\bfN]}=X_{P_O\cup D_F}$. On the other hand,  if condition (C3) holds, then from 
Proposition \ref{prop:higher_block_shift} we have $\Lambda^{[\bfN]}=X_{P_O\cup P_F}$. Furthermore, from \eqref{eq:P_F} and \eqref{eq:D_F}, it follows that $M_{P_F}$ as well as $M_{D_F}$ will be finite provided that $M_F$ is finite.
On the other hand, since $\NZ$ is a group (under (C3)) or can be extended to a group  with the property that for all $g,h\in\NZ$ we have $g^{-1}h\in\NZ$ or $h^{-1}g\in\NZ$ (under (C2)), we have $X_{P_O}=X_{F_O}$ where $F_O$ is given by \eqref{eq:F_O_1} and has $M_{F_O}$ finite. Therefore, taking $P=F_O\cup D_F$ or $P=F_O\cup P_F$ (depending on if (C2) or (C3)) we have $M_P$ finite and $X_P=\Lambda^{[\bfN]}$.\\

Finally, suppose that (C4) holds, and let $\Lambda=X_F$ with $M_F$ finite. As in the case of (C3), we have from Proposition \ref{prop:higher_block_shift} and from \eqref{eq:P_F} that $\Lambda^{[\bfN]}=X_{P_O\cup P_F}$ where $M_{P_F}$ is finite due to $M_F$ to be finite. Thus, we just need to find $\hat F_O\subset\NN_{(A^{[\bfN]})^\NZ}$ with $M_{\hat F_O}$ finite and such that $X_{P_O}=X_{\hat F_O}$. Define 
\begin{equation}\label{eq:F_O_2}\hat F_O:=\{ b_1b_h \in \NN_{(A^{[\bfN]})^\NZ}: b_1=(\beta^1_i)_{i\in L},\ h\in L,\ \ b_h=(\beta^h_i)_{i\in M},\ \beta^1_h\neq\beta^h_1\}.\end{equation}
Observe that $M_{\hat F_O}=\bigcup_{M\in\bfM_\bfN} M$, and then, since $\bfM_\bfN$ is finite, 
it follows that $M_{\hat F_O}$ is finite. Let us check that $X_{\hat F_O}=X_{P_O}$. It is direct that $X_{P_O}\subset X_{\hat F_O}$, and to check the opposite inclusion let $\mathbf{b}=(b_j)_{j\in\NZ}=\big((\beta^j_i)_{i\in L^j}\big)_{j\in\NZ}\in X_{\hat F_O}$. Suppose $g,h\in\NZ$, $\ell\in L^g$ and $m\in L^h$ such that $g\ell=hm$. Let $\mathbf{d}:=\s^g(\mathbf{b})$ given by \eqref{eq:d},
and let
 \begin{equation}\label{eq:e}\mathbf{e}=(e_j)_{j\in\NZ}=\big((\varepsilon^j_i)_{i\in N^j}\big)_{j\in\NZ}:=\s^h(\mathbf{b})=(b_{hj})_{j\in\NZ}=\big((\beta^{hj}_i)_{i\in L^{hj}}\big)_{j\in\NZ}.\end{equation}
Note that $M^1=L^g$ and $N^1=L^h$, and so $\ell\in M^1$ and $m\in N^1$. Hence $$\beta^g_\ell=_{a}\delta^1_\ell=_{b}\delta^\ell_1=_{c}\beta^{g\ell}_1=_{d}\beta^{hm}_1=_{e}\varepsilon^m_1=_{f}\varepsilon^1_m=_{g}\beta^h_m,$$
where $=_{a}$ and $=_{c}$ come from the definition of $\mathbf{d}$, $=_{b}$ is due $\mathbf{d}\in X_{\hat F_O}$, $=_{d}$ is due $g\ell=hm$, $=_{e}$ and $=_{g}$ come from the definition of $\mathbf{e}$, and $=_{f}$  is due $\mathbf{e}\in X_{\hat F_O}$. Hence we have proved that $\mathbf{b}$ satisfies the restrictions imposed by $P_O$ and so $\mathbf{b}\in X_{P_O}$.\\

Conversely, suppose (C3) or (C4) holds, and $\Lambda^{[\bfN]}$ is an SFT. 
Since $\bfM_\bfN$ is finite, we can define $\bar \bfN$ as the partition of $A^\NZ$ by all the cylinders defined on the coordinates $\bar M:=\bigcup_{M\in\bfM_\bfN}M$. In order to prove that $\Lambda$ is an SFT, we shall construct $\Theta:A^\NZ\to B^\NZ$, an SBC with local rule defined on the set of coordinates $\bar M$, where $B^\NZ\supset (A^{[\bfN]})^\NZ$, and such that  $\Lambda=\Theta^{-1}(\Lambda^{[\bfN]})$, and then to apply Corollary \ref{cor:inverse_image_of_SFT}.

First, take $\wp\notin A^{[\bfN]}$ and define $B:= A^{[\bfN]}\cup\{\wp\}$.
So, given $\x=(x_i)_{i\in\NZ}$ and $g\in\NZ$ we will define $\big(\Theta(\x)\big)_g$ as follows:
 If there exists exactly one set $M^{g}\in\bfM_\bfN$ such that $[(x_{gi})_{i\in M^{g}}]_{\Lambda}\in\bfN$, then we define 
$\big(\Theta(\x)\big)_g:=(x_{gi})_{i\in M^{g}}$; otherwise, define  $\big(\Theta(\x)\big)_g:=\wp$.

It is direct from its definition, that to compute  $\big(\Theta(\x)\big)_g$ we need to know the entries $(x_\ell)_{\ell\in g\bar M}$, thus $\Theta$ is an SBC whose local rule is defined on the coordinates in $\bar M$. Furthermore, observe that, since $\bfN$ is a partition of $\Lambda$, if for some $g\in\NZ$ there is not a cylinder of $\bfN$ contained in $[(x_{gi})_{i\in \bar M}]_{A^\NZ}$, then $\x$ does not belongs to $\Lambda$. For the same reason, if there are more then one cylinder of $\bfN$ contained in $[(x_{gi})_{i\in \bar M}]_{A^\NZ}$ for some $g\in\NZ$, then $\x\notin \Lambda$ as well.  In both cases, we have $\big(\Theta(\x)\big)_g=\wp$ which implies that $\Theta(\x)\in B^\NZ\setminus (A^{[\bfN]})^\NZ$. On the other hand, if for all $g\in\NZ$ there exists exactly one cylinder of $\bfN$ contained in  $[(x_{gi})_{i\in \bar M}]_{A^\NZ}$, then $\Theta(\x)\in (A^{[\bfN]})^\NZ$. Hence, taking the shift space $X:=\Theta^{-1}\Big( (A^{[\bfN]})^\NZ\Big)$, it follows that $\Lambda\subset X$ and $\Theta$ restrict to $X$ has the same local rule than $\Phi^{[\bfN]}$, that is $\Theta$ is the $\bfN$-higher block code on $X$. Therefore, under either (C3) or (C4) we have that $\Theta$ restrict to $X$ is one-to-one, and then we conclude that $\Lambda=\Theta^{-1}(\Lambda^{[\bfN]})$.

\end{proof}

We notice that that under condition (C3) or (C4) we have $\Lambda$ and $\Lambda^{[\bfN]}$ being uniformly conjugated (Proposition \ref{prop:higher_block_shift}). Furthermore,  when (C3) holds we can construct other set $\tilde F_O$ to replace $P_O$, which gives a better characterization of the shift $X_{P_O}=X_{\tilde F_O}$ (see Equation \eqref{eq:F_O_3}).

Recall that, as we have discussed in the paragraph after \eqref{eq:D_F}, we cannot assure that $X_P\subset \Lambda^{[\bfN]}$ in general (where $P\subset\NN_{(A^{[\bfN]})^\NZ}$ is a set of forbidden patterns which is defined from the forbidden patterns of $\Lambda$). Observe that conditions (B1)-(B5) and (C1)-(C4) assumed in the results concerning on higher block presentations of shift spaces along this article (sections \ref{sec:SFTs}, \ref{sec:SBCs_and_GSBCs}, \ref{sec:sofic} and \ref{sec:FDSs_and_SVLs}) aim to ensure that inclusion. Conversely, if $\Lambda^{[\bfN]}=X_P$ for some forbidden set of patterns $P$ with $M_P$ finite, it is not direct that $\Lambda$ is an SFT (or even whether it is a shift space or not). In fact, although $\Lambda^{[\bfN]}=\Phi^{[\bfN]}(\Lambda)$, it is possible that $\Phi^{[\bfN]}$ is the restriction on $\Lambda$ of a map $\Theta:X\to \Lambda ^{[\bfN]}$ where $X$ is not an SFT. To contour this problem, in the last part of the proof of Proposition \ref{prop:higher_block_shift_SFT-2} we assume that (C3) or (C4) holds and then we get that $\Phi^{[\bfN]}(X\setminus\Lambda)\cap\Lambda^{[\bfN]}=\emptyset$. This leads us to a more general condition under which $\Lambda^{[\bfN]}$ being an SFT implies that $\Lambda$ is an SFT:

\begin{prop}\label{prop:higher_block_shift_SFT-1} Let $\Lambda\subset A^\NZ$ be a shift space and $ \bfN$ be a partition of $\Lambda$ by cylinders with $\bfM_\bfN$ finite, and such that $\Phi^{[\bfN]}$ can be extended for all $A^\NZ$ as a sliding block code $\Theta$ such that $\Theta(A^\NZ\setminus\Lambda)\cap\Lambda^{[\bfN]}=\emptyset$. If 
 $\Lambda^{[\bfN]}$ is a shift of finite type, then  $\Lambda$ is also a shift of finite type.
\end{prop}

\begin{proof}
If $\Lambda^{[\bfN]}$  is an SFT, then there exists $\Xi:(A^{[\bfN]})^\NZ\to\{0,1\}^\NZ$, an SBC such that  $\Lambda^{[\bfN]}=\Xi^{-1}(\mathbf{1})$. Therefore $\Psi:=\Xi\circ \Theta$ is an SBC from $A^\NZ$ to $\{0,1\}^\NZ$, such that $\Lambda=\Psi^{-1}(\mathbf{1})$, and so $\Lambda$ is an SFT.

\end{proof}

For shift spaces on the classical lattices, Proposition \ref{prop:higher_block_shift_SFT-2} gives that:

\begin{cor}\label{cor:higher_block_shift_SFT_on_Z^d-N} Let $\NZ$ be the lattice $\N$ or $\Z^d$ with the usual sum. Let $\Lambda\subset A^\NZ$ be a shift space and $\bfN$ be a partition of $\Lambda$ by cylinders with $\bfM_\bfN$ finite. If
$\Lambda$ is a shift of finite type, then $\Lambda^{[\bfN]}$ is also a shift of finite type. Conversely, if (C3) or (C4) holds and $\Lambda^{[\bfN]}$ is a shift of finite type, then 
$\Lambda$ is also a shift of finite type.

\end{cor}

\begin{proof}
Note that if $\NZ$ is $\N$ or $\Z^d$ with the usual sum, then condition (C2) holds, and from Proposition \ref{prop:higher_block_shift_SFT-2} it follows that $\Lambda$ being an SFT implies that $\Lambda^{[\bfN]}$ is a shift of finite type. If (C3) or (C4) holds, then the converse also comes directly from  Proposition \ref{prop:higher_block_shift_SFT-2}.

\end{proof}

We remark that if $\bfM_\bfN$ is not finite, then it is possible that $\Lambda^{[\bfN]}$ is not an SFT though $\Lambda$ to be one (this is discussed in Section \ref{sec:FDSs_and_SVLs}, and it is related to Problem \ref{prob:higher_block_shift} in Section \ref{sec:final_discussion}).

\section{Sofic shifts}\label{sec:sofic}

 In \cite{LindMarcus} sofic shifts over finite alphabets and on the lattice $\N$ or $\Z$ with the usual sum are defined as the shifts that can be generated from some finite directed labeled graph. 
However, when the alphabet is infinite that is not a good definition for sofic shifts, as showed in the next theorem (see Section \ref{sec:graphs} for formal definitions of directed labeled graphs and their relationships with shift spaces on the lattices $\N$ and $\Z$).

\begin{theo}\label{theo:motivation} Any shift space $\Lambda\subset A^\N$, where the lattice $\N$ is equipped with the usual sum, can be generated by a directed labeled graph.
\end{theo}

\begin{proof}
This proof uses the same construction given in Proposition 3.2.9 in \cite{LindMarcus}. Let $\Lambda\subset A^\N$ and let $W(\Lambda):=\left\{w\in \bigcup_{n\geq 1}A^n:\ \text{$w$ appears in some sequence of $\Lambda$}\right\}$. For any $w\in W(\Lambda)$ define $V(w):=\{u\in W(\Lambda):\ wu\in W(\Lambda)\}$. Let $\epsilon$ denote the empty word (the identity element of the free group of $A$) and define $V(\epsilon):=W(\Lambda)$. Hence it is direct to check that  the labeled graph $\mathcal{G}$, whose set of vertices is $\mathcal{V}:=\{V(w):\ w\in W(\Lambda)\cup\{\epsilon\}\}$ and such that there is an edge labeled with the symbol $a\in A$ from $V(v)$ to $V(w)$ if and only if $V(va)=V(w)$, generates $\Lambda$. 

\end{proof}

Yet the proof of the previous theorem does not work directly for shift spaces on the lattice $\Z$ because we cannot assure that its follower set graph contains an essential subgraph that generates the shift space, it is possible to find examples of non-sofic shift spaces on $\Z$ that are generate from infinite graphs (see Example \ref{ex:free_context}.)\\

A key to understand why the definition of sofic shifts used in the finite-alphabet case does not work for infinite-alphabet shifts, lays on the fact that sofic shifts (in the framework of finite alphabet) correspond to the smallest class of shift spaces which contains SFTs and is invariant for factors (consequence of \cite[Theorem 3.2.1.]{LindMarcus}). However, while factors in the finite-alphabet case are always locally bounded finite-to-one sliding block codes, in the infinite-alphabet case we have the possibility of having factor maps that are not locally bounded finite-to-one sliding block codes and even factor maps that are not sliding block codes. Taking this observation into account, we propose here a definition for sofic shift spaces which recovers for the general case (for any alphabet and for any monoid) the result that a sofic shift is the smallest class of shift spaces which contains SFTs and is invariant under uniform locally bounded finite-to-one factors (Proposition \ref{prop:factor_of_sofic}).

\begin{defn}\label{defn:sofic} A shift space $\Lambda\subset A^\NZ$ is said to be a {\bf sofic shift} if and only if there exist a shift of finite type $\Gamma\subset B^\NZ$ and an onto locally bounded finite-to-one sliding block code $\Psi:\Gamma\to\Lambda$. In particular, if the map $\Psi:\Gamma\to\Lambda$ has order $k$ (that is, its local rule is at most $k$-to-1), then we will say that $\Lambda$ is a sofic with {\bf order $\mathbf{k}$}.
If the sliding block code $\Psi$ is just onto and locally finite-to-one, we say that $\Lambda$ is a {\bf weakly sofic shift}.
\end{defn}

Note that the class of weakly sofic shifts contains the class of sofic shifts, which contains the class of SFTs. The next proposition shows that both the classes of weakly sofic and sofic shifts are invariant through the correct class of sliding block codes.

\begin{prop}\label{prop:factor_of_sofic} A shift space which is factor of a (weakly) sofic shift through a locally bounded finite-to-one sliding block code is itself a (weakly) sofic shift.
A shift space which is factor of a weakly sofic shift through a locally finite-to-one sliding block code is itself a weakly sofic shift.

\end{prop}

\begin{proof}

We shall just prove the first statement, since the second one follows the same reasoning.

Let $\Lambda$ be a weakly sofic shift, and $\Omega$ be a uniform factor of $\Lambda$ through the locally bounded finite-to-one sliding block code $\Phi:\Lambda\to\Omega$. If $\Lambda$ is a weakly sofic shift, then $\Lambda$ is a uniform factor of a shift of finite type $\Gamma$ through a locally finite-to-one sliding block code $\Psi:\Gamma\to\Lambda$. Thus, the map $\Phi\circ\Psi:\Gamma\to\Omega$ is a locally finite-to-one uniform factor map, which means that $\Omega$ is a weakly sofic shift. If $\Lambda$ is a sofic shift, then  $\Psi$ can be taken locally bounded finite-to-one, and then $\Phi\circ\Psi$ will also be a locally bounded finite-to-one uniform factor map, which implies that $\Omega$ is also a sofic shift.
\end{proof}

The above proposition is a general version of Corollary 3.2.2 in \cite{LindMarcus}. We enforce that in its statement we assumed that $\Phi(\Lambda)$ is a shift space, and then we proved that it is a (weakly) sofic shift  (recall from Section \ref{sec:SBCs_and_GSBCs} that in general $\Phi(\Lambda)$ is not a shift space). In the particular case of $\NZ$ being countable and cancellative, from Corollary \ref{cor:image_of_SFT_by_Phi}, it follows that $\Phi(\Lambda)$ is always a weakly sofic shift space.\\

The next three results concern on the higher block presentations of (weakly) sofic shifts.\\

\begin{prop}\label{prop:Lambda_W_Sofic=>Lambda^M_W_Sofic}  Let $\Lambda\subset A^\NZ$ be a (weakly) sofic shift. Suppose $\bfN$ is a partition of $\Lambda$ by cylinders, with $\bfM_\bfN$ finite, and that at least one of the conditions (B1)-(B5) given in Proposition \ref{prop:higher_block_shift} holds. Then, the $\bfN$-higher block presentation of $\Lambda$ is also a (weakly) sofic shift.
\end{prop}

\begin{proof} 

Any of the conditions given in  Proposition \ref{prop:higher_block_shift} assures that $\Lambda^{[\bfN]}$ is a shift space. Furthermore, from Remark \ref{rmk:higher_block_presentation} we have that $\bfM_\bfN$ finite implies that $\Phi^{[\bfN]}$ is a locally bounded finite-to-one SBC. Hence from Proposition \ref{prop:factor_of_sofic} we conclude that $\Lambda^{[\bfN]}$ is a (weakly) sofic shift.

\end{proof}

It is well known that classical sofic shifts (over finite alphabets and on the lattice $\N$ or $\Z$) can always be written as images of 1-step shifts by 1-block SBCs\footnote{In fact, in \cite{LindMarcus} sofic shifts where defined as shift spaces generated from a labeled graph, which is equivalent to be the image of 1-step shifts by an 1-block SBCs  \cite[Definition 3.1.3]{LindMarcus}, and then it was proved that the class of sofic shifts coincides with the class of shifts that are image of any SFT by any SBC \cite[Theorem 3.2.1]{LindMarcus}.}. The following theorem generalizes this result for any alphabet and any lattice:

\begin{prop}\label{prop:sofic and 1-block} Let $\Omega\subset A^\NZ$ be a weakly sofic shift, where $\Lambda\subset B^\NZ$ is the shift of finite type and $\Phi :\Lambda\to\Omega$ is the onto locally finite-to-one sliding block code such that $\Omega=\Phi(\Lambda)$. Let $\bfN$ be the partition of $\Lambda$ by the cylinders used in the finitely defined sets that define $\Phi$. If any of the conditions (C1)-(C4) of Proposition \ref{prop:higher_block_shift_SFT-2} holds, then there exist  a shift of finite type $\Gamma\subset E^\NZ$ and a locally finite-to-one sliding block code $\Psi:\Gamma\to\Omega$ with local rule $\psi:W_{\{1\}}(\Gamma)\to A$ such that $\Omega=\Psi(\Gamma)$. In particular, if $\Omega$ is a sofic of order $k$, then $\Psi$ is locally bounded finite-to-one of order $k$.
\end{prop}

\begin{proof}

Let $\{C_\omega\}_{\omega\in A}$ be the collection of finitely defined sets used in the definition of $\Phi$. Since $\Phi$ is a locally finite-to-one SBC, then there exist a finite number of coordinates such that each $C_\omega$ is union of finitely many cylinders on those coordinates.
Thus, by defining $\bfN$ as the collection of such cylinders that compose the finitely defined sets  $\{C_\omega\}_{\omega\in A}$, we have that $\bfN$ is a partition of $\Lambda$ with $\bfM_\bfN$ finite. Thus, 
under any of the conditions (C1)-(C4) it follows from Proposition \ref{prop:higher_block_shift_SFT-2} that $\Gamma:=\Lambda^{[\bfN]}$ is an SFT. 

Observe that for any $\mathbf{b}=(b_h)_{h\in\NZ}=\big((\beta^h_i)_{i\in M^h}\big)_{h\in\NZ}\in\Gamma$ and for any $g\in\NZ$ it follows that $\big(\s^g(\mathbf{b})\big)_1=b_g=(\beta^g_i)_{i\in M^g}$, where $[(\beta^g_i)_{i\in M^g}]_\Lambda\in\bfN$. Hence we can define $\Psi:\Gamma\to A^\NZ$ with local rule $\psi:W_{\{1\}}(\Gamma)\to A$ which take $b= (\beta_i)_{i\in M}$ to $\omega\in A$ such that $[ (\beta_i)_{i\in M}]_\Lambda\subset C_\omega$. Thus, for all $\mathbf{b}=\big((\beta^h_i)_{i\in M^h}\big)_{h\in\NZ}\in \Gamma$ we have 

$$\Psi(\mathbf{b})=\Psi\Big(\big((\beta^h_i)_{i\in M^h}\big)_{h\in\NZ}\Big)=(\omega_h)_{h\in\NZ},$$
where $\omega_h$ is such that $[(\beta^h_i)_{i\in M^h}]_\Lambda\subset C_{\omega_h}$.
Since for all $\x\in\Lambda$ and $\mathbf{b}=\Phi^{[\bfN]}(\x)\in\Gamma$ we have $\Phi(\x)=\Psi(\mathbf{b})$, it follows that $\Psi(\Gamma)=\Omega$.\\

In the particular case of $\Omega$ being a sofic shift of order $k$, then each $C_\omega$ will be written as the union of at most $k$ cylinders, which implies that there at most $k$ distinct letters that are take to the same symbol by $\psi:W_{\{1\}}(\Gamma)\to A$.

\end{proof}

Hence,

\begin{cor}\label{cor:higher_block_shift_Sofic_on_Z^d-N} Let $\NZ$ be the lattice $\N$ or the lattice $\Z^d$ with the usual sum. It follows that $\Omega\subset A^\NZ$ is a weakly sofic shift if and only if there exist  a shift of finite type $\Gamma\subset E^\NZ$ and a locally finite-to-one sliding block code $\Psi:\Gamma\to\Omega$ with local rule $\psi:W_{\{0\}}(\Gamma)\to A$ such that $\Omega=\Psi(\Gamma)$. In particular, $\Omega$ is a sofic of order $k$ if and only if $\Psi$ is locally bounded finite-to-one of order $k$.
\end{cor}

\begin{proof}
Suppose $\Omega=\Psi(\Gamma)$ for some shift of finite type $\Gamma$ and for some locally finite-to-one sliding block code $\Psi:\Gamma\to\Omega$ with local rule $\psi:W_{\{0\}}(\Gamma)\to A$. Hence, condition (A2) of Theorem \ref{theo: image_of_shifts} holds, and then $\Omega$ is a shift space. Therefore, by definition, it is a weakly sofic shift. In particular, if $\Psi$ is locally bounded finite-to-one of order $k$, then $\Omega$ is a sofic of order $k$.\\

Conversely, if $\Omega$ is a weakly sofic shift, since $\NZ$ is $\N$ or $\Z^d$ with the usual sum, then condition (C2) of Proposition \ref{prop:higher_block_shift_SFT-2} holds, and therefore from Proposition \ref{prop:sofic and 1-block} we have $\Omega=\Psi(\Gamma)$ for some shift of finite type $\Gamma$ and for some locally finite-to-one sliding block code $\Psi:\Gamma\to\Omega$ with local rule $\psi:W_{\{0\}}(\Gamma)\to A$.
In particular, if $\Omega$ is a sofic of order $k$, then  $\Psi$ is locally bounded finite-to-one of order $k$.

\end{proof}

The following lemma gives a sufficient conditions for the union of sofic shifts on the lattice $\N$ or $\Z$ being a (weakly) sofic shift.

\begin{lem}\label{lem:union_of_Sofics} Let $\NZ$ be $\N$ or $\Z$ with the usual sum, and $I\subseteq\N$. Let $\{A_k\}_{k\in I}$ be a disjoint family of nonempty sets, and for each $k\in I$ let $\Lambda_k\subset A_k^{\NZ}$ be a sofic shift. Let $\Gamma_k$ and $\Psi_k:\Gamma_k\to\Lambda_k$ be the correspondent SFT and locally bounded finite-to-one SBC. Suppose that all SFTs $\Gamma_k$ have the same step. Then $\Lambda:=\bigcup_{k\in I}\Lambda_k$ is a weakly sofic. If additionally we have all $\Lambda_k$ being sofic shifts with the same order $m$, then $\Lambda$ also is a sofic with order $m$.
\end{lem}

\begin{proof}

From Lemma \ref{lem:union_of_SFTs}, we have that $\Gamma:=\bigcup_{k\in I}\Gamma_k$ is an SFT. Furthermore, the map $\Psi:\Gamma\to\Lambda$, given by $\Psi(\x)=\Psi_k(\x)$ if $\x\in\Gamma_k$, is an onto locally finite-to-one SBC. Therefore, $\Lambda$ is a weakly  sofic shift. Furthermore, if each $\Lambda_k$ has order $m$, then each $\Psi_k$ can be taken as a locally bounded finite-to-one SBC with order $m$, and then $\Psi$ will also be a locally bounded finite-to-one SBC with order $m$, and thus $\Lambda$ will have order $m$.

\end{proof}

\section{Graph presentations of weakly sofic shifts on the lattices $\N$ and $\Z$}\label{sec:graphs}

In this section we shall establish results that allow to characterize SFTs, sofic shifts and weakly sofic shifts on the classical lattices  $\N$ and $\Z$ through directed labeled graphs representing them, and thus recovering for infinite-alphabet shifts some classical results on finite-alphabet shift spaces. Recall that due to the algebraic properties of the lattices $\N$ and $\Z$ with the usual sum, the language can be defined without specify the coordinates where some word appears. Thus, for simplicity of notation we will consider along this section $W_m(\Lambda)=\{w_0...w_{m-1}\in A^m:\ \exists (x_i)_{i\in\NZ}\text{ s. t. } x_k...x_{k+m-1}=w_0...w_{m-1}\text{ for some }k\in\NZ\}$.\\

A {\bf directed graph} $G$ is a quadruple $G=(\mathcal{V}, \mathcal{E},s,r)$ where $\mathcal{V}$ and  $\mathcal{E}$ are nonempty sets (the set of vertices and set of edges, respectively), and  $s:\mathcal{E}\to\mathcal{V}$ and $r:\mathcal{E}\to\mathcal{V}$ are maps (the source map and the range map, respectively). We say that an edge $e\in\mathcal{E}$ goes from the vertex $u\in\mathcal{V}$ to the vertex $v\in\mathcal{V}$, whenever $s(e)=u$ and $r(e)=v$.

Given a directed graph $G$, we define the respective {\bf edge shift space} $X_G\subset \mathcal{E}^\NZ$ ($\NZ$ being $\N$ or $\Z$) given by  
$$X_G:=\{(e_i)_{i\in\NZ}\in \mathcal{E}^\NZ:\ s(e_{i+1})=r(e_i)\ \forall i\in\NZ\}.$$\\

Recall that, given a shift space $\Lambda$ on the lattice $\N$ or $\Z$, and $N\geq 1$,  we can define $\Lambda^{[N]}$ which is the $N^{th}$-higher block shift of $\Lambda$ given by \eqref{eq:Nth_higher_block_presentation}. Furthermore, $\Lambda$ and  $\Lambda^{[N]}$ are uniformly conjugate through the map $\Phi^{[N]}:\Lambda\to \Lambda^{[N]}$ given  by 
$\Phi^{[N]}\big((x_i)_{i\in\NZ}\big):=(x_i...x_{i+N-1})_{i\in\NZ}$,
whose inverse is given by 
$\left(\Phi^{[N]}\right)^{-1}\big((x_i...x_{i+N-1})_{i\in\NZ}\big)=(x_i)_{i\in\NZ}$ (Proposition \ref{prop:higher_block_shift} or it can also be proved using the same arguments used in \cite[Example 1.5.10]{LindMarcus}). 
In particular, if $\Lambda$ is an SFT, then  $\Lambda^{[N]}$ is also an SFT (Corollary \ref{cor:higher_block_shift_SFT_on_Z^d-N} or it can also be proved using the same arguments used in \cite[Theorem 2.1.10.]{LindMarcus}).
The following result about SFTs is a general version of Theorem 2.3.2 in \cite{LindMarcus}. Although its proof is exactly the same given in \cite{LindMarcus}, we will reproduce it here since it will be useful to understand further results in this section.

\begin{theo}\label{theo:SFT<->edge shift conjugacy} Let $\NZ$ be the lattice $\N$ or $\Z$ with usual sum. $\Lambda\subset A^\NZ$ is an SFT if and only if there exists $M\in\N$ such that $\Lambda^{[M+1]}$ is an edge shift space.
\end{theo}

\begin{proof} 

Suppose $\Lambda$ is an $m$-step shift (that is, it can be obtained from a forbidden words with length $m+1$). Then for any $M\in\N$ we have the higher block presentation $\Lambda^{[M+1]}$ being an SFT. Thus take $M\geq m$ and
consider the graph $G$ where $\mathcal{V}:= W_{M}(\Lambda)$, $\mathcal{E}:= W_{M+1}(\Lambda)$ and $s$ and $r$ are defined for all $(e_0...e_M)\in \mathcal{E}$ by $s(e_0...e_{M})=(e_0...e_{M-1})$ and $r(e_0...e_M)=(e_1...e_M)$. Hence, using \cite[Proposition 2.2.5]{DarjiGoncalvesSobottka2020} one conclude that $X_G=\Lambda^{[M+1]}$.\\

Conversely, suppose $\Lambda^{[M+1]}$ is an edge shift for some $M\in\N$. Since $\Lambda$ is uniformly conjugate to $\Lambda^{[M+1]}$, from Theorem \ref{theo:uniform_conjugacy-SFT} we conclude that $\Lambda$ is an SFT.

\end{proof}

A {\bf directed labeled graph} is, roughly speaking, obtained by assigning for each edge, of a given directed graph, a label from a set $L$ of labels. More precisely, a directed labeled graph is a quintuple $\mathcal{G}:=(G,\mathcal{L}):=(\mathcal{V},\mathcal{E},s,r,\mathcal{L})$ where $G=(\mathcal{V},\mathcal{E},s,r)$ is a directed graph and $\mathcal{L}:\mathcal{E}\to L$ is the label map. 

We will say that a directed labeled graph $\G$ is {\bf left-resolving} if for each vertex $\mathfrak{v}\in \mathcal{V}$ there are not two distinct incoming edges with the same label (that is, $\mathcal{L}$ is one-to-one on $r^{-1}(\mathfrak{v})$). In the same way, we will say that $\G$ is {\bf right-resolving} if for each vertex $\mathfrak{v}\in \mathcal{V}$ there are not two distinct outgoing edges with the same label (that is, $\mathcal{L}$ is one-to-one on $s^{-1}(\mathfrak{v})$).

Given a directed labeled graph $\mathcal{G}$, we can define the {\bf graph shift} as the shift space of $L^\NZ$ given by
 $$X_\mathcal{G}:=\{\big(\mathcal{L}(e_i)\big)_{i\in\NZ}\in L^\NZ:\ s(e_{i+1})=r(e_i)\ \forall i\in\NZ\}.$$

We will say that a directed labeled graph $\mathcal{G}$ is {\bf a presentation of a shift space} $\Lambda\subset A^\NZ$ (or generates $\Lambda$) if $\Lambda=X_\mathcal{G}$. Note that, if $\Lambda$ is two-sided (that is, if $\NZ=\Z$) and $\mathcal{G}$ is a presentation of it, then $\mathcal{G}$ has a labeled subgraph which is an essential (that is, any of its vertices has incoming and outgoing edges) presentation of $\Lambda$. In spite of the fact that Theorem \ref{theo:motivation} concerns only to shift spaces on the lattice $\N$, it follows that non-sofic shift spaces on the lattice $\Z$ may also have graph presentations.

\begin{ex}\label{ex:free_context} Consider $\Z$, the lattice of the integers with the usual sum, and let $\Lambda\subset\{a,b,c\}^\Z$ be the {\bf context free shift}, which is defined from the forbidden set of words $F:=\{ab^mc^na:\ m\neq n\}$. It is well known that the context free shift is not a sofic shift, and then it has not a graph presentation by a finite directed labeled graph (see Example 3.2.8. and Theorem 3.210 in \cite{LindMarcus}). However, it has a graph presentation by an infinite directed labeled graph.

\begin{figure}[H]
\centering
\includegraphics[width=0.8\linewidth=1.0]{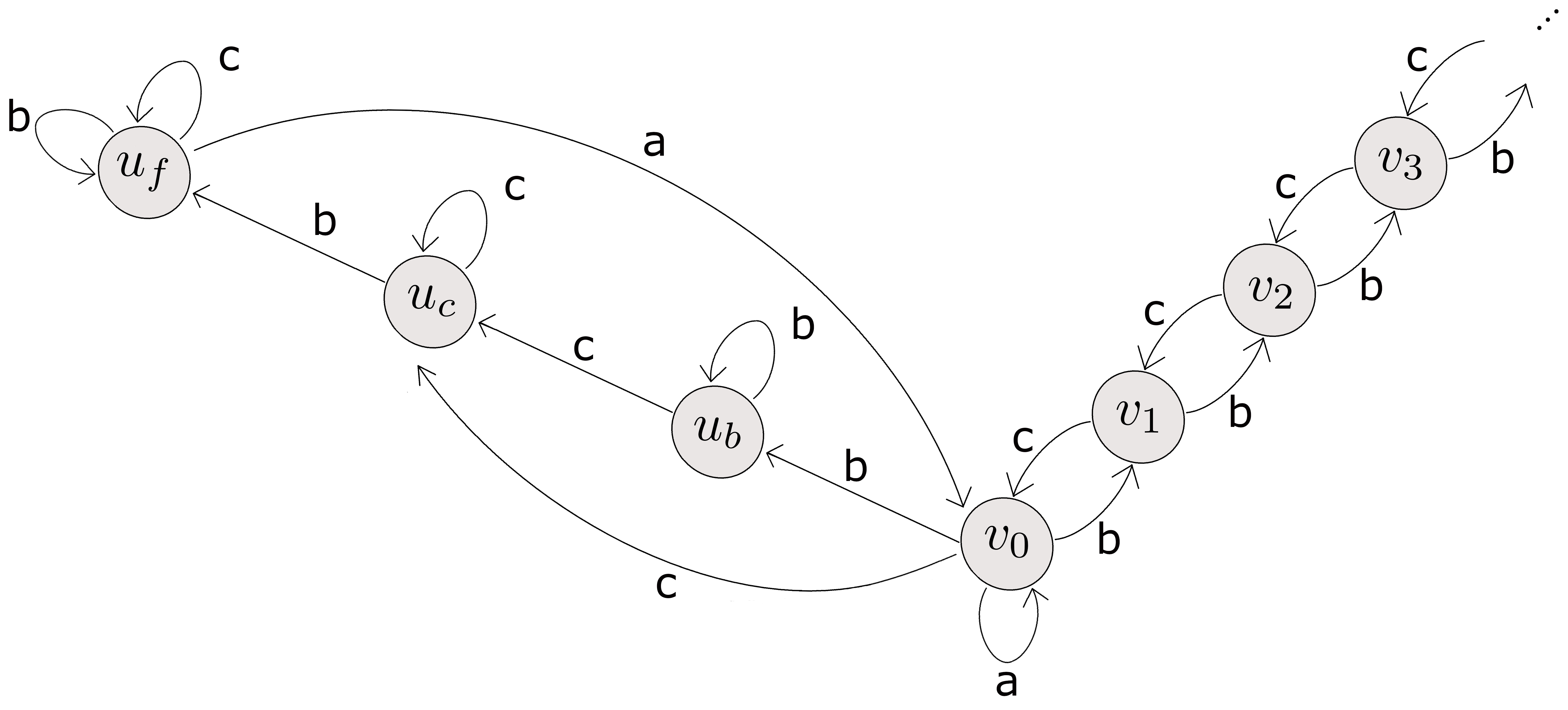}
\caption{A graph presentation for the context free shift.}\label{fig_context-free}
\end{figure}

\end{ex}

Given a shift space $\Lambda\subset A^\NZ$ and a graph presentation  $\mathcal{G}=(\mathcal{V},\mathcal{E},s,r,\mathcal{L})=(G,\mathcal{L})$ of it, we will say that a path $\pi=e_1...e_m$ in $G$ represents $w=w_1...w_m\in W(\Lambda)$ in $\mathcal{G}$ if $\mathcal{L}(e_i)=w_i$ for all $i=1,...,m$ (recall that it is possible that there are several different paths representing a same word $w$).  

Consider the extension of $s$ and $r$ to paths of edges in $\mathcal{G}$, that is, for a path of edges $\pi=e_1...e_m$ let $s(\pi):=s(e_1)$ and $r(\pi):=r(e_m)$. 
Now, for each $w=w_1...w_m\in W(\Lambda)$ define
$$\begin{array}{lrl}I_\mathcal{G}(w)&:=&\{\mathfrak{v}\in\mathcal{V}:\ \mathfrak{v}\text{ is the initial vertex of some path representing } w \text{ in the graph } \mathcal{G}\}\\\\ &=&\{s(\pi_w):\ \pi_w \text{ is any path representing } w \text{ in the graph } \mathcal{G}\}\end{array} $$
and
$$\begin{array}{lrl}T_\mathcal{G}(w)&:=&\{\mathfrak{v}\in\mathcal{V}:\ \mathfrak{v}\text{ is the terminal vertex of some path representing } w \text{ in the graph } \mathcal{G}\}\\\\&:=&\{r(\pi_w):\ \pi_w \text{ is any path representing } w \text{ in the graph } \mathcal{G}\}\end{array} $$\\

Recall that the class of edge shifts is strictly contained in the class of Markov shift spaces (that is, 1-step SFTs), which is strictly contained in the class of SFTs. While Theorem \ref{theo:SFT<->edge shift conjugacy} states that any SFT is uniformly conjugated to an edge shift, the next result gives a way to characterize SFTs through their graph presentations by labeled graphs.

\begin{theo}\label{theo:Markov SFT<=>ultragraph presentation}  Let $\NZ$ be the lattice $\N$ or $\Z$ with usual sum, and let $\Lambda\subset A^\NZ$ be a shift space. The following are equivalent:
\begin{enumerate}

\item $\Lambda$ is an $M$-step;

\item\label{theo:item:Markov I finite} $\Lambda=X_\mathcal{G}$ for some labeled directed graph $\mathcal{G}$ with $|I_\mathcal{G}(w)|=1$ for all $w\in W_m(\Lambda)$ and $m\geq M$;

\item\label{theo:item:Markov T finite} $\Lambda=X_\mathcal{H}$ for some labeled directed graph $\mathcal{H}$ with $|T_\mathcal{H}(w)|=1$ for all $w\in W_m(\Lambda)$ and $m\geq M$.

\end{enumerate}

In particular, when $i.$ holds, then the graphs in $ii.$ and $iii.$ can be taken left-resolving and right-resolving, respectively.

\end{theo}

\begin{proof}\phantom\\

\begin{description}

\item[$i.\Rightarrow ii.$]

Let $\Lambda\subset A^\NZ$ be an SFT. If $\Lambda$ is the full shift (a 0-step shift), then $\Lambda=X_\mathcal{G}$ where $\mathcal{G}$ is a graph with a single vertex. Therefore, any word (even the empty word which has length 0) starts and ends in the unique vertex. Suppose $\Lambda$ is an $M$-step shift with $M\geq 1$ (that is, the forbidden words have length $M+1$). 

Let $G=(\mathcal{V}_G,\mathcal{E}_G,s_G,r_G)$ be the directed labeled graph such that $X_G=\Lambda^{[M+1]}$ which is given by Theorem \ref{theo:SFT<->edge shift conjugacy}. Recall that  $\mathcal{E}_G:=W_{M+1}(\Lambda)$.\\

By defining the label map $\mathcal{L}_\mathcal{G}:\mathcal{E}_G\to A$ given by 
$\mathcal{L}_\mathcal{G}(v_0...v_{M}):=v_0$ we have that $\mathcal{G}=(G,\mathcal{L}_\mathcal{G})$ is such that $\Lambda=X_\mathcal{G}$.\\

To check that $\mathcal{G}$ is left-resolving just observe that given a vertex $\mathfrak{v}=v_1...v_M\in W_M(\Lambda)$ it follows that any of its incoming edges is in the form $e=v_0v_1...v_M\in W_{M+1}(\Lambda)$, which means that it is an outgoing edge from the vertex $\mathfrak{u}=v_0...v_{M-1}$ and it is labeled with  $v_0$. Hence, distinct edges incoming in $\mathfrak{v}$ are labeled with distinct labels.\\

 To prove that $|I_\mathcal{G}(w)|=1$ for all $w\in W_m(\Lambda)$ with $m\geq M$, just observe that any path in $\mathcal{G}$ presenting $w=w_0...w_{m-1}\in W_m(\Lambda)$ shall necessarily starts at the vertex $\mathfrak{w}=w_0...w_{M-1}$, that is, $I_\mathcal{G}(w)=\{\mathfrak{w}\}$.

\item[$i.\Rightarrow iii.$]

If $\s(\Lambda)=\Lambda$ (which always occurs when $\NZ=\Z$), then we can define $\mathcal{H}=(G,\mathcal{L}_\mathcal{H})$ as the directed labeled graph where $G$ is the same as in the proof $i.\Rightarrow ii.$, and $\mathcal{L}_\mathcal{H}$ given by $\mathcal{L}_\mathcal{H}(v_0...v_{M}):=v_{M}$.
On the other hand, if $\s(\Lambda)\subsetneq\Lambda$ (which could occur when $\NZ=\N$), then we define $\mathcal{H}=(H,\mathcal{L}_\mathcal{H})$, where $H=(\mathcal{V}_H,\mathcal{E}_H,s_H,r_H)$ is the directed graph with $\mathcal{V}_H:=W_M(\Lambda)\cup\{(v_0...v_{M-1}):\ \exists 0\leq k\leq M-1 \text{ s. t. } v_i=\epsilon\ \forall i\leq k\text{ and } v_{k+1}...v_{M-1}\in W_{M-k-1}(\Lambda)\}$, $\mathcal{E}_H:=W_{M+1}(\Lambda)\cup\{(v_0...v_{M}):\ \exists 0\leq k\leq M-1 \text{ s. t. } v_i=\epsilon\ \forall i\leq k\text{ and } v_{k+1}...v_{M}\in W_{M-k}(\Lambda)\}$, $s_H(v_0...v_{M})=(v_0...v_{M-1})$ and $r_H(v_0...v_{M})=(v_1...v_{M})$, and the label map $\mathcal{L}_\mathcal{H}(v_0...v_{M}):=v_{M}$ as before. 

In both cases above we have $\Lambda=X_\mathcal{H}$ and, using the same arguments given in the proof $i.\Rightarrow ii.$, we get $\mathcal{H}$ is a right-resolving graph such that $T_\mathcal{H}(w)$ is unitary for all $w\in W_m(\Lambda)$ with $m\geq M$.

\item[$ii.\Rightarrow i.$] 

Now, suppose $\mathcal{G}$ is a directed labeled graph with $|I_\mathcal{G}(w)|=1$  for all $w\in W_m(X_\mathcal{G})$ with $m\geq M$. Let $uw,wv\in W(X_\mathcal{G})$ with $w\in W_M(X_\mathcal{G})$. Since $I_\mathcal{G}(w)$ contains only one vertex, then there is a path in $\mathcal{G}$ presenting $u$ which ends at the unique vertex of $I_\mathcal{G}(w)$, and there is path in $\mathcal{G}$ presenting $wv$ which starts at the unique vertex of $I_\mathcal{G}(w)$. Then there exists in $\mathcal{G}$ a path labeled as $uwv$, and from Proposition 2.2.5. in \cite{DarjiGoncalvesSobottka2020} we conclude that $X_\mathcal{G}$ is an $M$-step shift.

\item[$iii.\Rightarrow i.$] 

For a graph $\mathcal{H}$ with $|T_\mathcal{H}(w)|=1$ for all $w\in W_m(X_\mathcal{H})$ with $m\geq M$, given $uw,wv\in W(X_\mathcal{H})$ with $w\in W_M(X_\mathcal{H})$ we get that $uwv\in W(X_\mathcal{H})$ by using the same reasoning made in the proof of $ii.\Rightarrow i.$ above.

\end{description}

\end{proof}

Note that we could write the statements made in item \ref{theo:item:Markov I finite} and item \ref{theo:item:Markov T finite} of Theorem \ref{theo:Markov SFT<=>ultragraph presentation} simply by saying that $|I_\mathcal{G}(w)|=1$ and  $|T_\mathcal{H}(w)|=1$  for all $w\in W_M(\Lambda)$, since it would imply the properties for all $w\in W_m(\Lambda)$ with $m\geq M$.

We remark that from the previous theorem we have $\Lambda$ being a Markov shift if and only if $\Lambda=X_\mathcal{G}=X_\mathcal{H}$ for some  graphs with $|I_\mathcal{G}(a)|=|T_\mathcal{H}(a)|=1$ for all $a\in W_1(X_\mathcal{G})$. Note that graphs $\mathcal{G}$ and $\mathcal{H}$ with these properties correspond exactly to {\bf ultragraphs}  as defined in \cite{Tomforde2003} (we might say $\mathcal{G}$ is an outgoing ultragraph while $\mathcal{H}$ is an incoming ultragraph).

The following results establish several connections between weakly sofic shift spaces and  labeled graphs with special features. In spite of which occurs in the finite-alphabet framework, where a sofic shift is always presented by a finite labeled graph, in the infinite-alphabet case we have not a complete characterization of (weakly) sofic shifts through labeled graphs.

\begin{theo}\label{theo:wsofic=>graph shift conjugacy} If $\Lambda\subset A^\NZ$ is a weakly sofic shift, then there exists $M\in\N$ such that
\begin{enumerate}

\item\label{theo:wsofic=>graph shift conjugacy_1} $\Lambda^{[M+1]}=X_\mathcal{\hat G}$ for some directed labeled graph $\mathcal{\hat G}$ where each label is used just finitely many times. In the particular case of $\Lambda$ being a sofic shift of order $k$, then each label is used at most $k^{M+1}$ times in $\mathcal{\hat G}$.

\item\label{theo:wsofic=>graph shift conjugacy_2} $\Lambda=X_\mathcal{G}=X_\mathcal{H}$ where $\mathcal{G}$ and $\mathcal{H}$ are directed labeled graphs such that for all $w\in W_m(\Lambda)$ with $m\geq M$ we have $|I_\mathcal{G}(w)|<\infty$ and $|T_\mathcal{H}(w)|<\infty$. In the particular case of $\Lambda$ being a sofic shift of order $k$, then for all $w\in W_m(\Lambda)$ with $m\geq M$ we have $|I_\mathcal{G}(w)|\leq k^m$ and $|T_\mathcal{H}(w)|\leq k^m$.

\end{enumerate}

\end{theo}

\begin{proof} From Corollary \ref{cor:higher_block_shift_Sofic_on_Z^d-N} we have $\Lambda=\Phi(\Gamma)$ for some shift of finite type $\Gamma\subset B^\NZ$ and an onto locally finite-to-one sliding block code $\Phi :\Gamma\to\Lambda$ with local rule $\phi:W_1(\Gamma)\to A$. Suppose $\Gamma$ is an $M$-step shift for some $M\geq 1$ (if $M=0$ then $\Lambda$ is a full shift and the result follows directly).  From Theorem \ref{theo:SFT<->edge shift conjugacy}, we have $\Gamma$ uniformly conjugated to an edge shift $X_G=\Gamma^{[M+1]}$ where $G$ is the graph whose edges are the words of $W_{M+1}(\Gamma)$.
\begin{enumerate}

\item  Let $\mathcal{\hat G}=(G,\mathcal{L})$ be the directed labeled graph with label map $\mathcal{L}:W_{M+1}(\Gamma)\to W_{M+1}(\Lambda)$ given by $\mathcal{L}(w_1...w_{M+1})=\phi(w_1)...\phi(w_{M+1})$. It follows that $X_\mathcal{\hat G}=\Lambda^{[M+1]}$ and, since $\phi$ is finite-to-one, then each label of $W_{M+1}(\Lambda)$ will appear only finitely many times in $\mathcal{\hat G}$.

In particular, if $\Lambda$ is a sofic shift of order $k$, then each $a\in A$ has at most $k$ reverse image by $\phi$, and then each label of $W_{M+1}(\Lambda)$ will appear at most $k^{M+1}$ times in the graph $\mathcal{\hat G}$.

\item We will prove only the existence of the directed labeled graph $\mathcal{G}$, since the proof of the existence of $\mathcal{H}$  is analogous.

 From Theorem \ref{theo:Markov SFT<=>ultragraph presentation} there exists a directed labeled graph $\mathcal{K}=(G,\mathcal{L}_\mathcal{K})$ such that $\Gamma=X_\mathcal{K}$ and for all $m\geq M$ and $u\in W_m(\Gamma)$ we have $|I_\mathcal{K}(u)|=1$. Thus, by defining $\mathcal{G}$ as the directed labeled graph $\mathcal{G}=(G,\phi\circ\mathcal{L}_\mathcal{K})$ we have $\Lambda=X_\mathcal{G}$. Since $\phi$ is finite to one, then for each $w\in W_m(\Lambda)$ there is only a finite number of words in $u^1,..,u^{n(w)}\in W_m(\Gamma)$ whose image by $\phi$ is $w$ (here we are denoting as $\phi$ the local rule of $\Phi$, which is a map from $W_1(\Gamma)$ to $W_1(\Lambda)$ as well as its extension to from  $W_m(\Gamma)$ to $W_m(\Lambda)$). Since for $m\geq M$ we have $|I_\mathcal{K}(u^i)|=1$ for each $i=1,...,n(w)$, it follows that $|I_\mathcal{G}(w)|\leq n(w)$.

Now, suppose that $\Lambda$ is a sofic shift of order $k$, that is, such that each $a\in A$ has at most $k$ reverse image by $\phi$. Thus, each word $w\in W_m(\Lambda)$ has at most $k^m$ reverse images by $\phi$, that is, $n(w)\leq k^m$.

\end{enumerate}

\end{proof}

Note that $\Lambda$, in the above theorem, is uniformly conjugated to $\Lambda^{[M+1]}=X_\mathcal{G}$ through the higher block code $\Phi^{[M+1]}:\Lambda\to X_\mathcal{G}$ given by $\Phi^{[M+1]}\big((x_i)_{i\in\NZ}\big)=(x_i....x_{i+M})_{i\in\NZ}$. However, since in general $(\Phi^{[M+1]})^{-1}$  is not a locally finite-to-one map, we cannot directly conclude the converse of Theorem \ref{theo:wsofic=>graph shift conjugacy}.\ref{theo:wsofic=>graph shift conjugacy_1} (Problem \ref{prob: Lambda^M sofic=>Lambda sofic} in the final section). Furthermore, we remark that we could write the statements made in Theorem \ref{theo:wsofic=>graph shift conjugacy}.\ref{theo:wsofic=>graph shift conjugacy_2} simply as $|I_\mathcal{G}(w)|< \infty$ and  $|T_\mathcal{H}(w)|<\infty$  for all $w\in W_M(\Lambda)$, since it would imply the properties for all $w\in W_m(\Lambda)$ with $m\geq M$.\\

Although a general converse for Theorem \ref{theo:wsofic=>graph shift conjugacy}.\ref{theo:wsofic=>graph shift conjugacy_2} remains still open (Problem \ref{Graph finite many paths=>Lambda sofic} in the final section),  the next result gives a partial converse for it.

\begin{theo}\label{theo:wsofic<=>graph presentation}   Let $\NZ$ be the lattice $\N$ or $\Z$ with usual sum, and let $\Lambda\subset A^\NZ$ be a shift space. The following are equivalent:
\begin{enumerate}

\item\label{theo:item:wsofic} $\Lambda$ is a weakly sofic shift such that $\Lambda=\Phi(\Gamma)$ where $\Gamma$ is  an $m$-step shift and $\Phi$ is a locally finite-to-one $m$-block code  (that is, $\Phi$ has local rule $\phi:W_m(\Gamma)\to W_1(\Lambda)$);

\item\label{theo:item:I finite} $\Lambda=X_\mathcal{G}$ for some labeled directed graph $\mathcal{G}$ with $|I_\mathcal{G}(a)|<\infty$ for all $a\in W_1(\Lambda)$;

\item\label{theo:item:T finite} $\Lambda=X_\mathcal{H}$ for some labeled directed graph $\mathcal{H}$ with $|T_\mathcal{H}(a)|<\infty$ for all $a\in W_1(\Lambda)$.

\end{enumerate}

In particular, the additional hypothesis of $\Lambda$ being a sofic shift of order $k$ is equivalent to $|I_\mathcal{G}(a)|\leq k$ and $|T_\mathcal{H}(a)|\leq k$ for all $a\in A$, while the additional hypothesis  of $\Lambda$ being a Markov shift is equivalent to $|I_\mathcal{G}(a)|=1$ and $|T_\mathcal{H}(a)|=1$ for all $a\in A$.

\end{theo}

\begin{proof} \phantom\\

\begin{description}
\item[(\ref{theo:item:I finite}$\Rightarrow$\ref{theo:item:wsofic}) and 
(\ref{theo:item:T finite}$\Rightarrow$\ref{theo:item:wsofic}):]  We will prove the result only for the case when  $|I_\mathcal{G}(a)|<\infty$ since the proof for the other case is analogous. Consider the directed labeled graph $\mathcal{K}$ with the same vertices and edges of $\mathcal{G}$ but labeled as follows: if $e_{vw}$ is an edge from $v$ to $w$ which is labeled as $a$ in $\mathcal{G}$, then it will be labeled as $a_v$ in $\mathcal{K}$. Thus, for any label $a_v$ we have $I_\mathcal{K}(a_v)=\{v\}$ and from Theorem \ref{theo:Markov SFT<=>ultragraph presentation} we get that $X_\mathcal{K}$ is a Markov shift. Now, consider the 1-block sliding block code $\Theta:X_\mathcal{K}\to X_\mathcal{G}$ whose local rule $\theta:W_1(X_\mathcal{K}) \to W_1(X_\mathcal{G})$ is give by $\theta(a_v)=a$. It is direct that $\Theta(X_\mathcal{K})=X_\mathcal{G}$. Furthermore, since for each label $a$ in $\mathcal{G}$ we have $I_\mathcal{G}(a)=\{v_1,...,v_m\}$ for some $m\in\N$, then there are just a finite number of labels $a_{v_1}$, ..., $a_{v_m}$ that are taken by $\theta$ to the label $a$. Thus, $\Theta$ is a locally finite-to-one SBC and $X_\mathcal{G}$ is a weakly sofic.
In particular, if there exists $k\in\N$ such that $|I_\mathcal{G}(a)|\leq k$ for all label $a$, then $\theta$ defined here, will take at most $k$ distinct labels of $\mathcal{K}$ to each label of $\mathcal{G}$, which means that $X_\mathcal{G}$ is a sofic of order $k$. Furthermore, if $|I_\mathcal{G}(a)|=1$, then from Theorem \ref{theo:Markov SFT<=>ultragraph presentation} we have that $X_\mathcal{G}$ is a Markov shift.\\

\item[(\ref{theo:item:wsofic}$\Rightarrow$\ref{theo:item:I finite}):] If $\Lambda$ is a Markov shift, then from  Theorem \ref{theo:Markov SFT<=>ultragraph presentation}.\ref{theo:item:Markov I finite} we get that $\Lambda=X_\mathcal{G}$ for some labeled graph $\mathcal{G}$ with $|I_\mathcal{G}(a)|=1$ for all $a\in W_1(\Lambda)$. For the general case of $\Lambda$ being a weakly sofic, we will consider the labeled graph $\mathcal{G}=(G,\mathcal{L}_\mathcal{G})$ where $G$ is the directed graph with vertices $\mathcal{V}_G:=W_m(\Gamma)$, edges $\mathcal{E}_G:=W_{m+1}(\Gamma)$; $s_G:\mathcal{E}_G\to \mathcal{V}_G$ and $r_G:\mathcal{E}_G\to \mathcal{V}_G$ given by $s_G(v_0...v_{m})=v_0...v_{m-1}$ and $s_G(v_0...v_{m})=v_1...v_{m}$; and the label map $\mathcal{L}_\mathcal{G}:\mathcal{E}_G\to A$ is given by 
$\mathcal{L}_\mathcal{G}(v_0...v_{m}):=\phi(v_0...v_{m-1})$. Note that, since $\Gamma$ is an $m$-step shift, we have that $\Gamma^{[m+1]}=X_G$ and, therefore, it follows that $X_\mathcal{G}=\Lambda$. 
Now, observe that any edge in $\mathcal{G}$ labeled with a symbol $a$ is in the form $v_0...v_{m}$ with $v_0...v_{m-1}\in\phi^{-1}(a)$, and from the construction of $G$ such edge shall be an outgoing edge of the vertex $v_0...v_{m-1}$. Thus, $I_\mathcal{G}(a)=\phi^{-1}(a)$. Hence, we finish by recalling that $\Phi$ is locally finite-to-one, which means $\phi^{-1}(a)$ is finite.
In particular, if $\Lambda$ is a sofic shift of order $k$, then for each $a$, $I_\mathcal{G}(a)=\phi^{-1}(a)$ has at most $k$ elements.

\item[(\ref{theo:item:wsofic}$\Rightarrow$\ref{theo:item:T finite}):]  If $\Lambda$ is a Markov shift, then from  Theorem \ref{theo:Markov SFT<=>ultragraph presentation}.\ref{theo:item:Markov T finite} we get that $\Lambda=X_\mathcal{H}$ for some labeled graph $\mathcal{H}$ with $|T_\mathcal{H}(a)|=1$ for all $a\in W_1(\Lambda)$. For the general case  we consider the labeled graph $\mathcal{H}=(H,\mathcal{L}_\mathcal{H})$. If  $\s(\Gamma)=\Gamma$, then we define $H$ as the same graph $G$ above. If $\s(\Gamma)\subsetneq\Gamma$,  then we define $H$ with set of vertices $\mathcal{V}_H=W_m(\Gamma)\cup\{\epsilon v_1...v_{m-1}:\ v_1...v_{m-1}\in W_{m-1}(\Gamma)\}$, set of edges $\mathcal{E}_H=W_{m+1}(\Gamma)\cup\{\epsilon v_1...v_{m}:\ v_1...v_{m}\in W_m(\Gamma)\}$, and $s_H=s_G$ and $r_H=r_G$.
 In any case, we define the label map $\mathcal{L}_\mathcal{H}$ given by  $\mathcal{L}_\mathcal{H}(v_0...v_{m}):=\phi(v_1...v_{m})$. Hence, we conclude the proof by using the same reasoning as before to check that $\Lambda=X_\mathcal{H}$ and $T_\mathcal{H}(a)=\phi^{-1}(a)$.

\end{description}
\end{proof}

Recall that if $\Lambda$ is a weakly sofic shift such that $\Lambda=\Phi(\Gamma)$ where $\Gamma$ is an $\ell$-step shift and $\Phi$ is a locally finite-to-one $m$-block code with $\ell<m$, then applying Corollary \ref{cor:higher_block_shift_Sofic_on_Z^d-N} we will find that there exists a 1-step shift $\Omega$ and a locally finite-to-one 1-block code $\Theta:\Omega\to\Lambda$ such that $\Lambda=\Theta(\Omega)$. Thus, Theorem \ref{theo:wsofic<=>graph presentation} still holds if we replace  item \ref{theo:item:wsofic} by the condition of $\Lambda$ being a weakly sofic shift such that $\Lambda=\Phi(\Gamma)$ where $\Gamma$ is  an $\ell$-step shift and $\Phi$ is a locally finite-to-one $m$-block code with $\ell\leq m$.
However, when we have that $\Gamma$ and $\Phi$ can be taken with $\ell$ strictly less than $m$ we can assure that $\Lambda$ has a graph presentation with a strong property:

\begin{theo}\label{theo:wsofic=>graph presentation-2}   Let $\NZ$ be the lattice $\N$ or $\Z$ with usual sum. Suppose  $\Lambda\subset A^\NZ$ is a weakly sofic shift such that $\Lambda=\Phi(\Gamma)$ where $\Gamma$ is an $\ell$-step shift and $\Phi$ is a locally finite-to-one $m$-block code with $\ell<m$. Then $\Lambda=X_\mathcal{G}$ for some labeled directed graph where each label is used just finitely many times. In the particular case of $\Lambda$ being a sofic shift of order $k$, we have each label being used at most $k$ times in $\mathcal{G}$. 
\end{theo}

\begin{proof} 

Let $\phi:W_m(\Gamma)\to A$ be the local rule of $\Phi$. We will proceed as in Theorem \ref{theo:wsofic<=>graph presentation}, but since $\ell$ is strictly less than $m$ we can define $\mathcal{G}$ with $\mathcal{V}_\mathcal{G}:=W_{m-1}(\Gamma)$,  $\mathcal{E}_\mathcal{G}:=W_{m}(\Gamma)$,  $s_\mathcal{G}:\mathcal{E}_\mathcal{G}\to \mathcal{V}_\mathcal{G}$ and $r_\mathcal{G}:\mathcal{E}_\mathcal{G}\to \mathcal{V}_\mathcal{G}$ given by $s_\mathcal{G}(v_0...v_{m-1})=v_0...v_{m-2}$ and $r_\mathcal{G}(v_0...v_{m-1})=v_1...v_{m-1}$, and $\mathcal{L}_\mathcal{G}:\mathcal{E}_\mathcal{G}\to A$ given by $\mathcal{L}_\mathcal{G}(v_0...v_{m-1})=\phi(v_0...v_{m-1})$. Thus $\Lambda=X_\mathcal{G}$ and a label $a\in A$ is used $|\phi^{-1}(a)|$ times in $\mathcal{G}$.

\end{proof}

We recall that Theorem \ref{theo:wsofic<=>graph presentation} and Theorem \ref{theo:wsofic=>graph presentation-2} cover the case of weakly sofic shifts which can be obtained as the image of an $\ell$-step shift through a locally finite-to-one $m$-block code, such that $\ell\leq m$. Under such a condition, those theorems state that we always can find graph presentations for the weakly sofics where the labels are assigned in a particular way. On the other hand, whenever $\ell>m$ we only can use the general result given by Theorem \ref{theo:wsofic=>graph shift conjugacy}.\ref{theo:wsofic=>graph shift conjugacy_2} which state that we always can find graph presentations for the weakly sofics where the paths are assigned in a particular way. 
In fact, there exist weakly sofic shifts that are not image of any $\ell$-step SFT through a locally finite-to-one $m$-block SBC with $\ell\leq m$ (see Example \ref{ex:2-step and graph}). Such weakly sofic shifts can never be presented for labeled graphs with  $I_\mathcal{G}(a)$  being finite for all $a\in A$ or with $T_\mathcal{G}(a)$ being finite for all $a\in A$, but one always can find $M\geq 0$ and a labeled graph $\mathcal{G}$ such that for $m\geq M$ we have $I_\mathcal{G}(w)$ being finite for all $w\in W_m(\Lambda)$ or with $I_\mathcal{G}(w)$  being finite for all $w\in W_m(\Lambda)$. \\

Given a shift space $\Lambda\subset A^\NZ$, its {\bf follower set graph} is the directed labeled graph $\mathcal{G}_\mathcal{F}=(G_\mathcal{F},\mathcal{L}_\mathcal{F})$ constructed as made in Theorem \ref{theo:motivation}:  Let $\mathcal{V}_\mathcal{F}:=\{V(w):\ w\in W(\Lambda)\cup\{\epsilon\}\}$ be the set of vertices, where $\epsilon$ stands for the empty word and $V(w):=\{u\in W(\Lambda):\ wu\in W(\Lambda)\}$; the set of edges is $\mathcal{E}_\mathcal{F}:=\{e_{V(v)V(va)}:\ v\in W(\Lambda)\cup\{\epsilon\},\ a\in A,\text{ and } va\in W(\Lambda)\}$; the source and range maps are defined by $s(e_{V(v)V(va)})=V(v)$ and $r(e_{V(v)V(va)})=V(va)$; finally, the label map is $\mathcal{L}_\mathcal{F}:\mathcal{E}\to A$ given by $\mathcal{L}_\mathcal{F}(e_{V(v)V(va)})=a$.

\begin{ex}\label{ex:follower_set_graph} If $A=\{a,\ b_i,\ c:\ i\in\N^*\}$, and $\Lambda:=X_\mathcal{G}\subset A^\N$ is the sofic shift space defined by the labeled graph $\mathcal{G}$ given in Figure \ref{Fig_ex_Sofic_follower_set_shift} below,

\begin{figure}[H]
\centering
\includegraphics[width=0.4\linewidth=1.0]{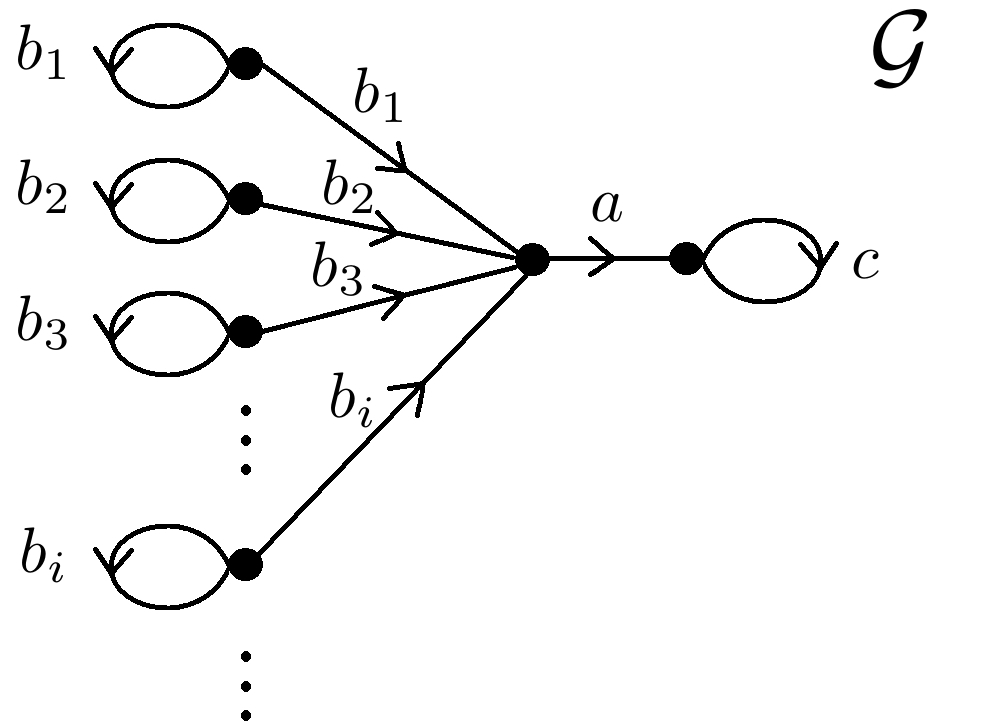}
\caption{The graph $\mathcal{G}$ that generates the Markov shift $\Lambda$ of Example \ref{ex:follower_set_graph}.}\label{Fig_ex_Sofic_follower_set_shift}
\end{figure}

then the follower that the follower set graph of $\Lambda$, is the graph $\mathcal{G}_\mathcal{F}$ given in Figure  \ref{Fig_ex_Sofic_follower_set_shift_2}.

\begin{figure}[H]
\centering
\includegraphics[width=0.75\linewidth=1.0]{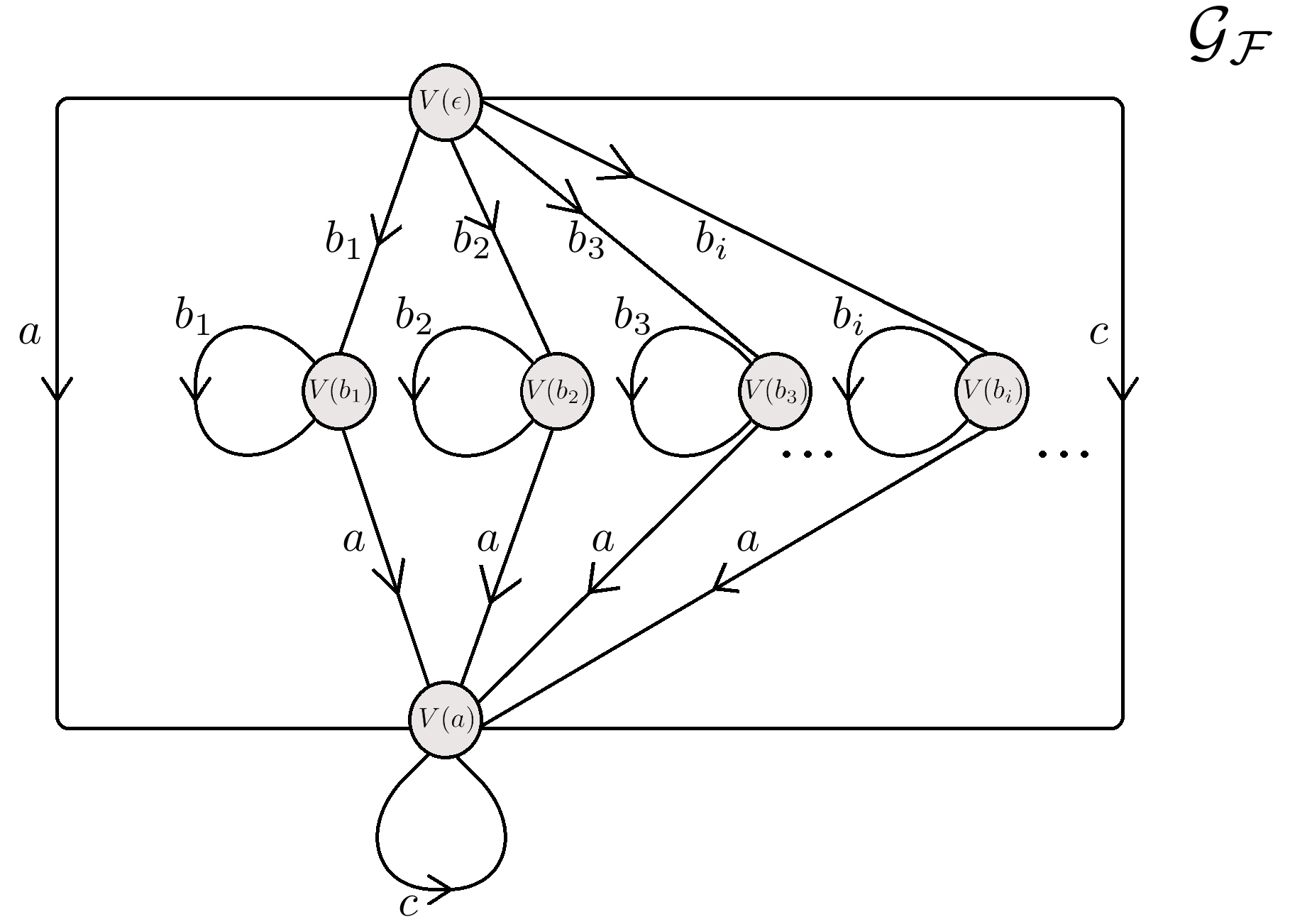}
\caption{The follower set graph $\mathcal{G}_\mathcal{F}$ for the shift $\Lambda$ of Example \ref{ex:follower_set_graph}.}\label{Fig_ex_Sofic_follower_set_shift_2}
\end{figure}

\end{ex}

Directly from the definition of  follower set graphs we get that:

\begin{prop} The follower set graph is right-resolving.
\end{prop}

\qed

The following results are versions of Proposition 3.2.9 and Proposition 3.2.10 in \cite{LindMarcus} for weakly sofic shifts. In particular, they state that the follower set graph of a weakly sofic shift corresponds to a graph with the properties given in Theorem \ref{theo:wsofic=>graph shift conjugacy}.\ref{theo:wsofic=>graph shift conjugacy_2} and Theorem \ref{theo:wsofic<=>graph presentation}.\ref{theo:item:T finite}.

\begin{theo}\label{theo:sofic-follower_set_graph-1} Let $\Lambda\subset A^\NZ$ be a shift space, where $\NZ$ is $\N$ or $\Z$, and let $\mathcal{G}_\mathcal{F}$ be its follower set graph. If $\Lambda$ is a weakly sofic shift, then there exists $M\geq 0$ such that for all $m\geq M$ and $w\in W_m(\Lambda)$ we have $|T_{\mathcal{G}_\mathcal{F}}(w)|<\infty$. If $\Lambda$ is a sofic shift of order $k$, then for
 all $m\geq M$ and $w\in W_m(\Lambda)$ we have $|T_{\mathcal{G}_\mathcal{F}}(w)|\leq 2^{k^m}-1$. Moreover, if $\Lambda$ is an $M$-step SFT, then  for
 all $m\geq M$ and $w\in W_m(\Lambda)$ we have $|T_{\mathcal{G}_\mathcal{F}}(w)|= 1$.
\end{theo}

\begin{proof}
Let $\Lambda\subset A^\N$ be a weakly sofic shift. Let $\mathcal{H}=(\mathcal{V},\mathcal{E},s,r,\mathcal{L})$ be the graph which generate $\Lambda$ given by Theorem \ref{theo:wsofic=>graph shift conjugacy}.\ref{theo:wsofic=>graph shift conjugacy_2} for which there exists $M\geq 0$ such that $T_\mathcal{H}(w)$ is finite for all $w\in W_m(\Lambda)$ with $m\geq M$. Let $\mathcal{G}_\mathcal{F}=(\mathcal{V}_\mathcal{F},\mathcal{E}_\mathcal{F},s_\mathcal{F},r_\mathcal{F},\mathcal{L}_\mathcal{F})$ be the follower set graph of $\Lambda$. Recall that for each $w=w_0...w_{m-1}\in W_m(\Lambda)$ the paths in $\mathcal{G}_\mathcal{F}$ representing $w$ are in the form $\pi_w=e_{V(v)V(vw_0)}e_{V(vw_0)V(vw_0w_1)}e_{V(vw_0w_1)V(vw_0w_1w_2)}...e_{V(vw_0w_1...w_{m-2})V(vw_0w_1...w_{m-1})}$  with $v\in W(\Lambda)\cup\{\epsilon\}$ and $vw\in W(\Lambda)$, and thus 
\begin{equation}\label{eq:T_G_F}T_{\mathcal{G}_\mathcal{F}}(w)=\{V(vw):\ v\in W(\Lambda)\cup\{\epsilon\}\text{ and }vw\in W(\Lambda)\}.\end{equation}

Given a vertex $\mathfrak{v}$ of $\mathcal{H}$, define 
$$F_\mathcal{H}(\mathfrak{v}):=\{u\in W(\Lambda): \text{ exists a path }\pi_u \text{ in }\mathcal{H}\text{ starting at }\mathfrak{v}\text{, which represents }u\}.$$

Observe that for any word $vw\in W(\Lambda)$ we have that
\begin{equation}\label{eq:V(a)}V(vw)=\bigcup_{\mathfrak{v}\in T_\mathcal{H}(vw)}F_\mathcal{H}(\mathfrak{v})
.\end{equation}

Since we have $T_\mathcal{H}(vw)\subset T_\mathcal{H}(w)$, and for $w\in W_m(\Lambda)$ with $m\geq M$ we have $ T_\mathcal{H}(w)$ finite, it follows that by varying $v$ we can find at most $2^{|T_{\mathcal{H}}(w)|}-1$ number of distinct sets $V(vw)$ (at most one for each nonempty subset of $T_{\mathcal{H}}(w)$). Hence, $T_{\mathcal{G}_\mathcal{F}}(w)$ is finite.

If $\Lambda$ is a sofic shift of order $k$, then for all $w\in W_m(\Lambda)$ with $m\geq M$, the set $T_\mathcal{H}(w)$ has at most $k^m$ elements (Theorem \ref{theo:wsofic=>graph shift conjugacy}.\ref{theo:wsofic=>graph shift conjugacy_2}). Therefore, there are at most $2^{k^m}-1$ nonempty subsets of $T_\mathcal{H}(w)$, which, from \eqref{eq:V(a)}, implies at most $2^{k^m}-1$ possible distinct sets in the form $V(vw)$. Moreover, in the particular case of $\Lambda$ being an $M$-step shift we have from Theorem \ref{theo:Markov SFT<=>ultragraph presentation} that $T_\mathcal{H}(w)$ is unitary for all $w\in W_m(\Lambda)$ and, consequently, there is only one set in the form $V(vw)$.

\end{proof}

Note that, if we have a shift $\Lambda$ whose follower set graph $\mathcal{G}_\mathcal{F}$ is such that there exists $M\geq 0$ such that $|T_{\mathcal{G}_\mathcal{F}}(w)|=1$ for all $w\in W_m(\Lambda)$ with $m\geq M$, then from Theorem \ref{theo:Markov SFT<=>ultragraph presentation} we get that $\Lambda$ is an $M$-step shift. However, 
since we have not a general converse for Theorem \ref{theo:wsofic=>graph shift conjugacy}.\ref{theo:wsofic=>graph shift conjugacy_2} we also have not a general converse for Theorem \ref{theo:sofic-follower_set_graph-1}, that is,  in general we cannot conclude that a shift space whose follower set graph $\mathcal{G}_\mathcal{F}$ such that there exists $M\geq 0$ such that $|T_{\mathcal{G}_\mathcal{F}}(w)|<\infty$ for all $w\in W_m(\Lambda)$ with $m\geq M$ is a weakly sofic. However, we can get a complete characterization of weakly sofic shifts in terms of their following set graphs, for shift spaces that satisfy the properties of Theorem \ref{theo:wsofic<=>graph presentation}:

\begin{theo}\label{theo:sofic-follower_set_graph-2} Let $\Lambda\subset A^\NZ$ be a shift space, where $\NZ$ is $\N$ or $\Z$, and let $\mathcal{G}_\mathcal{F}$ be its follower set graph. Then,

\begin{enumerate}
\item $\Lambda$ is a weakly sofic shift such that $\Lambda=\Phi(\Gamma)$ where $\Gamma$ is  an $m$-step shift and $\Phi$ is a locally finite-to-one $m$-block code if and only if  for all $a\in W_1(\Lambda)$ we have $|T_{\mathcal{G}_\mathcal{F}}(a)|<\infty$;

\item If $\Lambda$ is a sofic shift of order $k$ such that $\Lambda=\Phi(\Gamma)$ where $\Gamma$ is  an $m$-step shift and $\Phi$ is a locally bounded finite-to-one $m$-block code of order $k$ then $|T_{\mathcal{G}_\mathcal{F}}(a)|\leq 2^k-1$ for all $a\in W_1(\Lambda)$. Conversely, if $|T_{\mathcal{G}_\mathcal{F}}(a)|\leq k$  for all $a\in W_1(\Lambda)$ then $\Lambda$ is a sofic shift of order $k$ such that $\Lambda=\Phi(\Gamma)$ where $\Gamma$ is  an $m$-step shift and $\Phi$ is a locally bounded finite-to-one $m$-block code of order $k$;

\item $\Lambda$ is a Markov shift if and only if $|T_{\mathcal{G}_\mathcal{F}}(a)|=1$  for all $a\in W_1(\Lambda)$.

 \end{enumerate}
\end{theo}

\begin{proof}

Since $X_{\mathcal{G}_\mathcal{F}}=\Lambda$, from Theorem \ref{theo:wsofic<=>graph presentation} follows that if for all $a\in W_1(\Lambda)$ we have $|T_{\mathcal{G}_\mathcal{F}}(a)|<\infty$, then $\Lambda$ is a weakly sofic, and if there exists $k\in\N$ such that for all $a\in W_1(\Lambda)$ we have $|T_{\mathcal{G}_\mathcal{F}}(a)|\leq k$, then $\Lambda$ is a sofic of order $k$. In any case, also from Theorem \ref{theo:wsofic<=>graph presentation}, $X_{\mathcal{G}_\mathcal{F}}$ is the image of a 1-step shift through a locally finite-to-one 1-block code. On the other hand, from Theorem \ref{theo:Markov SFT<=>ultragraph presentation} we get that if for all $a\in W_1(\Lambda)$ we have $|T_{\mathcal{G}_\mathcal{F}}(a)|=1$, then $\Lambda$ is an SFT.\\

Conversely, if $\Lambda$ is a Markov shift, then from Theorem \ref{theo:sofic-follower_set_graph-1} we have $|T_{\mathcal{G}_\mathcal{F}}(a)|=1$  for all $a\in W_1(\Lambda)$. For the case of $\Lambda$ being a weakly sofic shift (or sofic shift of order $k$) such that $\Lambda=\Phi(\Gamma)$ where $\Gamma$ is  an $m$-step shift and $\Phi$ is a locally finite-to-one $m$-block (of order $k$), let $\mathcal{H}$ be the graph given by Theorem \ref{theo:wsofic<=>graph presentation}.\ref{theo:item:T finite} which generate $\Lambda$ and has $T_\mathcal{H}(a)$ finite (with at most $k$ vertices) for all $a\in W_1(\Lambda)$. Then, by using the same reasoning as in the proof of Theorem \ref{theo:sofic-follower_set_graph-1} with $w=a\in W_1(\Lambda)$ we get that $|T_{\mathcal{G}_\mathcal{F}}(a)|<\infty$ (or $|T_{\mathcal{G}_\mathcal{F}}(a)|\leq 2^k-1$) for all $a\in W_1(\Lambda)$.

\end{proof}

We recall that while Theorem \ref{theo:sofic-follower_set_graph-1} uses Theorem \ref{theo:wsofic=>graph shift conjugacy}.\ref{theo:wsofic=>graph shift conjugacy_2} to prove that the follower set graph of a weakly sofic corresponds to that which the existence is stated in Theorem \ref{theo:wsofic=>graph shift conjugacy}.\ref{theo:wsofic=>graph shift conjugacy_2}, Theorem \ref{theo:sofic-follower_set_graph-2} uses Theorem \ref{theo:wsofic<=>graph presentation}.\ref{theo:item:T finite} to prove that, under the same conditions, the follower set graph of wekaly sofic shifts correspond to that stated in Theorem \ref{theo:wsofic<=>graph presentation}.\ref{theo:item:T finite}. Thus, one could suppose that under the same conditions of Theorem \ref{theo:wsofic=>graph presentation-2} one could use that theorem to prove that the follower set graph of a weakly sofic which is the image of $\ell$-step shift through a locally finite-to-one $m$-block code with $\ell<m$ would use just a finitely many times each label. However, it is not true, and Example \ref{ex:follower_set_graph} provides a counterexample where in spite of a shift having a graph presentation where each label is used only finitely many times, its follower set graph has not the same property.\\

Next, we give an example of weakly sofic shift $\Lambda$ which does accomplish the properties stated in Theorem \ref{theo:sofic-follower_set_graph-2} above, that is, which is not the image of an $m$-step shift through an $\ell$-block code with $\ell\geq m$, and thus whose follower set graph has not $T_{\mathcal{G}_\mathcal{F}}(a)$ finite for all $a\in W_1(\Lambda)$.

\begin{ex}\label{ex:2-step and graph}
Consider $\NZ$ be $\N$ or $\Z$ with the usual sum, and $A=\N$. Let $\Lambda \subset A^\NZ$ be the 2-step SFT given by the set of forbidden words $$F:=\{00\}\cup\{w_0w_1\in A^2:\ w_0,w_1\in \N^*\}\cup\{w_00w_2\in A^3:\ w_0\neq w_2\}.$$ 
Thus $\Lambda$ is composed only by periodic sequences in the form $$\x=(...a0a0a0a0a0...)$$
for $a\neq 0$.

Hence, we have that $\mathcal{G}_\mathcal{F}$, the follower set graph of $\Lambda$, has vertices
$$\begin{array}{l}V(\epsilon)=\{0,1,2,3,...,01,02,03,...,10,20,30,...,010,020,030,...,101,202,303,...\};\\\\
V(0)=\{1,2,3,...,10,20,30,...,101,202,303,...,1010,2020,2030,...\};\\\\
V(...a0a0a0a0...a0)=V(a0)=\{a,a0,a0a,a0a0,...\},\qquad\text{ for each } a\in\N^*;\\\\
V(...0a0a0a0a...0a)=V(0a)=\{0,0a,0a0,0a0a,...\},\qquad\text{ for each } a\in\N^*;
\end{array}$$
and is such that for each $a\in\N^*$ there is an edge labeled with $0$ from $V(0a)$ to $V(a0)$. Then $I_{\mathcal{G}_\mathcal{F}}(0)=\{V(0a):\ a\in \N^*\}$ and $T_{\mathcal{G}_\mathcal{F}}(0)=\{V(a0):\ a\in\N^*\}$.\\

We notice that $|T_{\mathcal{G}_\mathcal{F}}(0)|=\infty$ due to the fact that $\Lambda$ does not accomplish the hypothesis of to be a weakly sofic shift that can be obtained as the image of an $m$-step SFT through a locally finite-to-one SBC whose local rule is $\ell$-block for some $\ell\geq m$. 
Indeed, $\Lambda$ can be thought as the image of itself through the identity map (that is, the image of a 2-step shift through a localy finite-to-one 1-block SBC). In particular, from theorems \ref{theo:wsofic<=>graph presentation} and \ref{theo:sofic-follower_set_graph-2} together, we can conclude that it is neither possible to present $\Lambda$ by a graph $\mathcal{G}$ such that $|I_{\mathcal{G}}(a)|<\infty$ or $|T_{\mathcal{G}}(a)|<\infty$ for all $a\in A$, nor to find out an $\ell$-step shift $\Gamma$ and a locally finite-to-one $m$-block code $\Phi$ with $\ell\leq m$ such that $\Lambda=\Phi(\Gamma)$. However, from Theorem \ref{theo:sofic-follower_set_graph-1}, and since $\Lambda$ is a 2-step shift, we get that for all $m\geq 2$ and $w\in W_m(\Lambda)$ it follows that $|I_{\mathcal{G}_\mathcal{F}}(w)|=|T_{\mathcal{G}_\mathcal{F}}(w)|=1$.  

\end{ex}

\section{Finitely defined shifts and shifts of variable length}\label{sec:FDSs_and_SVLs}

In this section, we shall introduce a new class of shift spaces. Such a class contains the class of SFTs and can be viewed as a generalization of them.

Given a full shift $A^\NZ$ and a set of forbidden words $F\subset\NN_{A^\NZ}$, 
we will define 
\begin{equation}\label{eq:S(F)}\mathcal{S}_\NZ(F):=\bigcup_{(w_i)_{i\in D_w}\in F}\big[(w_i)_{i\in D_w}\big]_{A^\NZ},\end{equation}
where each $[(w_i)_{i\in D_w}]_{A^\NZ}$ is the cylinder defined by a pattern $(w_i)_{i\in D_w}\in F$.

Recall that $\Lambda\subset A^\NZ$ is an SFT if and only if there exists $F\subset\NN_{A^\NZ}$ such that $M_F$ is finite. Note that $M_F$ be finite implies that $\mathcal{S}_\NZ(F)$ is a finitely defined set  (see Remark \ref{rmk:finitely_defined_set}). This observation lead us to consider a more general class of shift spaces that are generate from forbidden sets of words that define a finitely defined set. However, as we shall see below, to consider the class shift spaces $X_F$ for which  $\mathcal{S}_\NZ(F)$ is a finitely defined set could excludes many shift spaces that, in spite of not having this property, have languages that are basically identical to the language of shift spaces that hold this property. 

Consider two monoids $\NZ$ and $\mathbb{G}$ such that $\NZ\subset\mathbb{G}$, $F\subset \NN_{A^\NZ}$ and the shift spaces $\Lambda_\NZ:=X_F\subset A^\NZ$ and $\Lambda_\mathbb{G}:=X_F\subset A^\mathbb{G}$. Both, $\Lambda_\NZ$ and  $\Lambda_\mathbb{G}$ are defined using the same forbidden words, however $\Lambda_\mathbb{G}$ is define on a larger lattice. Hence, it is direct that $\Lambda_\NZ$ is an SFT (or a sofic shift) if and only if $\Lambda_\mathbb{G}$ is an SFT (or a sofic shift). Furthermore, if  
$\mathcal{S}_\NZ(F)$ is a finitely defined set of $A^\NZ$, then $\mathcal{S}_\mathbb{G}(F)$ is a finitely defined set of $A^\mathbb{G}$. However, it is possible that $\Lambda_\mathbb{G}$ could be defined from some set of forbidden words $H\subset\NN_{A^\mathbb{G}}$ such that $\mathcal{S}_\mathbb{G}(H)$ is a finitely defined set of $A^\mathbb{G}$, but  $\Lambda_\NZ$ could not be defined using any set of forbidden words $K\subset\NN_{A^\NZ}$ such that $\mathcal{S}_\NZ(K)$ is a finitely defined set of $A^\NZ$. This fact is captured by example below:

\begin{ex}\label{ex:SVL1} Consider the monoids $\N$ and $\Z$ with the usual sum. Let $A:=\Z$ and define $F\subset\NN_{A^\N}$ given by $F:=\{(x_0x_{k}):\ x_0=0\text{ and } x_{k}=k,\ k\in\N^*\}$. Let $\Lambda_\N=X_F\subset A^\N$ and $\Lambda_\Z=X_F\subset A^\Z$. Hence, sequences belonging to the shift $\Lambda_\N$ or to the shift $\Lambda_\Z$ are exactly those where there is not a symbol `0' appearing $k$ sites before than a symbol `$k$'. 

It is direct that neither $\Lambda_\N$ nor $\Lambda_\Z$ is an SFT (it sufficient to observe that $M_F=\N$ and it is not possible to find another set of forbidden words that generates the same restrictions and uses only a finite number of coordinates).  

However, while $\Lambda_\Z$ can be obtained from an alternative set $H\subset\NN_{A^\Z}$ of forbidden words such that  $\mathcal{S}_\Z(H)$ is a finitely defined set,  $\Lambda_\N$ cannot be obtained from any set of forbidden words with such property.

In fact, taking $H:=\{(x_{-k}x_0)\in \NN_{A^\Z}:\ x_0=k\text{ and } x_{-k}=0,\ k\in\N^*\}$, then $\Lambda_\Z=X_{H}$, and $\mathcal{S}_\Z(H)$ is a finitely defined set. On the other hand, any set of forbidden words $K\subset\NN_{A^\N}$ such that $\Lambda_\N=X_{K}$ shall be such that sequences of $A^\N$ having $x_0=0$ and $x_k=k$ belong to $\mathcal{S}_\N(K)$ while the sequence $\z\in A^\N$ with $z_i=0$ for all $i\in\N$ does not belong to $\mathcal{S}_\N(K)$. But it is not possible to decide that $\z$ does not belong to $\mathcal{S}_\N(K)$ only checking a finite number of coordinates of $\z$, which means that  $\mathcal{S}_\N(K)$ is not a finitely defined set.

\end{ex}

The shifts in the above example have the same languages (modulo translations) and both can be constructed in the same way: Given a past $(x_i)_{i\leq n-1}$ without forbidden words of $F$, to decide whether $x_{n}$ can be a symbol `$k$' or not , one just needs to look whether $x_{n-k}$ is the symbol `0' or not. This lead us to propose the following definition for shift spaces based finitely defined sets:

\begin{defn}\label{defn:FDS} Let $\NZ$ be a monoid, $F\subset\NN_{A^\NZ}$, and $\Lambda:=X_F\subset A^\NZ$. We say that $\Lambda$ is a 
{\bf finitely defined shift (FDS)} if and only if there exist a monoid $\mathbb{G}\supset\NZ$ and $F'\subset \NN_{A^\mathbb{G}}$ such that $\mathcal{S}_\mathbb{G}(F')$ is a finitely defined set of $A^\mathbb{G}$ and $X_F=X_{F'}\subset A^{\mathbb{G}}$. We will say that that a finitely defined shift $\Lambda:=X_F\subset A^\NZ$ is {\bf proper} if $\mathcal{S}_\NZ(F)$ is a finitely defined set of $A^\NZ$.
\end{defn}

Clearly, the class of FDSs includes the class of SFTs. Furthermore, whenever $A$ is finite, due to compactness of $A^\NZ$, we have that $\mathcal{S}_\NZ(F)$ is a finitely defined set if and only if it is the union a finite number of cylinders, and then any FDS is actually an SFT. On the other hand, when $A$ is infinite there are $FDS$ which are not $SFT$ (for instance, the shifts of Example \ref{ex:SVL1}). Thus, the class of FDSs includes a subclass of shift spaces which can only exist in the infinite-alphabet framework:

\begin{defn}\label{defn:SVL} An FDS that is not an SFT is said to be a {\bf shift of variable length (SVL)}.
\end{defn}

An SVL on the lattices $\N$ or $\Z$ with the usual sum can be interpreted as the topological version of variable length Markov chains \cite{BuhlmannWyner1999}. Indeed, while SFTs of order $m$ on the lattices $\N$ or $\Z$ are topological versions of $m^{th}$-order Markov chains, that is, for any given $(x_i)_{i\leq k-1}$ we need to look at most the coordinates $x_{k-m}...x_{k-1}$ to decide what letters could follow, an SVL is such that for each given $(x_i)_{i\leq k-1}$ we start looking along the sequence until find a finite but arbitrary number of coordinates  $x_{k-L}...x_{k-1}$ that allows to decide which letters could follow (the number $L$ always exists but depends on the given sequence). This procedure of looking for forbidden patterns in sequences of an FDS is captured in the next theorem, which gives an alternative definition for FDSs.

\begin{theo}\label{theo:GSBC_and_FDS} A shift space $\Lambda_\NZ:=X_F\subset A^\NZ$ is  a finitely defined shift if and only if 
there exist a monoid $\mathbb{G}\supset\NZ$ and $\Phi:A^\mathbb{G}\to \{0,1\}^\mathbb{G}$, a GSBC, such that $X_F=\Phi^{-1}(\mathbf{1})$.
\end{theo}

\begin{proof}
 Suppose $\Lambda:=X_F\subset A^\NZ$ is an FDS. Let $\mathbb{G}\supset\NZ$ and $\Lambda_\mathbb{G}:=X_F\subset A^\mathbb{G}$ be such that there exist $F'\subset\NN_{A^\mathbb{G}}$ such that $X_{F'}=\Lambda_\mathbb{G}$ and $\mathcal{S}_\mathbb{G}(F')$ a finitely defined set. Then, setting $C_0:=\mathcal{S}_\mathbb{G}(F')$ and $C_1:=A^\mathbb{G}\setminus C_0$, we have that $\Phi:A^\mathbb{G}\to\{0,1\}^\mathbb{G}$ given by $\bigl(\Phi(\x)\bigr)_g=\sum_{b=0}^1b\mathbf{1}_{C_b}\circ\sigma^g(\x)$ is a GSBC and $\Lambda_\mathbb{G}:=X_F=\Phi^{-1}(\mathbf{1})$.
 
 Conversely, suppose there exist $\mathbb{G}\supset\NZ$ and $\Phi:A^\mathbb{G}\to \{0,1\}^\mathbb{G}$, a GSBC, such that $\Lambda_\mathbb{G}:=X_F=\Phi^{-1}(\mathbf{1})$. Then,  the set $C_0$ in the definition of $\Phi$ is a finitely defined set of $A^\mathbb{G}$, which implies that it can be written as a union of cylinders of $A^\mathbb{G}$. Thus, defining $F'\subset\NN_{A^\mathbb{G}}$ as the set of forbidden words that contains exactly the patterns used to define the cylinders that compose $C_0$, we get that $\Lambda_\mathbb{G}=X_{F'}$ and $\mathcal{S}_\mathbb{G}(F')=C_0$ is a finitely defined set.

\end{proof}

\begin{ex}\label{ex:SVL2} Consider the alphabet $A:=\N$ and the monoid $\N$ with the usual sum. Let $\Lambda$ be the shift space where $\x\in\Lambda$ if and only if $x_{i+x_i}\geq x_i$ for $i\in\N$. Note that, in terms of Theorem \ref{theo:GSBC_and_FDS}, $\Lambda=\Phi^{-1}(\mathbf{1})$ where $\Phi:A^\N\to \{0,1\}^\N$
is the generalized sliding block code whose the local rule is defined for all $g\in\N$ and $\x\in A^\N$ by
$\bigl(\Phi(\x)\bigr)_g=\sum_{b=0}^1b\mathbf{1}_{C_b}\circ\sigma^g(\x),$ where  $C_0:=\bigcup_{k\in\N}\bigcup_{j<k}\{\x\in A^\N:\ x_0=k,\ x_{x_0}=j\}$
and
$C_1:=\bigcup_{k\in\N}\bigcup_{j\geq k}\{\x\in A^\N:\ x_0=k,\ x_{x_0}=j\}$.

Hence, $\Lambda$ is an FDS, but since there is not any $F\subset\NN_{A^\N}$ with $M_F$ finite, such that $\Lambda=X_F$, it follows that $\Lambda$ is not an SFT. Therefore we get that $\Lambda$ is an SVL.

\end{ex}

In what follows we shall apply the characterization of FDSs given in Theorem \ref{theo:GSBC_and_FDS} to prove several results.
 
\begin{cor}\label{cor:inverse_image_of_FDS}
If $\Phi:A^\NZ\to B^\NZ$ is a generalized sliding block code and $\Gamma\subset B^\NZ$ is a proper FDS, then $\Lambda:=\Phi^{-1}(\Gamma)$ is also a proper FDS.
\end{cor}

\begin{proof} It follows by using the same arguments used in the proof of  Corollary \ref{cor:inverse_image_of_SFT} but with $\Phi$ being a GSBC.

\end{proof}

\begin{prop}\label{prop:higher_block_shift_FDS-1} Let $\Lambda\subset A^\NZ$ be a shift space and $ \bfN$ be a partition of $\Lambda$ by cylinders such that $\Phi^{[\bfN]}$ can be extended for all $A^\NZ$ as a generalized sliding block code $\Theta$ such that $\Theta(A^\NZ\setminus\Lambda)\cap\Lambda^{[\bfN]}=\emptyset$. If 
 $\Lambda^{[\bfN]}$ is a proper finitely defined shift, then  $\Lambda$ is also a finitely defined shift.
\end{prop}

\begin{proof}Analogous to the proof of Proposition \ref{prop:higher_block_shift_SFT-1}.

\end{proof}

As discussed in the paragraph before Lemma \ref{lem:union_of_SFTs}, it is direct that the intersection of any family of SFTs with step $L$, will be also an SFT with step $L$. For FDSs we have the following result:

\begin{cor}\label{cor:intersct_of_FDSs} Let $\{\Lambda_i\}_{i\in I}$ be a finite family of finitely defined shifts of  $A^\NZ$. Then $\bigcap_{i\in I}\Lambda_i$ is also a finitely defined shift.
\end{cor}

\begin{proof} Let $\mathbb{G}\supset\NZ$ and denote as $\bar\Lambda_i$ the shift space in $A^\mathbb{G}$ which extends  $\Lambda_i$ according to the definition of FDSs. Only note that for each $i\in I$ there exists $\Phi_i:A^\mathbb{G}\to \{0,1\}^\mathbb{G}$, a GSBC, such that $\bar\Lambda_i=\Phi_i^{-1}(\mathbf{1})$. Therefore, $\Theta:A^\mathbb{G}\to \{0,1\}^\mathbb{G}$ defined as $\Theta:=\prod_{i\in I}\Phi_i$ is a GSBC such that $\bigcap_{i\in I}\bar\Lambda_i=\Theta^{-1}(\mathbf{1})$.

\end{proof}

The next results use conditions (C3) and (C4) stated in Proposition \ref{prop:higher_block_shift_SFT-2}.

\begin{prop}\label{prop:higher_block_shift_FDS} Let $\Lambda\subset A^\NZ$ be a shift space and $\bfN$ be a partition of $\Lambda$ by cylinders. Suppose that condition (C3) or condition (C4) holds. If $\Lambda$ is a proper finitely defined shift, then $\Lambda^{[\bfN]}$ is also a proper finitely defined shift. Conversely, if $\bfM_\bfN$ is finite and $\Lambda^{[\bfN]}$ is a proper finitely defined shift, then $\Lambda$ is also a proper finitely defined shift.

\end{prop}

\begin{proof}

Recall that under (C3) we have $\Lambda^{[\bfN]}=X_{P_O\cup P_F}=X_{F_O\cup P_F}$, and under  (C4) we have $\Lambda^{[\bfN]}=X_{P_O\cup P_F}=X_{\hat F_O\cup P_F}$, where $P_O$ is given in \eqref{eq:P_O}, $P_F$ is given in \eqref{eq:P_F}, $F_O$ is given in \eqref{eq:F_O_1}, and $\hat F_O$ is given in \eqref{eq:F_O_2}.\\

Suppose (C3) holds, that is, $\NZ$ is a group and $\bigcap_{M\in\bfM_\bfN}M\neq\emptyset$. Take $n\in \bigcap_{M\in\bfM_\bfN}M$ and define the set of forbidden patterns 
\begin{equation}\label{eq:F_O_3}\tilde F_O:=\{ b_1b_h \in \NN_{(A^{[\bfN]})^\NZ}: b_1=(\beta^1_i)_{i\in L}, h=\ell n^{-1}\ for\  \ell\in L,\ b_h=(\beta^h_i)_{i\in M}, \beta^1_\ell \neq \beta^h_n\}.\end{equation}
Since $\tilde F_O\subset P_O$, it follows that $X_{P_O}\subset X_{\tilde F_O}$. For the opposite inclusion, suppose that $\mathbf{b}=(b_j)_{j\in\NZ}=\big((\beta^j_i)_{i\in L^j}\big)_{j\in\NZ}\in X_{\tilde F_O}$ is such that there exists $g,h\in\NZ$, $\ell\in L^g$ and $m\in L^h$ such that $g\ell=hm$. Define $\tilde h:=g^{-1}hmn^{-1}$, $\mathbf{d}:=\s^g(\mathbf{b})$ and $\mathbf{e}:=\s^h(\mathbf{b})$ as given in \eqref{eq:d} and \eqref{eq:e}, respectively. Therefore
$$\beta^g_\ell=\delta^1_\ell=\delta^{\tilde{h}}_n=\beta^{g\tilde{h}}_n=\beta^{hmn^{-1}}_n=\varepsilon^{mn^{-1}}_n=\varepsilon^1_m=\beta^h_m.$$
Thus,
 $\mathbf{b}\in X_{P_O}$, and we conclude that $X_{\tilde F_O}=X_{P_O}=X_{F_O}$.\\
 
 Observe that $\tilde F_O$ becomes $\hat F_O$ if we take $n=n^{-1}=1$. Thus, under the convention that $n=n^{-1}=1$ if (C4) holds, we can write $X_{P_O}=X_{\tilde F_O}$ whenever (C3) or (C4) holds. Now, under any of these conditions, we define $\Delta:(A^{[\bfN]})^\NZ\to \{0,1\}^\NZ$ given for all $\mathbf{b}=\big((\beta^g_i)_{i\in M^g}\big)_{g\in\NZ}\in (A^{[\bfN]})^\NZ$ and $g\in \NZ$  by 
$$\big(\Delta(\mathbf{b})\big)_g=\left\{\begin{array}{lll} 1 &,\ if& \beta^g_\ell=\beta^{g\ell n^{-1}}_n\ \forall \ell\in L^g\\\\ 0 &,& otherwise.\end{array}\right.$$ 
Hence $\Delta$ is a GSBC such that $X_{\tilde F_O}=\Delta^{-1}(\mathbf{1})$ and then $X_{\tilde F_O}$ is an FDS. 

Now, define the 1-block code $\Psi:(A^{[\bfN]})^\NZ\to A^\NZ$, which takes each $\mathbf{b}=(b_g)_{g\in\NZ}=\big((\beta^g_i)_{i\in M^g}\big)_{g\in \NZ}\in 
(A^{[\bfN]})^\NZ$ to the sequence $(\beta^{gn^{-1}}_n)_{g\in \NZ}\in A^\NZ$. Recall that $\Psi$ restricted to $\Lambda^{[\bfN]}$ coincides with the inverse of $\Phi^{[\bfN]}$. 

Since $\Lambda\subset A^\NZ$ is a proper FDS, there exists $\Xi:A^\NZ\to \{0,1\}^\NZ$, GSBC such that $\Xi^{-1}(\mathbf{1})=\Lambda$.

Hence, we can define the map $\Omega:(A^{[\bfN]})^\NZ\to\{0,1\}^\NZ$ given for all $\mathbf{b}\in (A^{[\bfN]})^\NZ$ and $g\in\NZ$ by
$$\Big(\Omega(\mathbf{b})\Big)_g:=\Big(\Delta(\mathbf{b})\Big)_g\cdot\Big(\Xi\circ \Psi(\mathbf{b})\Big)_g.$$
From its definition, it is direct that $\Omega$ is a GSBC. Furthermore, $\Big(\Delta(\mathbf{b})\Big)_g=1$ for all $g\in\NZ$ if and only if $\mathbf{b}$ belongs to $X_{\tilde F_O}$, while $\Big(\Xi\circ \Psi(\mathbf{b})\Big)_g=1$ for all $g$ if and only if $\Psi(\mathbf{b})$ belongs to $\Lambda$. Thus, we have that $\Big(\Omega(\mathbf{b})\Big)_g=1$ for all $g$ if and only if $\mathbf{b}\in X_{\tilde F_O}\cap \Psi^{-1}(\Lambda)=X_{\tilde F_O}\cap X_{P_F}=X_{\tilde F_O\cup P_F}=\Lambda^{[\bfN]}$. That is, $\Omega^{-1}(\mathbf{1})=\Lambda^{[\bfN]}$, and from Theorem \ref{theo:GSBC_and_FDS} we conclude that 
$\Lambda^{[\bfN]}$ is a proper FDS.\\

To prove the converse under the additional assumption of $\bfM_\bfN$ being finite, we just need to consider the map $\Theta:A^\NZ\to B^\NZ$ constructed in the proof of Proposition \ref{prop:higher_block_shift_SFT-2}, and then, since $\Lambda=\Theta^{-1}(\Lambda^{[\bfN]})$, from Corollary \ref{cor:inverse_image_of_FDS} we conclude that $\Lambda$ is a proper FDS.

\end{proof}

Note that we could prove Proposition \ref{prop:higher_block_shift_SFT-2} under conditions (C3) or (C4), using $\Omega$ constructed in the proof of Proposition \ref{prop:higher_block_shift_FDS} above. In such case, $\Lambda$ is assumed to be an SFT and $\bfM_\bfN$ is assumed finite, and so $\Delta$ and $\Xi$ can be taken as SBCs, which implies $\Omega$ is also an SBC.

We remark that, in general, higher block codes with $\bfM_\bfN$ infinite do not preserve the classes of SFTs and SVLs. In fact, from \eqref{eq:F_O_3}, we get that if $\bfM_\bfN$ infinite and (C3) or (C4) holds, then $(A^\NZ)^{[\bfN]}=X_{\tilde F_O}$ is an SVL though $A^\NZ$ is an SFT. On the other hand,  Lemma \ref{lem:SVLs-Sofic_and_SFT} gives an example with $\bfM_\bfN$ infinite, where $\Lambda$ is an SVL and $\Lambda^{[\bfN]}$ is an SFT.

\section{Relationship between shift spaces}\label{sec:relationships}

In this section we shall examine the relationship between SFTs, sofic shifts, and FDSs. Although it is direct that (weakly) sofic shifts and FDSs contain the class of SFTs, the relationship between those classes (and even the relationship between them and with the class of SFTs) is not so direct.

In what follows, we shall see through examples several facts about how these shift spaces are related one to the other.

\begin{claim}\label{claim:sofic-non fds} There exist sofic shifts that are not FDSs (that is, they are neither SFTs nor SVLs).
\end{claim}

\begin{claim}\label{claim:wsofic-non sofic non fds} There exist weakly sofic shifts that are neither sofic shifts nor FDSs.
\end{claim}

The first claim can be easily checked by recalling that there does not exist SVLs over finite alphabet, and then any sofic shift  over a finite alphabet which is not an SFT cannot be an FDS. The following result gives a general way to construct shift spaces that hold the property stated in Claim \ref{claim:sofic-non fds} and shift spaces that hold the property stated in Claim \ref{claim:wsofic-non sofic non fds}.

\begin{lem}\label{cor:union_of_finite_Sofics_order-m} Let $\NZ$ be $\N$ or $\Z$ with the usual sum, and $I\subseteq \N$. Let $\{A_k\}_{k\in I}$ be a disjoint family of nonempty finite sets, and for each $k\in I$ let $\Lambda_k\subset A_k^{\NZ}$ be a sofic shift which is not an SFT. Let $\Gamma_k$ and $\Psi_k:\Gamma_k\to\Lambda_k$ be the correspondent SFT and locally bounded finite-to-one SBC. Supose that all SFTs $\Gamma_k$ have the same step. Therefore,
\begin{enumerate}
\item If all $\Lambda_k$ are sofic shifts with the same order $m$, then $\Lambda:=\bigcup_{k\in I}\Lambda_k$ is a sofic with order $m$, but not an FDS.

\item If for each $m\in\N^*$ there exists $k_m\in I$ such that $\Lambda_{k_m}$ is not a sofic with order $m$, then $\Lambda:=\bigcup_{k\in I}\Lambda_k$ is a weakly sofic which is neither a sofic nor an FDS.
\end{enumerate}
\end{lem}

\begin{proof}\phantom\\

\begin{enumerate}
\item From Lemma \ref{lem:union_of_Sofics}, we have that $\Lambda$ is a sofic with order $m$. Let $A:=\bigcup_{k\in I}A_k$, and let us check that $\Lambda$ is not an FDS. Indeed, if by contradiction we suppose $\Lambda$ is an FDS, then from Theorem \ref{theo:GSBC_and_FDS} there exists a generalized sliding block $\Phi:A^\NZ\to\{0,1\}^\NZ$ such that $\Lambda=\Phi^{-1}(\mathbf{1})$. Since $\Lambda$ is composed by the disjoint union of the sets $\Lambda_k\subset A_k^\NZ$, it follows that for each $k\in\N^*$ we have $\Phi_{|_{A_k^\NZ}}^{-1}(\mathbf{1})=\Lambda_k$. However, since each $\Lambda_k$ is a shift over a finite alphabet, this would imply that it is an SFT, contradicting the hypothesis that it is not one.
 
\item By using the same reasoning as in item {\em i.} above, we get that $\Lambda$ is a weakly sofic which is not an FDS. Now, suppose by contradiction that $\Lambda$ is a sofic shift, say with order $L$, that is, there exist $\Omega$, an SFT,  and $\Phi:\Omega\to\Lambda$ an onto locally bounded finite-to-one SBC with order $L$. 
Since $\{\Lambda_k\}_{k\in I}$ is a disjoint family of shift spaces, then $\Omega$ shall be a disjoint union of SFTs  $\Omega_k:=\Phi^{-1}(\Lambda_k)$. Thus, for all $k\in I$ we have $\Phi_{|_{\Omega_k}}:\Omega_k\to\Lambda_k$ is an  onto locally bounded finite-to-one SBC with order $L$, which means that $\Lambda_k$ is a sofic of order $L$ for all $k\in I$, contradicting that there exists $k_L\in I$ such $\Lambda_{k_L}$ is not a sofic of order $L$.
\end{enumerate}
\end{proof}

\begin{claim} There exist FDSs that are not weakly sofic shifts. 
\end{claim}

Clearly an FDS which is not a weakly sofic shall be an SVL. To prove the above claim, consider $\Lambda:=X_F\subset A^\Z$ the SVL given in Example \ref{ex:SVL1}, where $A:=\Z$ and $F:=\{(x_0x_{k}):\ x_0=0\text{ and } x_{k}=k,\ k\in\N^*\}$. Let us show that there is not any $M\in\N$ for which $\Lambda^{[M+1]}=X_\mathcal{G}$ for some directed labeled graph $\mathcal{G}$ where each label is used just finitely many times. Thus, from Theorem \ref{theo:wsofic=>graph shift conjugacy}.\ref{theo:wsofic=>graph shift conjugacy_1} we can conclude that $\Lambda$ is not a weakly sofic shift. Indeed, take any $M\in\N$ and let $\mathcal{G}$ any directed labeled graph such that  $\Lambda^{[M+1]}=X_\mathcal{G}$ (which always exists due to Theorem \ref{theo:motivation}). Note that the constant sequence $(...0000000...)$ belongs to $\Lambda$, while the sequence
$$(...111.\underbrace{000000000000000...0}_{0\ appears\ k-1\ times}\ell000...)$$
belongs to $\Lambda$ if and only if $\ell\geq k$. For each $\ell\in\N$ we have $w_\ell:=\underbrace{00000000...0\ell}_{length\ M+1}$ belonging to $ W_1(\Lambda^{[M+1]})$ if and only if either $\ell=0$ or $\ell\geq M+1$. It follows that for all $n\geq 1$ the word $\mathbf{w}^n_\ell:=\underbrace{w_0w_0...w_0w_\ell}_{length\ n}$ lies in $W_{n}(\Lambda^{[M+1]})$ if and only if either $\ell=0$ or $\ell\geq M+n$. Then, we get that for $m\neq n$ the path in $\mathcal{G}$ that represents $\mathbf{w}^n_0$ cannot end at the same vertex than the path that represents $\mathbf{w}^m_0$, which implies that the symbol $w_0$ shall be used infinitely many times as label in $\mathcal{G}$.

\begin{claim}\label{claim:w-sofic_SVL} There exist shifts spaces that are simultaneously weakly sofic shifts and SVLs.
\end{claim}

Recall that if $\NZ$ is $\N$ or $\Z$, we say an SFT $\Lambda\subset A^\NZ$ has step $m$, if there exists a set of forbidden words $F\subset W(A^\NZ)$ such that $M_F\subset\{0,....,m\}$ and $\Lambda=X_F$. Thus, the previous claim can be proved by the following lemma.

\begin{lem}\label{lem:SVLs-Sofic_and_SFT} Let $\NZ$ be the monoid $\N$ or $\Z$ with the usual sum. Let $\{A_k\}_{k\in\N^*}$ be a disjoint family of nonempty finite sets, and for each $k\in\N^*$ let $\Lambda_k\subset A_k^{\NZ}$ be an SFT. Suppose that for each $m\in\N$ there exists $k_m\in\N^*$ such that $\Lambda_{k_m}$ is not an $m$-step shift. Then $\Lambda:=\bigcup_{k\in\N^*}\Lambda_k$ is an SVL and a weakly sofic shift.
\end{lem}

\begin{proof}

Let $A:=\bigcup_{k\in\N^*}A_k$ and $\Lambda:=\bigcup_{k\in\N^*} \Lambda_k\ \subset\ A^\NZ$. First, observe that $\Lambda$ is not an SFT, since if we have $\Lambda=X_F$ for some $F$ such that $M_F$ is finite, then for all $k\in\N^*$ we could take $G_k:=\{x_0\in A:\ x_0\in A\setminus A_k\}\cup F$ and we will have that $\Lambda_k=X_{G_k}$ where $M_{G_k}=M_F\cup\{0\}$, contradicting that  for each $m\in\N$ there exists $k_m\in\N^*$ such that $\Lambda_{k_m}$ is not an $m$-step shift.\\

Now, for each $k\in\N^*$, let $p_k\in\N$ be minimum value such that $\Lambda_k$ is a shift of step $p_k$. Let $F_k\subset W(A_k^\NZ)$ be such that $X_{F_k}=\Lambda_k$ and $M_{F_k}\subset\{0,...,p_k\}$  (from the definition of $p_k$ there is necessarily at least one word of $F_k$ that uses the coordinate $p_k$). Note that $\Lambda=X_F$ with $F:=\{(x_0x_1):\ x_0\in A_i\text{ and } x_1\in A_j \text{ for all }i\neq j\}\cup\bigcup_{k\in\N^*} F_k$. Observe that, $\mathcal{S}_\Z(F)$ is a finitely defined set of $A^\NZ$. In fact, to decide if given $\x\in A^\NZ$ belongs or not to $\mathcal{S}_\Z(F)$, we need to check in which $A_k$ the entry $x_0$ lies, and then to check whether $x_0... x_{p_k}$ belongs or not to $F_k$.\\ 

Now, let us show that $\Lambda$ is a weakly sofic shift. More specifically we will show that there exists a directed labeled graph $\mathcal{G}$ where each symbol of $A$ is used as label only finitely many times and such that $\Lambda=X_\mathcal{G}$.

For each $k\in\N$ let $\beta_k:\Lambda_k\to\Lambda_k^{[p_k+1]}$ be the higher block code and $\Lambda_k^{[p_k+1]}:=\beta_k(\Lambda_k)$ be the respective $(p_k+1)^{th}$ higher block shift of $\Lambda_k$ (see \cite[Section 1.4]{LindMarcus}). From \cite[Theorem 2.3.2]{LindMarcus} it follows that $\Lambda_k^{[p_k+1]}$ is an edge shift, that is, there exists a finite directed graph $G_k$ such that $\Lambda_k^{[p_k+1]}=X_{G_k}$. Furthermore, $\beta_k$ is a uniform conjugacy between $\Lambda_k$ and 
$\Lambda_k^{[p_k+1]}$, such that $\beta_k^{-1}:\Lambda_k^{[p_k+1]}\to\Lambda_k$ is a locally bounded finite-to-one sliding block code whose local rule is 1-block. In particular, setting $\mathcal{L}_k$ as the local rule of $\beta_k^{-1}$, we have that $\Lambda_k=X_{\mathcal{G}_k}$ where $\mathcal{G}_k$ is the finite directed labeled graph $\mathcal{G}_k=(G_k,\mathcal{L}_k)$. 

Now, define $\beta:\Lambda\to \bigcup_{k\in\N^*}\Lambda_k^{[p_k+1]}$, given by $\beta(\x):=\beta_k(\x)$ if $\x\in\Lambda_k$. It follows that $\beta$ is a conjugacy (not uniform) between $\Lambda$ and $\bigcup_{k\in\N^*}\Lambda_k^{[p_k+1]}$, such that $\beta^{-1}$ is a locally finite-to-one sliding block code whose local rule is 1-block. Hence, taking $\mathcal{G}$ as the directed labeled graph which is the union of all $\mathcal{G}_k$ (and then whose label map $\mathcal{L}$ is $\beta^{-1}$), we have that $\Lambda=X_\mathcal{G}$.

\end{proof}

\begin{figure}[h]
\centering
\includegraphics[width=0.6\linewidth=1.0]{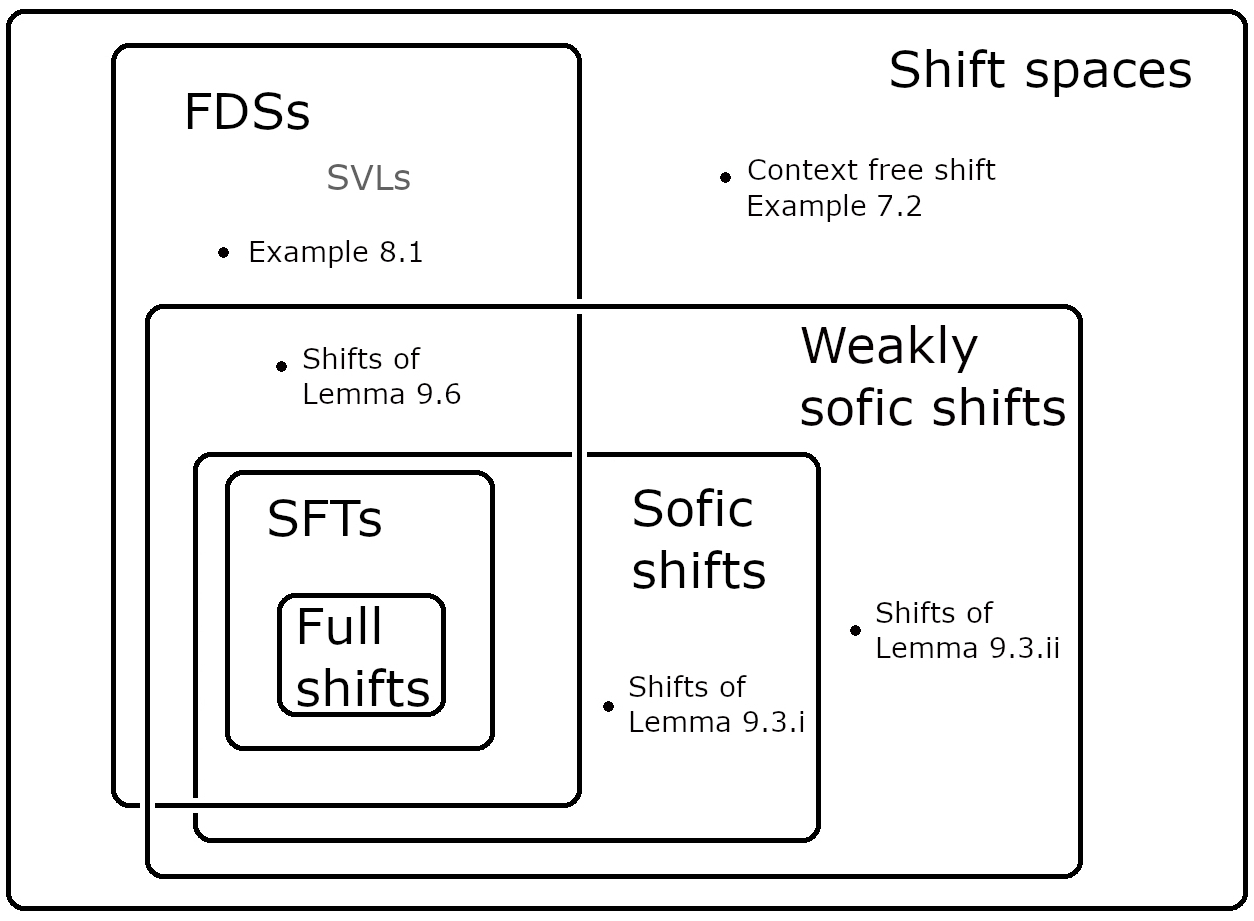}
\caption{Relationship between classes of shift spaces.}\label{Fig_shift_spaces_relationship}
\end{figure}

Other important relationship between shift spaces is given by conjugacies and uniform conjugacies. Theorem \ref{theo:uniform_conjugacy-SFT} proved that the class of SFTs is invariant under uniform conjugacies whenever the monoid is $\N$ or $\Z$. On the other hand, the class of weakly sofic shifts (and the class of sofic shifts) was proved to be always invariant under locally (bounded) finite-to-one uniform factors (Proposition \ref{prop:factor_of_sofic}).  However, the arguments used in Theorem \ref{theo:uniform_conjugacy-SFT} cannot be extended for shift spaces over other monoids. Furthermore, in the context of finite-alphabet shifts, any conjugacy is actually a locally bounded finite-to-one uniform conjugacy. However, in the general context, there are several different types of conjugacies, and thus it is possible for two shift spaces of distinct classes to be conjugated. For instance, besides to be simultaneously an SVL and a weakly sofic shift, the shift $\Lambda$ obtained in Lemma \ref{lem:SVLs-Sofic_and_SFT}  is conjugated to the SFT $\bigcup_{k\in\N^*}\Lambda_k^{[p_k+1]}$ (note that the map $\beta$, which conjugate them, is a nonuniform GSBC, while its inverse is a locally (non-bounded) finite-to-one SBC). Thus, Lemma \ref{lem:SVLs-Sofic_and_SFT} provides a proof for the following claims:

\begin{claim}\label{claim:w-sofic_SFT_conjug} A weakly sofic shift which is not an SFT can be conjugate to an SFT (through a locally finite-to-one GSBC).
\end{claim}

\begin{claim}\label{claim:SFT_SVL_conjug} An SFT can be conjugated to an SVL (through a locally finite-to-one GSBC).
\end{claim}

\section{Open problems and conjectures}\label{sec:final_discussion}

In this final section we shall present some open problems and conjectures regarding the classification scheme presented in this article.\\

The first problem we will state concern to Theorem \ref{theo: image_of_shifts}:

\begin{prob}\label{prob:image_GSBC} To find a set of sufficient and necessary conditions under which a generalized sliding block code $\Phi:\Lambda\subset A^\NZ\to B^\NZ$ is: 
\begin{enumerate}
\item such that $\Phi(\Lambda)$ is a shift space;
\item a closed map.  
\end{enumerate}
\end{prob}

For the specific case of higher block codes and their images we have the following problems related to find generalizations of Proposition \ref{prop:higher_block_shift} and of several results about higher block presentations of FDSs and weakly sofic shift given in sections \ref{sec:SFTs}, \ref{sec:sofic}, and \ref{sec:FDSs_and_SVLs}.

\begin{prob}\label{prob:higher_block_shift} To find sufficient and necessary conditions for:
\begin{enumerate}
\item  the higher block presentation of a shift space to be a shift space too;

\item the higher block shift of an FDS (or specifically an SFT or an SVL) to be an FDS (or specifically an SFT or an SVL) too;

\item the higher block shift of a (weakly) sofic to be a (weakly) sofic too.

\end{enumerate}

\end{prob}

Still about higher block presentations, it would be useful to find a converse for Proposition \ref{prop:Lambda_W_Sofic=>Lambda^M_W_Sofic}, that is, to answer the following question:

\begin{prob}\label{prob: Lambda^M sofic=>Lambda sofic} If $\Lambda^{[\bfN]}$ is a weakly sofic shift, then is $\Lambda$ a weakly sofic shift? (If it is not true in general, what are the sufficient and necessary conditions for?)
\end{prob}

Note that the answer to the last problem above, if stated for the lattices $\N$ and $\Z$ with usual sum, would provide a converse of Theorem \ref{theo:wsofic=>graph shift conjugacy}.\ref{theo:wsofic=>graph shift conjugacy_1} about graph presentation of weakly sofic shifts.\\

Other open question about graph presentations is stated in the next problem, which concerns to the converse of Theorem \ref{theo:wsofic=>graph shift conjugacy}.\ref{theo:wsofic=>graph shift conjugacy_2}:

\begin{prob}\label{Graph finite many paths=>Lambda sofic} Let $\NZ$ be the lattice $\N$ or $\Z$ with usual sum and $\Lambda\subset A^\NZ$ be a shift space. If there exist a graph $\mathcal{G}$ and  $M\in\N$ such that $\Lambda=X_\mathcal{G}$ and $|T_\mathcal{G}(w)|<\infty$ for all $w\in W_M(\Lambda)$ (or if $|I_\mathcal{G}(w)|<\infty$ for all $w\in W_M(\Lambda)$), then is $\Lambda$ a weakly sofic shift?
\end{prob}

Note that an affirmative answer to Problem \ref{Graph finite many paths=>Lambda sofic} would also allow to find a converse to Theorem \ref{theo:sofic-follower_set_graph-1}, and so we will have a complete characterization of weakly sofic shifts in terms of the follower set graphs.

The next problem, motivated by Claim \ref{claim:w-sofic_SVL}, concerns to the relationship between sofic shifts and FDSs:

\begin{prob}\label{prob:FDSs-sofic} Is there exist some shift space which is simultaneously SVL and sofic shift?
\end{prob}

It is direct that the class of SFTs is contained in the intersection of the classes of sofics and FDSs. However, it is not clear if the intersections of those classes is greater than the class of SFTs or not. Furthermore, motivated by claims \ref{claim:w-sofic_SFT_conjug} and \ref{claim:SFT_SVL_conjug}, we could consider the following questions about shift space classes and conjugacies:

\begin{prob}\label{prob:conjugation_between_classes}\phantom{}
\begin{enumerate}
\item Can a sofic shift which is not an SFT be conjugate to an SFT?
\item\label{sofic-FDS} Can a (weakly) sofic shift which is not an FDS be conjugate to an FDS? 
\item Can an SVL which is not a weakly sofic be conjugate to an SFT?
\end{enumerate}
\end{prob}

On the other hand, inspired by the definition of sofic shifts, we could also ask for other classes of shift spaces that could be defined in some analogous way, but using other type of factors:

\begin{prob}\label{prob:SFTs-GSBC} How is the class of shift spaces that are factors of FDSs through locally finite-to-one SBCs or through locally bounded finite-to-one SBCs?
\end{prob}

Note that the two classes in Problem \ref{prob:SFTs-GSBC} contain the class of sofic shifts, and each of them could correspond to some kind of generalization of sofic shifts. Other generalizations of sofic shifts could also be obtained considering the following problem: 

\begin{prob}\label{prob:characterize_image_GSBC} How is the class of shift spaces that are factors of SFTs (or FDS) through locally (bounded) finite-to-one GSBCs?
\end{prob}

It would also be useful to find good definitions  in the general context for other known shift spaces, as well as their relationships with the classes defined in this work. For instance, we could consider the following question:

\begin{prob}\label{prob:Dyck-Motzkin_shifts} How could we define Dyck shifts (see \cite{BealBlockeletDima2013,BealBlockeletDima2014}) and Motzkin shift (see \cite{Hamachi_et_Al2009,Inoue2006}) in a general context?
\end{prob}

Finally, it would be interesting, and a good test for the definition given here for SFTs on general lattices, to establish a relationship between SFTs and shift spaces with the shadowing property as it was made in \cite[Prop. 2.3.5]{DarjiGoncalvesSobottka2020} for shift spaces on the lattice $\N$.

\begin{prob}\label{prob:shadowing} Does the class of shift spaces with the shadowing property coincide to the class of SFTs?
\end{prob}

Now, to finish, we state two conjectures. Both conjectures are inspired by the alternative definitions given for SFTs and FDSs by theorems \ref{theo:SBC_and_SFT} and \ref{theo:GSBC_and_FDS}, respectively.

\begin{conj}\label{conj:SFT-SFT}
Any shift which is a uniformly conjugate to an SFT is itself an SFT.
\end{conj}

Note that the conjecture above holds for the particular case of shift spaces on the monoid $\N$ or $\Z$ (Theorem \ref{theo:uniform_conjugacy-SFT}), however the machinery used to prove it for that particular case cannot be used in the general context.
Note that if the Problem \ref{prob:shadowing} has a positive answer, then from \cite[Prop. 2.1.6]{DarjiGoncalvesSobottka2020} we get that Conjecture \ref{conj:SFT-SFT} is true.

The second conjecture, below, gives a negative answer to Problem \ref{prob:conjugation_between_classes}.\ref{sofic-FDS}.

\begin{conj}
A shift space which is conjugated to an FDS is itself an FDS.
\end{conj}


\section*{Acknowledgments}

\noindent This work was supported by the CNPq-Brasil grant 301445/2018-4, and developed while the author was a visiting professor at CAPES-Brasil at the Pacific Institute for the Mathematical Sciences, University of British Columbia. The author thanks Professor Brian Marcus and his research group for their hospitality. The author also thanks Charleen Stroud for her hospitality and for her help in materializing a topological example.

\newpage

\appendix

\section{Appendix: Local rule's radius of an SBC}\label{sec:local rule radius}

The next theorems give some features of SBCs, recovering, under some conditions, the description of an SBC as a map $\Phi$ where $\left(\Phi(\x)\right)_g$ is function of a window around $x_g$ whose width is the same for all $g\in\NZ$. Such fixed width of the window is given in terms of the metric of the monoid instead in terms of the quantity of sites inside the window (although both will coincide if we have a conservative metric monoid as $\N$ and $\Z$ with the usual sum).

\begin{defn}\label{defn:GSBC_radius} Let $\NZ$ be a metric monoid and $\Phi:\Lambda\subset A^\NZ\to B^\NZ$ be a GSBC. For each $b\in B$ let $C_b$ be the set which define $\Phi$ according to Equation \eqref{eq:LR_block_code}, and let $M_b\subset \NZ$ be the set of  coordinates needed to define $C_b$ which satisfies: i. $C_b$ can be written as a union of cylinders that use only coordinates of $M_b$; ii. if $C_b$ can be written as a union of cylinders that use only coordinates of a set $N_b\subset \NZ$, then  $\sup_{ i\in M_b}{d}(i,1)\leq \sup_{ j\in N_b}{d}(j,1)$.
We define the $\Phi$ {\bf local rule's radius} as 
$$r(\Phi):=\sup_{\tiny\begin{array}{c}b\in B\\ i\in M_b\end{array}}{d}(i,1).$$
\end{defn}

\begin{theo}\label{theo:SBC=>preserves_distance} Let $\NZ$ be a metric monoid. If $\Phi:\Lambda\subset A^\NZ\to B^\NZ$ is an SBC, then $r(\Phi)<\infty$.
\end{theo}

\begin{proof} 
If $\Phi$ is an SBC then there exists a finite $M\subset\NZ$ such that for all $b\in B$, we have $M_b\subset M$. Thus, we have $r(\Phi)\leq\sup_{i\in M}{d}(i,1)<\infty$.

\end{proof}

We notice that the converse of Theorem \ref{theo:SBC=>preserves_distance} is not valid in general, that is, $\Phi$ being a GSBC with $r(\Phi)<\infty$ does not imply that $\Phi$ is an SBC. In fact, if we consider $d$ as the $0-1$ metric on a left-cancellative monoid $\NZ$ (that is, for all $a,b,c\in\NZ$ it follows that $ab=ac$ implies $b=c$), then any generalized sliding block code $\Phi$ will have $r(\Phi)\leq1$, even if it is not an SBC. Example \ref{Ex:finite_radius} below shows other case where the converse of Theorem \ref{theo:SBC=>preserves_distance} fails.

\begin{ex}\label{Ex:finite_radius} Let $\mathbb{Q}^*_+$ be the set of all positive rational number with the usual product. On $\mathbb{Q}^*_+$ consider the metric $d$ such that $d(g,h):=|\ln(g)-\ln(h)|$ for all $g,h\in \mathbb{Q}^*_+$, which is invariant under translations.

Let $A:=\N^*$ and define $\Phi:A^{ \mathbb{Q}^*_+}\to A^{ \mathbb{Q}^*_+}$ as the GSBC defined on the sets $$C_b:=\bigcup_{k\in A}\left\{\x\in A^{ \mathbb{Q}^*_+}:\ x_1=k\text{ and } x_{(k+1)/k}=b\right\},\quad \forall b\in A.$$
Note that any $b\in A$ we have $M_b=\{1\}\cup\{(k+1)/k:\ k\ge 1\}$ and therefore $\Phi$ cannot be an SBC. However, $$r(\Phi)=\sup_{\tiny\begin{array}{c}b\in A\\ i\in M_b\end{array}}{d}(i,1)=d(2,1)=\ln(2).$$

\end{ex}

The next theorem gives a converse for Theorem \ref{theo:SBC=>preserves_distance} for some particular class of monoids $\NZ$:

\begin{theo}\label{theo:preserving_distance=>SBC}  Let $\NZ$ be a metric monoid and suppose that its metric $d$ is such that there exists $R\ge 0$ such that the closed ball $B_d[1,R]:=\{h\in\NZ:\ d(h,1)\le R\}$ is a finite set. If $\Phi:\Lambda\subset A^\NZ\to B^\NZ$ is a GSBC with $r(\Phi)\le R$, then $\Phi$ is an SBC.
\end{theo}

\begin{proof} 
We just need to check that there exists a finite $M\subset\NZ$ such that $M_b\subset M$ for all $b\in B$. In fact, if
 $\Phi$ a GSBC such that $r(\Phi)\le R$ for some $R\ge 0$ with $\#B_d[1,R]<\infty$, then $M:=B_d[1,r(\Phi)]$ is finite and for all $b\in B$ and $g\in M_b$ it follows that $d(g,1)\le r(\Phi)\leq R$ which means that $g\in M$.

\end{proof}

\section{Appendix: Shifts with metric memory}\label{sec:shift_d-memory}

In this section we propose an alternative definition for the concept of memory for shift spaces, which is based on the metric of the monoid.

\begin{defn}\label{defn:shift_memory_m} Let $\NZ$ be a metric monoid for some metric $d$. Given $m\geq 0$, a shift space $\Lambda\subset A^\NZ$ will be said to be a {\bf shift with $d$-memory $m$} (or {\bf shift with metric memory $m$})  if and only if $\Lambda=X_F$ for some $F\subset\NN_{A^\NZ}$ such that $M_F\subset B_d[1,m]$.
\end{defn}

Note that since the metric on $\NZ$ is invariant by left products, $\Lambda=X_F$ to be a shift with $d$-memory $m$ could be interpreted as that one could check if $(x_g)_{g\in\NZ}$ belongs to $\Lambda$ or not, by checking for each $g$ if there is a forbidden pattern in $(x_h)_{h\in B_d[g,m]}$. This idea recovers the classical notion of SFTs where one moves a fixed-length window along the configuration $(x_g)_{g\in\NZ}$ checking if there is some forbidden pattern in that window.
In fact, since any SFT can be defined from some forbidden set of patterns $F$ with $M_F$ finite, it follows that 

\begin{theo}\label{theo:SFT->d-memory} If $\NZ$ is a metric monoid and $\Lambda\subset A^\NZ$ is an SFT, then $\Lambda$ has metric memory $m$ for some $m\in[0,\infty)$.
\end{theo}

\qed

However, as we can see in the examples below, to be a shift with metric memory is not a sufficient condition for being an SFT, neither it is a necessary condition to be an FDS. 

\begin{ex}\label{Ex:shift-with-memory_trivial} If $\NZ$ is left cancellative (that is, whenever $fg=fh$ it follows that $f=g$) and $d$ is the $0-1$ metric on $\NZ$, then any shift space $\Lambda\subseteq A^\NZ$ is a shift with $d$-memory $1$.
\end{ex}

\begin{ex}\label{Ex:shift-with-memory_non-FDS_1} Let $\mathbb{Q}^*_+$ with the usual product be the metric monoid with the metric $d$ given by $d(g,h):=|\ln(g)-\ln(h)|$ for all $g,h\in \mathbb{Q}^*_+$ (see Example \ref{Ex:finite_radius}). Let $A:=\{0,1\}$ and define $F\subset \NN_{A^{\mathbb{Q}^*_+}}$ as follows:
$$F:=\left\{(x_1x_{\frac{1}{2^k}}):\ x_1=x_{\frac{1}{2^k}}=1,\ k\geq 1\right\}.$$
Although any forbidden word is written using the index $1$ and exactly one more index in the form $1/2^k$ with $k\in\N^*$, the shift $X_F$ is not an SFT since $M_F=\{1/2^k: k\in\N\}$, and it is not possible to find any finite set of forbidden patterns that generate the same shift. In fact, $X_F$ is not even an FDS. In fact, any map $\Phi:A^{\mathbb{Q}^*_+}\to\{0,1\}^{\mathbb{Q}^*_+}$ such that $\Phi^{-1}(\mathbf{1})=X_F$ should necessarily be such that for the configuration $\x=(x_i)_{i\in \mathbb{Q}^*_+}$ with $x_i=1$ for all $i\in \mathbb{Q}^*_+$ one could not decide the outcome of $\big(\Phi(\x)\big)_1$ without knowing $\x_{M_F}$. 
On the other hand, $M_F\subset B_d(1,|\ln(\frac{1}{2})|)$ which means that $X_F$ has $d$-memory $m= |\ln(\frac{1}{2})|$.
\end{ex}


\begin{ex}\label{Ex:non-shift-with-memory_non-FDS} Let $\N^*$ be the metric monoid given in Example \ref{ex:shift_spaces_SFT} with the respective metric $d$. Let $A:=\{0,1\}$ and define the $F\subset \NN_{A^{\N^*}}$ as follows:
$$F:=\left\{(x_1x_{2^k}):\ x_1=x_{2^k}=1,\ k\geq 1\right\}.$$
Note that $X_F$ is not an FDS since $M_F=\{2^k:\ k\in\N\}$. Furthermore, $X_F$ is not a shift with $d$-memory, since $M_F$ is not contained in any closed $d$-ball.

\end{ex}

\begin{ex}\label{Ex:shift-with-memory_FDS} Consider the metric monoid $\mathbb{Q}^+$ with the usual sum and with the Euclidean metric $d$. Let $A=\Z$ and $F:=\{(x_0x_{\frac{1}{k}}):\ x_0=0\text{ and } x_{\frac{1}{k}}=k,\ k\in\N^*\}$. Then, $X_F\subset A^{\mathbb{Q}^+}$ has  $d$-memory 1, and by using the same strategy of Example \ref{ex:SVL1}, we get that it is is an FDS. 
\end{ex}

The next theorem gives a sufficient condition for the converse of Theorem \ref{theo:SFT->d-memory}. Its proof is direct.

\begin{theo}\label{theo:d-memory->SFT} Let $\NZ$ be a metric monoid with metric $d$, and is such that there exists $R\ge 0$ such that the closed ball $B_d[1,R]$ is finite. If $\Lambda\subset A^\NZ$ is a shift with $d$-memory $m$ for some $m\leq R$, then $\Lambda$ is an SFT.
\end{theo}

\qed

\newpage


\end{document}